\documentclass{amsart}
\usepackage{bm}
\usepackage{microtype}

\usepackage{etex}
\usepackage{fixltx2e}
\usepackage{tikz}
\usepackage{tikz-cd}
\usepackage{adjustbox}
\usepackage{enumitem}
\usepackage{mathtools}

\renewcommand{\mathbb}{\mathbf}
\newcommand{\chibar}{\overline{\chi}}

\newcommand{\sigmacomp}{\sigma^{\operatorname{co}}}
\newcommand{\sigmasigmacomp}{\sigma|\sigmacomp}

\newcommand{\kbase}{k}
\newcommand{\lbase}{l}

\newcommand{\Pone}{\mathbb{P}^1}

\newcommand{\Fptimes}{\F_p^\times}

\newcommand{\Qptimes}{\Qp^\times}

\newcommand{\Iw}{\mathrm{Iw}}
\newcommand{\tld}[1]{\widetilde{#1}}
\newcommand{\soc}{\operatorname{soc}}

\usepackage{amssymb,epsfig,amscd}
\usepackage[all,cmtip,poly]{xy}
\usepackage{color}

\newcommand{\lbar}[1]{\overline{#1}}

\newcommand{\fourmatrix}[4]{\begin{pmatrix} #1 & #2 \\ #3 & #4 \end{pmatrix}}
\newcommand{\isom}{\xrightarrow{\sim}}

\newcommand{\triv}{\mathrm{triv}}
\newcommand{\fm}{\mathfrak{m}}

\newcommand{\fa}{\mathfrak{a}}

\newcommand{\quoteslim}[1]{%
  \text{``}%
  \mathchoice
    {\varprojlim_{\mathclap{#1}}\mkern-3mu} %
    {\varprojlim_{#1}\mkern-10mu}            %
    {\varprojlim_{#1}\mkern-5mu}            %
    {\varprojlim_{#1}\mkern-5mu}            %
  \text{''}%
}

\newcommand{\colim}{\operatorname{colim}}

\newcommand{\fB}{\mathfrak{B}}

\newcommand{\fT}{\mathfrak{T}}
\newcommand{\fQ}{\mathfrak{Q}}

\newcommand{\To}{\longrightarrow}
\newcommand{\isoto}{\stackrel{\sim}{\To}}

\newcommand{\fg}{\operatorname{fg}}
\newcommand{\fp}{\operatorname{fp}}

\newcommand{\fl}{\operatorname{fl}}
\newcommand{\finl}{\operatorname{f.l.}}

\newcommand{\JH}{\operatorname{JH}}
    
\def\iso{\buildrel \sim \over \longrightarrow}
\def\inverseiso{\buildrel \sim \over \longleftarrow}

\newcommand{\id}{\operatorname{id}}

\usepackage{ifpdf}
\ifpdf \usepackage[bookmarksopen,bookmarksdepth=4]{hyperref} \fi

\usepackage{chngcntr}
\counterwithin{equation}{subsection}

\makeatletter
\let\oldsubsubsection\subsubsection

\renewcommand{\subsubsection}{\@ifstar{\subsubsectionstar}{\newsubsubsection}}

\newcommand{\subsubsectionstar}{\oldsubsubsection*}

\newcommand{\newsubsubsection}[1]{%
  \refstepcounter{equation}%
  \@startsection{subsubsection}{3}%
  {\z@}{.5\linespacing\@plus.7\linespacing}{-.5em}%
  {\normalfont\itshape}{#1}%
}
\makeatother

\newtheorem{theorem}[equation]{Theorem}
\newtheorem{thm}[equation]{Theorem}
\newtheorem{lemma}[equation]{Lemma}
\newtheorem{lem}[equation]{Lemma}

\newtheorem{cor}[equation]{Corollary}

\newtheorem{prop}[equation]{Proposition}

\theoremstyle{definition}
\newtheorem{df}[equation]{Definition}
\newtheorem{defn}[equation]{Definition}

\theoremstyle{remark}
\newtheorem{remark}[equation]{Remark}
\newtheorem{rem}[equation]{Remark}

\newif\iffinalrun
\iffinalrun
\else
  \fi

\iffinalrun
  \newcommand{\need}[1]{}
  \newcommand{\mar}[1]{}
\else
  \newcommand{\need}[1]{{\tiny *** #1}}
  \newcommand{\mar}[1]{\marginpar{\raggedright\tiny  %
		  fixme
      #1}}
\fi

\iffinalrun
  \newcommand{\finalmar}[1]{}
\else
  \newcommand{\finalmar}[1]{\marginpar{\raggedright\tiny   %
      #1}}
\fi

\newcommand{\A}{\AA}

\newcommand{\F}{\FF}

\newcommand{\Q}{\QQ}

\newcommand{\Z}{\ZZ}

\renewcommand{\AA}{{\mathbb A}}

\newcommand{\FF}{{\mathbb F}}

\newcommand{\QQ}{{\mathbb Q}}

\newcommand{\ZZ}{{\mathbb Z}}

\newcommand{\bA}{\ensuremath{\mathbf{A}}}

\newcommand{\bF}{\ensuremath{\mathbf{F}}}
\newcommand{\bG}{\ensuremath{\mathbf{G}}}

\newcommand{\bP}{\ensuremath{\mathbf{P}}}
\newcommand{\bQ}{\ensuremath{\mathbf{Q}}}

\newcommand{\bZ}{\ensuremath{\mathbf{Z}}}

\renewcommand{\bf}{\ensuremath{\mathbf{f}}}

\newcommand{\cA}{{\mathcal A}}
\newcommand{\cB}{{\mathcal B}}
\newcommand{\cC}{{\mathcal C}}
\newcommand{\cD}{{\mathcal D}}

\newcommand{\cH}{{\mathcal H}}
\newcommand{\cI}{{\mathcal I}}

\newcommand{\cK}{{\mathcal K}}

\newcommand{\cO}{{\mathcal O}}

\newcommand{\cX}{{\mathcal X}}

\newcommand{\Fbar}{\overline{\F}}
\newcommand{\Qbar}{\overline{\Q}}

\newcommand{\Fp}{\F_p}

\newcommand{\Fpbar}{\Fbar_p}

\newcommand{\Zp}{\Z_p}

\newcommand{\Qp}{\Q_p}

\newcommand{\Qpbar}{\Qbar_p}

\DeclareMathOperator{\QCoh}{QCoh}

\DeclareMathOperator{\coker}{coker}

\DeclareMathOperator{\End}{End}
\DeclareMathOperator{\Ext}{Ext}

\DeclareMathOperator{\Fun}{Fun}
\DeclareMathOperator{\Gal}{Gal}
\DeclareMathOperator{\GL}{GL}

\DeclareMathOperator{\Hom}{Hom}
\DeclareMathOperator{\RHom}{RHom}

\DeclareMathOperator{\Sets}{Sets}

\DeclareMathOperator{\Ind}{Ind}

\DeclareMathOperator{\Mod}{Mod}

\DeclareMathOperator{\PGL}{PGL}

\DeclareMathOperator{\Pro}{Pro}

\DeclareMathOperator{\Spec}{Spec}

\DeclareMathOperator{\Spf}{Spf}

\DeclareMathOperator{\Sym}{Sym}

\newcommand\cInd{c\text{-}\!\Ind}

\newcommand{\ladm}{\mathrm{l.adm}}

\newcommand{\nr}{\mathrm{nr}}

\newcommand{\St}{\mathrm{St}}

\newcommand{\rhobar}{\overline{\rho}}

\newcommand{\op}{\mathrm{op}}

\newcommand{\Res}{\operatorname{Res}}

\begin{document}

\title[Localization of smooth representations of $\mathrm{GL}_2(\mathbf{Q}_p)$]{Localization of smooth $\bm{p}$-power torsion %
representations of~$\bm{\mathrm{GL}_2(\mathbf{Q}_p)}$}

\author[A. Dotto]{Andrea Dotto}\email{andrea.dotto@kcl.ac.uk}\address{Department of Mathematics, King's College London,
London~WC2R 2LS,~UK}
\author[M. Emerton]{Matthew Emerton}\email{emerton@math.uchicago.edu}
\address{Department of Mathematics, University of Chicago,
5734 S.\ University Ave., Chicago, IL 60637, USA}
\author[T. Gee]{Toby Gee} \email{toby.gee@imperial.ac.uk} \address{Department of
  Mathematics, Imperial College London,
  London SW7 2AZ,~UK}
\thanks{AD was supported in various stages of this project by the Engineering and Physical Sciences
  Research Council [EP/L015234/1] (The London School of Geometry and Number Theory),
  the James D. Wolfensohn Fund at the Institute for Advanced Study,
  and a Royal Society University Research Fellowship. 
ME was supported in part by the
  NSF grants %
DMS-1601871,
DMS-1902307, and DMS-1952705.
 TG was 
  supported in part by an ERC Advanced grant, EPSRC grant EP/L025485/1 and a Royal Society Wolfson Research
  Merit Award. This project has received funding from the European Research Council (ERC) under the European Union’s Horizon 2020 research and innovation programme (grant agreement No. 884596).}
\begin{abstract}We show that the category of smooth representations of
  $\GL_2(\Qp)$ on $p$-power torsion modules localizes over a certain projective scheme, and give
  some applications.
\end{abstract}

\maketitle

\setcounter{tocdepth}{1}

\tableofcontents

\section{Introduction}%
In this paper we establish a localization theory
for the category of smooth representations of $\GL_2(\Qp)$ on
$\Zp$-modules on which $p$ is locally nilpotent,
for any prime $p \geq 5$.

\subsection{Initial statement of results}
In order to make a precise statement, we introduce some notation.
Let $\cO$ denote the ring of integers in a finite 
extension $E$ of $\Q_p$ and let $\cA$ denote the
category of smooth representations of $\GL_2(\Q_p)$ on $\cO$-modules
on which $p$ acts locally nilpotently, and which have a central
character equal to some fixed character
$\zeta: \Q_p^{\times} \to \cO^{\times}$. Let~$\F$ denote the residue
field of~$\cO$.

Let $X$ denote a chain of projective lines over $\F$ with ordinary double points, of length $(p- \zeta(-1))/2$.
Our definition of~$X$ is motivated by the
mod~$p$ semisimple local Langlands correspondence
for~$\GL_2(\Q_p)$ due to Breuil~\cite[Defn.\ 1.1]{BreuilGL2II}. %
Indeed, one can think of~$X$ as a moduli space of
semisimple representations $\rhobar:G_{\Qp}\to\GL_2(\Fpbar)$, with the
intersection points of the $\Pone$s corresponding to irreducible
representations. It then follows from Pa\v{s}k\={u}nas's
results~\cite{MR3150248} that the blocks in the
subcategory~$\cA^{\ladm}$ of locally admissible representations with
central character~$\zeta$ are in
bijection with the closed points of~$X$.

Given a closed subset $Y$ of $X$ with open complement
$U\coloneq X\setminus Y$, we let $\cA_Y$ denote the subcategory of $\cA$
consisting of those representations all of whose irreducible
subquotients lie in blocks corresponding to closed points of $Y$. The
subcategory~$\cA_Y$ is localizing (in the usual sense, recalled in
Appendix~\ref{app: localizing cats}), and is in particular a Serre
subcategory of~$\cA$, and we set $\cA_U\coloneq \cA/\cA_Y$.

Our main result shows that the category $\cA$ can be {\em localized}
over~$X$, in the following precise sense (see Theorem~\ref{thm: stack
  of abelian categories} and Remark~\ref{rem: what is a stack of abelian categories}).
\begin{thm}
\label{thm:intro localization}
The collection $\{\cA_U\}$
forms a stack {\em (}of abelian
categories{\em )} over the Zariski site of~$X$.  %
\end{thm}%

\subsection{Applications}\label{subsec: intro applications}In
Section~\ref{sec:examples} we demonstrate that %
our results have many concrete
consequences for the representation theory of~$\GL_2(\Qp)$. In
particular we  compute various  $\Ext^1$ groups between compact
inductions of Serre weights, and between such compact inductions and
irreducible representations. For example, in Proposition~\ref{ColmezquestionI}
we show that if ~$\sigma$ is a Serre weight and~$\pi$ is absolutely irreducible,
then
\[\dim \Ext_{\cA}^1(\pi, \cInd_{KZ}^G(\sigma))\leq 1;\]in addition, we
explicitly describe all of the non-split extensions that
arise. Similarly, in Proposition~\ref{prop: Colmez backwards
  extensions} we compute the dimensions of the~$\Ext^1$ groups for
extensions in the opposite direction (with a genericity assumption), and in Section~\ref{subsec:
  extensions between compact inductions} we compute the~$\Ext^1$
groups between full compact inductions of Serre weights.

In addition, in Section~\ref{subsec: structure of smooth
representations} we make precise the fashion in which the category of
finitely generated smooth representations is built out of finite
length representations (which admit finite filtrations by irreducible
representations), together with compact inductions of Serre
weights. In particular we prove Proposition~\ref{prop:structure},
which shows that every finitely generated representation~$\pi$ admits
a maximal subobject ~$\pi_{\fl}$ which is of finite length, and the
quotient $\pi/\pi_{\fl}$ is a successive extension of submodules of
representations $\cInd_{KZ}^G\sigma$, for $\sigma$ a Serre
weight. (The structure of these submodules is not mysterious; indeed,
as recalled in Remark~\ref{rem: submodules of cInd Serre weight}, most
such submodules are isomorphic to $\cInd_{KZ}^G\sigma$.)

\subsection{The relationship to the Bernstein centre}\label{subsec:
  Bernstein centre}
Our results can be understood by comparison with other, more classical contexts,
in the representation theory of $p$-adic reductive groups.  Note that the particular
choice of coefficient ring~$\cO$ does not affect things much, as long as its residue field~$\F$
has characteristic~$p$.    Thus it makes sense
to compare our results with the classical case of
smooth representations of a $p$-adic reductive group $G$ over a field of characteristic zero.

The first thing to note is that, in this classical case just as in our case,
there are many smooth representations, indeed naturally occurring ones,
that are not admissible.  For example, if $G = \GL_2(\Q_p)$
and $V$ is a finite dimensional representation of $KZ = \GL_2(\Z_p)\Q_p^{\times}$, then
$\cInd_{KZ}^G  V$ is admissible if and only if $V$ has vanishing ``Jacquet module''
(see e.g.~\cite[Thm.~1 supp.]{MR1068417} for a proof of a more general result),
a condition which certainly need not hold in characteristic zero (e.g.\ if $V$ is the
trivial representation). In characteristic~$p$, 
$\cInd_{KZ}^G V$ is \emph{never} admissible, provided $V$ itself is non-zero.
At least morally, this is because in characteristic~$p$, the Jacquet module
of a non-zero $V$ will never vanish.

However,
if $k$ is a field of characteristic zero, if $G$ is any $p$-adic reductive group, 
and if $\mathfrak Z$ denotes the Bernstein centre %
of the category of smooth $G$-representations on $k$-vector spaces,
then Bernstein's results \cite{MR771671}
show that any finitely generated %
smooth $G$-representation is admissible {\em over}~$\mathfrak Z$.
(This is the $p$-adic analogue of a theorem of Harish-Chandra,
to the effect that if $G$ is a real reductive group with Lie algebra~$\mathfrak g$,
then any finitely generated $\bigl(U(\mathfrak g),K\bigr)$-module is admissible over the centre
$\mathfrak Z$ of $U(\mathfrak g)$.)  
Thus we may think of a not-necessarily-admissible smooth $G$-representation
in characteristic zero as being a ``family'' of admissible representations 
parameterized by some subscheme of $\Spec(\mathfrak Z)$.
An analogous result has also recently been proved by Dat--Helm--Kurinczuk--Moss \cite{https://doi.org/10.48550/arxiv.2203.04929}
for the category
of smooth representations of~$G$ over a field of characteristic~$\ell \neq p$,
or more generally over a Noetherian $\Z_{\ell}$-algebra. %

On the other hand, the analogous result
is {\em not} true for our category~$\cA$.
Indeed,
independent results of A.D.\ and Ardakov--Schneider~\cite{https://doi.org/10.48550/arxiv.2105.06128, modpcentres} 
show that the Bernstein centre of~$\cA$
is trivial.
Thus any finitely generated non-admissible representation (such as the $\cInd_{KZ}^G V$
introduced above)
provides a counterexample to the analogue
of Bernstein's result.

The localization theory 
of this paper was developed in part to rectify this absence of an interesting 
Bernstein centre for~$\cA$:
rather than simply regarding objects of $\cA$ as lying over the Spec of its
Bernstein centre, we localize them over the non-affine variety~$X$. 
Then, by forming the 
Bernstein centres of the various categories~$\cA_U$,
we may endow $X$ with the structure of a topologically ringed space,   
and we expect to prove (in forthcoming work) that this endows $X$
with the structure of a formal scheme over~$\cO$ --- which we denote by $\widehat{X}$ --- 
such that $X$ itself is the underlying  reduced  scheme over~$\F$ of~$\widehat{X}$.
The triviality of the Bernstein centre of~$\cA$ 
would then correspond to the fact that the only global sections of the
structure sheaf $\cO_{\widehat{X}}$ are the constant functions.
In turn, this statement itself would be an extension to the thickening $\widehat{X}$
of the fact that the only
globally defined functions on the connected projective variety $X$ are the
constant functions --- thus the 
fact that $X$ is projective, rather than affine, is closely related to
the fact that the
Bernstein centre of $\cA$ is trivial.

\subsection{The relationship with local Langlands}%
Our definition of~$X$ suggests that there is a connection between our results and 
the mod $p$ or $p$-adic local Langlands correspondence, and indeed, this is the
case.  To explain and motivate this, we first recall the analogous result
in the $\ell \neq p$ context.  Namely, if $k$ is a field of characteristic
zero or of characteristic $\ell \neq p$, or more generally if $A$  is a complete DVR whose
residue field is such a field~$k$, then, for any $p$-adic field $F$ and
any $n \geq 1$,  we may consider the moduli stack $\cX$ parameterizing $n$-dimensional
representations of the Weil--Deligne group of~$F$ over $A$-algebras,
and the local Langlands correspondence over~$A$ identifies $\mathfrak Z$
--- the Bernstein centre of $A[\GL_n(F)]$ --- with the ring of functions
on~$\cX$, by a theorem of Helm--Moss \cite{MR3867634}.
Since $\cX$ is a reductive quotient of an affine scheme,
we may rephrase this as saying that $\Spec \mathfrak Z$ is the
moduli space associated to~$\cX$, or more precisely an adequate moduli space in the sense of Alper~\cite{MR3272912}. %

In forthcoming work we expect to prove 
a $p$-adic analogue of this result.  Namely, we will show
that $\widehat{X}$ --- the thickening of $X$ induced by the localized Bernstein centre
of~$\cA$ discussed above --- is a formal moduli space associated to
the %
stack~$\cX_{2, \bQ_p}$
parameterizing $(\varphi,\Gamma)$-modules arising from continuous rank two $p$-adic
representations of~$\Gal(\Qpbar/\Qp)$ 
of appropriately fixed determinant.

\begin{rem}\label{rem: categorical langlands remark about further localizing}
  In the $\ell \neq p$ case, the identification of the Bernstein
  centre with the ring of functions on~$\cX$ is an outward
  manifestation of a deeper phenomenon, namely the categorical local
  Langlands correspondence considered in \cite{benzvi2020coherent,
    fargues--scholze, hellmann2020derived, zhu2020coherent}. In the
  companion paper~\cite{DEGcategoricalLanglands}, we establish
  (using the results of this paper as one of our tools) an analogous
  categorical $p$-adic local Langlands correspondence
  for~$\GL_2(\Qp)$, in the form of a fully faithful functor from the
  derived category of~$\cA$ to an appropriate derived category of
  coherent sheaves on the stack of $(\varphi,\Gamma)$-modules ~$\cX_{2, \bQ_p}$
  mentioned above.

 However, while we anticipate that fully faithful functors of this kind exist
  in great generality (in particular, for the representations of
  $\GL_n(F)$ for any~$n$ and any $p$-adic local field~$F$), we don't
  expect the results of the present paper to generalise in any obvious
  way, even to $\GL_2(F)$ with $F\ne\Qp$.

  Relatedly, on the Galois side of the Langlands correspondence, we do
  not expect the stacks of $(\varphi,\Gamma)$-modules to admit interesting associated moduli spaces
  beyond the cases of~$\GL_2(\bQ_p)$ and $\GL_1(F)$, $F$ being any finite extension of~$\Q_p$. 
For example, if  $\cX_{2, \bQ_{p^2}}$ denotes the analogous stack for~$\GL_2(\Q_{p^2})$,
then the various specialisation relations in reducible
families cause any morphism $\cX_{2, \bQ_{p^2}} \to Z$, with $Z$ being a locally separated
formal algebraic space, to factor through the structure morphism $\cX_{2, \bQ_{p^2}} \to \Spf \cO$.
\end{rem}

\subsection{Methods of proof}
The subcategory $\cA^{\ladm}$
of  $\cA$ consisting of locally
admissible
representations is well-understood, thanks to the
results of Pa\v{s}k\={u}nas~\cite{MR3150248}. %
As already mentioned above, it  factors into a product of {\em
  blocks}, which are labelled
by the closed points of~$X$.  In our perspective, the objects
of~$\cA^{\ladm}$ are supported on finite closed subsets of~$X$,
so that our localization theory, for these objects, is just
a restatement of Pa\v{s}k\={u}nas's results.

The novel aspects of our theory are seen when the objects under consideration
are not locally admissible, since then they will localize over subsets of~$X$ with
non-empty interior.
A basic consequence of the localization theory  is that, if two representations
have disjoint support, then all $\Ext^i$ groups between them vanish.  Conversely, proving such statements is the key to proving Theorem~\ref{thm:intro
  localization}. The key result is Lemma~\ref{lem:ext vanishing},
which (at least morally) shows that  if~$\sigma$ is a Serre weight,
and $\pi$ is a supported on some closed subset~$Y$, then the various $\Ext^i$ between
$\pi$ and $\cInd_{KZ}^G \sigma$ are also supported on~$Y$. It is
proved by exploiting the fact that these $\Ext^i$ groups are typically not
finite dimensional over~$\F$, but  they are of countable dimension.
Since they are also modules over the Hecke algebra~$\cH(\sigma) \cong \F[T]$, this makes them amenable to an application of a well-known technique
of Dixmier~\cite{MR182682}, at least when $\F$ is uncountable; and we
can always replace~$\F$ by an uncountable extension via an appropriate
base-change. With this result in hand, we are able to show that for
any open cover of~$X$, the corresponding
``\v{C}ech resolution'' of any~$\pi$ is acyclic, and we deduce
Theorem~\ref{thm:intro localization} from this acyclicity.

\begin{rem}
  \label{rem: use of p-adic LL}We do not explicitly use the $p$-adic local Langlands
  correspondence in our arguments. However, we do use the results of Pa\v{s}k\={u}nas
~\cite{MR3150248}, whose proofs use $p$-adic local Langlands. In
particular we frequently use the
  classification of blocks in the locally admissible
  category~$\cA^{\ladm}$, and for some of the results in Section~\ref{subsec:completion}, we furthermore make use of Pa\v{s}k\={u}nas's description
  of the Bernstein centres of these blocks.
\end{rem}

\begin{rem}
  \label{rem: p equals 2 or 3}We expect that the results of this paper
  extend to the cases of $p=2$, $3$, with minor modifications to the
  statements (e.g.\ to the definition of~$X$ in the case $p=2$: in
  this case $X$ should consist of a single~$\Pone$). %
  Our arguments would
  necessarily become more complicated in the cases $p=2,3$, so that we could
  no longer argue uniformly in~$p$. Note that this already happens for the
  description of blocks in~\cite{MR3150248}, which is not carried out
  in loc.\ cit.\ for $p=2,3$, but is rather in~\cite{MR3444235, MR4350140}.  %
  Thus we have not attempted to make the modifications of our
  arguments that would be required to treat these cases.
\end{rem}

\subsection{A brief guide to the paper}
\label{sec:guide-paper}
In Section~\ref{subsec:smooth theory} we recall (and in some cases extend) some basic results on
the smooth $p$-adic representation theory of~$\GL_2(\Qp)$;  in particular, we explain
our interpretation of Pa\v{s}k\={u}nas'
classification~\cite{MR3150248} of the blocks of locally admissible
representations in terms of the closed points of~$X$.

We establish our localization theory in Section~\ref{sec: localization
  of G reps}. We begin with the basic definitions of our localizing
categories, and then prove the crucial Lemma~\ref{lem:ext vanishing}
mentioned above, and (with some work) deduce Theorem~\ref{thm:intro
  localization}. The rest of Section~\ref{sec: localization
  of G reps} is devoted to a discussion of a representation-theoretic notion
 of completion along a closed
subset of~$X$, culminating in Theorem~\ref{thm:BL gluing}, which is
an analogue of Beauville--Laszlo gluing in this setting. 

In Section~\ref{sec:examples} we give a variety of examples and
applications of our localization theory, showing in particular how it
can be used to compute extensions between compact inductions
$\cInd_{KZ}^G\sigma$, and extensions between such compact inductions
and admissible representations. Finally in Appendix~\ref{sec: cat
  theory background} we recall some background material in category
theory that we use in the body of the paper.

\subsection{Notation and conventions}\label{subsec: notation and
  conventions}We fix throughout the paper a prime~$p \geq 5$. %
Fix an algebraic closure~$\Qpbar$ of~$\Qp$, and write~$G_{\Qp}$ for the absolute Galois
group~$\Gal(\Qpbar/\Qp)$. %
Let $\cO$ denote the ring of integers in a fixed finite extension
$E$ of $\Q_p$, let~$\varpi$ be a uniformizer of~$\cO$, and let~$\F$ be the residue field of~$\cO$; all
representations considered in this paper will be on
$\cO$-modules. %

We write~$G=\GL_2(\Q_p)$,
$K=\GL_2(\Zp)$, $\Iw$ for the upper-triangular Iwahori subgroup of~$K$, $N$ for the normalizer of~$\Iw$ in~$G$, and~$Z=\Qp^\times$ for the centre of~$G$. 
We fix a
continuous character $\zeta:\Qptimes\to \cO^\times$ throughout the
paper. 
We will say that~$\zeta$ is even, respectively odd, if the restricted character $\bF_p^\times \subset \bZ_p^\times \to \cO^\times$ is a square, respectively is not a square; equivalently, $\zeta$ is even if and only if~$\zeta(-1)=1$.
We write $\omega: \bQ_p^\times \to \bF_p^\times$ for the reduction mod~$p$ of the character $\bQ_p^\times \to \bZ_p^\times, x \mapsto x|x|$. 

If~$A$ is an $\cO$-algebra, and~$f \in A[T]$ is a polynomial, we will write $\nr_{f}$ for the unramified $A[\bQ_p^\times]$-module~$A[T]/f(T)$ on which~$p$ acts as multiplication by~$T$.
(Here ``unramified'' means that $\Z_p^{\times}$ acts trivially.)
If~$\lambda \in A^\times$, we define $\nr_\lambda \coloneq  \nr_{T-\lambda}$.

We always let~$A$ denote a complete Noetherian local~$\cO$-algebra, and we write~$\kbase$ for its residue field.
Here ``local'' is understood in the strong sense that $A$ is a local
ring, and the morphism $\cO \to A$  is a local morphism;
given the first assumption,
this second hypothesis is equivalent to requiring that the residue
field~$\kbase$ be of characteristic~$p$.  We then let $\cA_{A}$ denote the
category of smooth $G$-representations with central character~$\zeta$
on locally $\mathfrak m_A$-torsion $A$-modules. In more detail, $\mathfrak m_A$ denotes
the  maximal ideal of~$A$, a module is {\em locally
  $\mathfrak m_A$-torsion} if each element is annihilated by some power of~$\mathfrak m_A$,
and a representation is \emph{smooth}
if each element is fixed by some open subgroup of~$G$;
so we are considering precisely those representations over~$A$ which are called
smooth in \cite{MR2667882} (and which have the required central character). 
In the case that $A=\cO$ we write $\cA$ for $\cA_A$, and~$\cA^{\ladm}$ for the full subcategory of locally admissible representations.
If~$V$ is a finite length object of~$\cA$ we will write~$\JH(V)$ for
the multiset of Jordan--H\"older factors of~$V$. 

We will write~$I_\zeta$ for the two-sided ideal of~$A[G]$ or~$A[KZ]$
(or other group algebras of groups containing~$Z$) generated
by~$[z]-\zeta(z)$ for~$z \in Z$.
We write~$A[KZ]_\zeta = A[KZ]/I_\zeta$, and similarly for other group algebras.

We write $\cC_{A}$ for the category of smooth
$KZ$-representations with central character~$\zeta$ on locally $\fm_A$-torsion $A$-modules. In the case that $A=\cO$ we write $\cC$ for $\cC_A$, and note that~$p$ is locally nilpotent on all objects of~$\cC$.
A \emph{Serre weight} is an irreducible object of~$\cC_{A}$, or equivalently of~$\cC_\kbase$.
Since every irreducible $\kbase$-representation of~$K$ is defined
over~$\bF_p$, and~$Z$ is acting by~$\zeta$, the set of isomorphism
classes of Serre weights is independent of~$A$.

More explicitly, the isomorphism classes of Serre weights are
represented by the representations $\Sym^b k^2 \otimes \det^{a}$
of~$\GL_2(\Fp)$, where $0\le a<p-1$ and~$0\le b\le p-1$, which we sometimes abbreviate to $\Sym^b \otimes \det^a$. 
It is sometimes
convenient to view~$a$ as an element of~$\Z/(p-1)\Z$, and we will do so
without further comment.

We make some use of tame types, by which we mean the principal series
and cuspidal representations of $E[\GL_2(\Fp)]$. For any pair of characters
$\chi_1\ne\chi_2:\Fptimes\to \cO^\times$, we have the principal series
representation
$I(\chi_1,\chi_2)=\Ind_{B(\Fp)}^{\GL_2(\Fp)}\chi_1\otimes\chi_2$, and
for any character~$\chi:\F_{p^2}^\times\to\cO^{\times}$ which does not
factor through the norm, we have the corresponding cuspidal
representation~$\Theta(\chi)$. (Up to sign, these are the
Deligne--Lusztig inductions of regular characters of the nonsplit
maximal torus in~$\GL_2(\bF_p)$.) These representations are irreducible
of dimensions~$p+1$, $p-1$ respectively, and the only isomorphisms
between them are that $I(\chi_1,\chi_2)\cong I(\chi_2,\chi_1)$ and
$\Theta(\chi)\cong\Theta(\chi^p)$.

The Jordan--H\"older factors of their reductions modulo~$p$ are as
follows. Firstly, if we write~$\chibar_i(x)=x^{n_i}$, then the Jordan--H\"older
factors of the reduction of~$I(\chi_1,\chi_2)$
are~$\Sym^{[n_1-n_2]}\otimes\det^{n_2}$ and~$\Sym^{[n_2-n_1]}\otimes\det^{n_1}$, where
$[n_i-n_j]$ denotes the unique integer in~$(0,p-1)$ congruent to
$n_i-n_j\pmod{p-1}$.
(Note that since the $\chi_i$ are distinct and of prime-to-$p$ order,
their reductions are also distinct, and so $n_i \neq n_j$.)

Secondly, if~$\chibar(x)=x^{i+(p+1)j}$ with~$1\le i\le p$, then
if~$1<i<p$ then the reduction of ~$\Theta(\chi)$ has two
Jordan--H\"older factors, namely~ $\Sym^{i-2}\otimes\det^{1+j}$ 
and~$\Sym^{p-1-i}\otimes\det^{i+j}$; while if~$i=1$ or~$p$, then it is
irreducible and isomorphic to~$\Sym^{p-2}\otimes\det^{1+j}$.

\subsection{Acknowledgements}\label{subsec: acknowledgements}
We would like to thank Pierre Colmez, Gabriel Dospinescu, and Vytautas
Pa\v{s}k\={u}nas for helpful correspondence and conversations, and
comments on an earlier draft of this paper.
We would also like to thank the referee for their careful reading of the paper, and their many helpful corrections and suggestions.
\section{Preliminaries on smooth \texorpdfstring{$\GL_2(\Q_p)$}{GL2(Qp)}-representations}
\label{subsec:smooth theory}

\subsection{Irreducible representations.}\label{subsec: irred reps}

If~$A = \kbase$ is an algebraically closed field of characteristic~$p$, the irreducible objects of~$\cA_{\kbase}$ have been classified in~\cite{BarthelLivneDuke, BreuilGL2I}. 
As we will often need to extend scalars to transcendental extensions of~$\kbase$, in this section we will extend their results to complete Noetherian local $\cO$-algebras~$A$ whose residue field is an arbitrary extension $\kbase/\F_p$.
While for our purposes in this paper we could assume that~$\kbase$ is perfect, for completeness we include the imperfect case.

Note that the maximal ideal of~$A$ acts trivially on every irreducible object of~$\cA_A$, and so the irreducible objects of~$\cA_A$ and~$\cA_k$ coincide.
We will often make implicit use of this observation.
Recall also that the restriction functor %
from~$\cA_A$ to~$\cC_A$ has an exact left adjoint, given by compact induction, and denoted
\[
V \mapsto \cInd_{KZ}^G(V).
\]
The case when~$V=\sigma$ is an irreducible $\kbase[KZ]_\zeta$-module (a
``Serre weight'') is of particular importance. We begin by recalling
the description, due to Barthel and Livn\'e~\cite[Proposition~8, Theorem~19]{BarthelLivneDuke},
of the endomorphism ring of $\cInd_{KZ}^G(\sigma)$, which we denote by $\cH(\sigma)\coloneq \End_{\cA_A} (\cInd_{KZ}^G \sigma)$.
Since~$\sigma$ is a $k$-module, we also have $\cH(\sigma) = \End_{\cA_k}(\cInd_{KZ}^G\sigma)$.

\begin{lemma}
\label{lem:Hecke algs}
Let~$A$ be a complete Noetherian local $\cO$-algebra, and let $\sigma = \Sym^b \otimes \det^a$ be a Serre weight.
We have an isomorphism
\begin{equation}\label{eqn: Hecke algebra of Serre weight}
\cH(\sigma)\cong \kbase[T],
\end{equation}
generated by 
the unique Hecke operator~$T$ supported on $KZ \left ( \begin{smallmatrix} 1&0\\0&p\end{smallmatrix} \right )^{-1}K$,
and normalized as in~\cite[Proposition~8]{BarthelLivneDuke},
i.e.\ sending $\left ( \begin{smallmatrix} 1&0\\0&p\end{smallmatrix} \right )^{-1}$
to $\Sym^b \left ( \begin{smallmatrix} 0&0\\0&1\end{smallmatrix} \right )$.
Furthermore,
$\cInd_{KZ}^G\sigma$ is free over~$k[T]$.
\end{lemma}
\begin{proof}
The proofs of~\cite[Proposition~8, Theorem~19]{BarthelLivneDuke} do not make use of the assumption that~$k$ is algebraically closed.
Alternatively, one can argue by extension of scalars, using~\cite[Lemma~5.1]{MR3150248}.
\end{proof}

\begin{rem}
In what follows we will denote the Hecke operator~$T$ by~$T_p$ in case of a notational clash, for instance if~$T \in \kbase[T]$ is a formal variable.
\end{rem}

We next recall
the classification of $G$-equivariant morphisms
between the compact inductions of Serre weights.
For~$\GL_2$, this is essentially contained in~\cite{BarthelLivneDuke}; see~\cite[Proposition~6.2, Example~6.14]{MR2845621} for a generalization to~$\GL_n$.
As for Lemma~\ref{lem:Hecke algs}, the proofs go through unchanged to the case of non-algebraically closed~$k$.

\begin{lemma}
\label{lem:cind homs}
Let~$A$ be a complete Noetherian local $\cO$-algebra.
If $\sigma$ and $\sigma'$ are distinct Serre weights
then 
\[
\Hom_{\cA_A}(\cInd_{KZ}^G \sigma', \cInd_{KZ}^G \sigma) = 0
\]  
{\em except} in the case when $\{\sigma,\sigma'\} =
\{ \det{}^s, \Sym^{p-1}\otimes \det{}^s\}$
for some~$s$,
in which case  
we have the following:
\begin{enumerate}
\item
All elements of
$$\Hom_{\cA_A}\bigl(\cInd_{KZ}^G \Sym^0\otimes  \det{}^s,
\cInd_{KZ}^G \Sym^{p-1}\otimes\det{}^s\bigr)$$
and of 
$$\Hom_{\cA_A}\bigl(\cInd_{KZ}^G \Sym^{p-1}\otimes\det{}^s,
 \cInd_{KZ}^G \Sym^0 \otimes \det{}^s\bigr)$$
are $k[T]$-equivariant, where $k[T]$ acts on each of source and target
in the natural way, i.e.\ with $T$ acting by the appropriate~$T_p$.
\item Each of  
$$\Hom_{\cA_A}\bigl(\cInd_{KZ}^G \Sym^0\otimes\det{}^s,
\cInd_{KZ}^G \Sym^{p-1}\otimes\det{}^s\bigr)$$
and  
$$\Hom_{\cA_A}\bigl(\cInd_{KZ}^G \Sym^{p-1}\otimes\det{}^s,
 \cInd_{KZ}^G \Sym^0\otimes\det{}^s\bigr)$$
is free of rank~$1$ over~$k[T]$.
\item We can choose $k[T]$-generators 
$$\alpha\in
\Hom_{\cA_A}\bigl(\cInd_{KZ}^G \Sym^0\otimes\det{}^s,
\cInd_{KZ}^G \Sym^{p-1}\otimes\det{}^s\bigr)$$
and 
$$\beta\in
\Hom_{\cA_A}\bigl(\cInd_{KZ}^G \Sym^{p-1}\otimes\det{}^s,
 \cInd_{KZ}^G \Sym^0\otimes\det{}^s\bigr)$$
such  that
$$\alpha\circ\beta=\beta\circ\alpha=T^2-\zeta(p).$$
The cokernels of~$\alpha$ and~$\beta$ have length two if~$\zeta(p)$ is a square in~$k^\times$, and length one otherwise. %
\end{enumerate}
\end{lemma}

We record a consequence of Lemma~\ref{lem:cind homs}.

\begin{lemma}\label{quotient by polynomial}
Let~$A$ be a complete Noetherian local $\cO$-algebra with residue field~$k$.
\begin{enumerate}
\item For every nonzero subobject $\pi \subset \cInd_{KZ}^G \sigma$, there exists $f(T) \ne 0 \in k[T]$ such that $f(T)\cInd_{KZ}^G \sigma \subset \pi$.
\item For every irreducible object $\pi \in \cA_A$, 
there exist a Serre weight~$\sigma$, an irreducible polynomial~$f(T) \in k[T]$,
and a surjection
$$\cInd_{KZ}^G \sigma / f(T)\cInd_{KZ}^G \sigma \to \pi.$$
\end{enumerate}
\end{lemma}
\begin{proof}
Proof of part~(1): By smoothness, $\pi$ necessarily contains a Serre weight~$\sigma'$, hence it contains the image of a nonzero morphism $\cInd_{KZ}^G(\sigma') \to \cInd_{KZ}^G(\sigma)$.
Then part~(1) follows from Lemma~\ref{lem:cind homs}. 

Proof of part~(2): Note that~$\pi$ is an object of~$\cA_k$,
and so there exist a Serre weight~$\sigma$ and a nonzero map $\cInd_{KZ}^G(\sigma) \to \pi$, which must be surjective, since~$\pi$ is irreducible.
Applying part~(1) to the kernel of this surjection, we find a nonzero $f(T) \in k[T]$ and a surjection
$\cInd_{KZ}^G \sigma / f(T)\cInd_{KZ}^G \sigma \to \pi.$
Factoring~$f(T)$ and passing to a direct summand, we can assume that~$f(T)$ is a power of an irreducible polynomial.
Then part~(2) follows by filtering $\cInd_{KZ}^G(\sigma)/f^i \cInd_{KZ}^G(\sigma)$ by powers of~$f$.
\end{proof}

By Lemma~\ref{quotient by polynomial}~(2), classifying irreducible objects of~$\cA_A$ amounts to analyzing the structure
of quotients 
$\cInd_{KZ}^G \sigma / f(T)\cInd_{KZ}^G \sigma$ with $f(T)$ irreducible. 
In the case when $k$ is algebraically closed,
so that $f(T) = T - \lambda$ for some $\lambda \in k$,
this was done by Barthel and Livn\'e when $\lambda\neq 0,$
and by Breuil~\cite{BreuilGL2I} when $\lambda= 0$. 
Theorem~\ref{classification} and Lemma~\ref{lem: quotient giving St and trivial} below summarize their results.

\begin{thm}\label{classification} %
Assume that~$\kbase$ is algebraically closed.
Then every irreducible object of~$\cA_{\kbase}$ is isomorphic to a representation in the following list:
\begin{enumerate}%
\item $\eta \circ \det$ for some smooth character $\eta: \Q_p^\times \to \kbase^\times$ such that $\eta^2 = \zeta$;
\item $(\eta \circ \det) \otimes \St$ for some smooth character $\eta: \Q_p^\times \to \kbase^\times$ such that $\eta^2 = \zeta$, 
where~$\St$ is the Steinberg representation of~$G$; or %
\item $\cInd_{KZ}^G(\sigma)/f(T)\cInd_{KZ}^G(\sigma)$, where~$\sigma$ is an irreducible representation of~$\kbase[KZ]_\zeta$ and~$f = T-\lambda$ for some~$\lambda \in \kbase$, and
\[
(\sigma, \lambda) \not \in \{(\Sym^0 \otimes \operatorname{det}^s, \pm \zeta(p)^{1/2}), (\Sym^{p-1} \otimes \operatorname{det}^s, \pm \zeta(p)^{1/2})\}.
\]
Conversely, all these representations are irreducible and pairwise non-isomorphic, with the following exceptions:
\begin{equation}\label{eqn:first exception}
(\cInd_{KZ}^G \Sym^0 \otimes \operatorname{det}^s)/(T-\lambda) \cong (\cInd_{KZ}^G \Sym^{p-1} \otimes \operatorname{det}^s)/(T-\lambda)
\end{equation}
if  $\lambda \ne \pm \zeta(p)^{1/2}$,
and
\begin{equation}\label{eqn:second exception}
(\cInd_{KZ}^G \Sym^r \otimes \operatorname{det}^s)/T \cong  (\cInd_{KZ}^G \Sym^{p-1-r} \otimes \operatorname{det}^{r+s})/T.
\end{equation}
\end{enumerate}
\end{thm}

By~\cite[Theorem~30]{BarthelLivneDuke}, the representations in case~(3) can be described as parabolic inductions from the upper-triangular Borel subgroup~$B$, via the isomorphism %
\begin{equation}\label{eqn:Hecke eigenvalue of parabolic induction}
\cInd_{KZ}^G(\Sym^r \otimes \det\nolimits^s)/(T-\lambda) \isom \Ind_B^G(\nr_{\lambda^{-1}\zeta(p)}\omega^s \otimes  \nr_\lambda\omega^{r+s}).
\end{equation}
We next recall what happens in the excluded cases of part~(3) of the preceding theorem
(without necessarily assuming that $k$ is algebraically closed).
We will employ the notation~$\nr_{T^2-\zeta(p)}$ from Section~\ref{subsec: notation and conventions}.

\begin{lem}
  \label{lem: quotient giving St and trivial}
  There are non-split short exact
  sequences %
{\small
\begin{gather*}
0 \to \cInd_{KZ}^G(\Sym^0 \otimes \det\nolimits^s) \xrightarrow{\alpha} \cInd_{KZ}^G(\Sym^{p-1} \otimes \det\nolimits^s) \to (\omega^s \nr_{T^2-\zeta(p)} \circ \det) \otimes \St \to 0\\
0 \to \cInd_{KZ}^G(\Sym^{p-1} \otimes \det{}^s) \xrightarrow{\beta} \cInd_{KZ}^G(\Sym^{0} \otimes \det{}^s) \to \omega^s \nr_{T^2-\zeta(p)} \circ \det \to 0.
\end{gather*}
}
If~$\zeta(p)$ is furthermore a square in~$k^\times$, then there are non-split short exact sequences
{\small
\begin{gather*}
0\to (\omega^s\nr_{\mp \zeta(p)^{1/2}}\circ \det) \to (\cInd_{KZ}^G\Sym^{p-1}\otimes \det{}^s)/(T\pm \zeta(p)^{1/2})\to (\omega^s\nr_{\mp \zeta(p)^{1/2}}\circ \det)\otimes\St\to 0\\
0\to (\omega^s\nr_{\mp \zeta(p)^{1/2}}\circ \det)\otimes\St \to (\cInd_{KZ}^G\Sym^{0}\otimes \det{}^s)/(T\pm \zeta(p)^{1/2})\to (\omega^s\nr_{\mp \zeta(p)^{1/2}}\circ \det) \to 0.
\end{gather*}
}

\end{lem}

The short exact sequences in Lemma~\ref{lem: quotient giving St and trivial} are induced
by the morphisms of Lemma~\ref{lem:cind homs}(3). 
It is a standard result that they are not split; 
for the first two, this is a consequence of Lemma~\ref{quotient by polynomial}, 
and for the last two, we will give a proof in the course of proving Lemma~\ref{lem: uniserial} below. 

We will now generalize Theorem~\ref{classification} to an arbitrary coefficient field.
See~\cite[Section~5.3]{MR3150248} for a similar result under the assumption that~$\kbase$ is perfect.
We will make use of the following lemma.

\begin{lem}
\label{lem:mapstoquotient}
Let~$A$ be a complete Noetherian local $\cO$-algebra.
Let~$\sigma_1, \sigma_2$ be Serre weights and write~$\pi_i = \cInd_{KZ}^G(\sigma_i)$. 
Assume that $\Hom_{\cA_A}(\pi_1, \pi_2) \ne 0$ and let~$g \in k[T]$ be a monic irreducible polynomial.
\begin{enumerate}
\item The exact sequence
\[
0 \to \pi_2 \xrightarrow{g^n} \pi_2 \to \pi_2/g^n\pi_2 \to 0
\]
induces a surjection $\Hom_{\cA_A}(\pi_1, \pi_2) \to \Hom_{\cA_A}(\pi_1, \pi_2/g^n\pi_2)$. 
\item We have equalities
\[
\Hom_{\cA_A}(\pi_1, \pi_2)/g^n\Hom_{\cA}(\pi_1, \pi_2) = \Hom_{\cA_A}(\pi_1, \pi_2/g^n\pi_2) = \Hom_{\cA_A}(\pi_1/g^n\pi_1, \pi_2/g^n\pi_2), 
\]
and the two $k[T]$-actions \emph{(}by the $T_p$-operator on~$\pi_1$ and~$\pi_2$\emph{)} on this space coincide.
\end{enumerate}
\end{lem}
\begin{proof}
Let~$k[T]$ act on~$\pi_2$ via the Hecke operator~$T_p$.
We begin by proving part~(1).
By Lemma~\ref{lem:cind homs}(2), the image of $\Hom_{\cA_A}(\pi_1, \pi_2) \to \Hom_{\cA_A}(\pi_1, \pi_2/g^n\pi_2)$ is $k[T]$-isomorphic to 
\[
\Hom_{\cA_A}(\pi_1, \pi_2)/g^n \Hom_{\cA_A}(\pi_1, \pi_2) \cong k[T]/g^nk[T] 
\]
and so it has $k$-dimension $n\deg(g)$.
So it suffices to prove that 
\[
\dim_k \Hom _{\cA_A}(\pi_1, \pi_2/g^n\pi_2) \leq n\deg(g).
\]
By d\'evissage we reduce to the case~$n = 1$.
Since~$\pi_1$ and~$\pi_2$ are finitely generated over~$\kbase[G]$, by~\cite[Lemma~5.1]{MR3150248} the inequality can be checked after extending scalars to an algebraic closure~$\overline{\kbase}$ of~$\kbase$.
Replacing~$\kbase$ by~$\overline{\kbase}$ we can assume that~$g$ is a product of linear factors (possibly with multiplicity if~$\kbase$ is not perfect) and so $\pi_2/g\pi_2$ has a finite filtration with graded pieces isomorphic to $\pi_2/(T-\lambda)\pi_2$ for some~$\lambda \in k$.
Now it suffices to prove that %
\[
\dim_k\Hom_{\cA_A}(\pi_1, \pi_2/(T-\lambda)\pi_2) \leq 1.
\]
If~$\pi_2/(T-\lambda)\pi_2$ is irreducible this is true because the $K$-socle of absolutely irreducible $G$-representations is multiplicity-free.
Otherwise, by Theorem~\ref{classification}, we necessarily have that~$\lambda$ is a square root of $\zeta(p)$. 
Then the inequality is true because Lemma~\ref{lem: quotient giving St and trivial} shows that~$\pi_2/(T-\lambda)\pi_2$ has length two, 
and the two irreducible factors have non-isomorphic $K$-socles.
This concludes the proof of~(1), and also shows that
\[\Hom_{\cA_A}(\pi_1, \pi_2)/g^n\Hom_{\cA_A}(\pi_1, \pi_2) \isoto \Hom_{\cA_A}(\pi_1, \pi_2/g^n\pi_2).\]

We now prove the second claim. 
The two $k[T]$-module structures on $\Hom_{\cA_A}(\pi_1, \pi_2)$ agree by Lemma~\ref{lem:cind homs}.
It follows from what we have just proved that they also agree on $\Hom_{\cA_A}(\pi_1, \pi_2/g^n\pi_2)$. 
Since~$g^n = 0$ on the second factor, we deduce that 
\[
\Hom_{\cA_A}(\pi_1, \pi_2/g^n\pi_2) = \Hom_{\cA_A}(\pi_1/g^n\pi_1, \pi_2/g^n\pi_2).
\qedhere
\]
\end{proof}

We can now describe the cokernels of irreducible Hecke polynomials on $\cInd_{KZ}^G \sigma$.

\begin{theorem}\label{classifyirreducibles}
Let $A$ be a complete Noetherian local $\cO$-algebra with residue field~$k$.
\begin{enumerate}
\item Let~$\sigma$ be a Serre weight, let $f(T) \in k[T]$ be an irreducible polynomial, and assume that~$f(T)$ is coprime to~$T^2-\zeta(p)$ if 
 $\sigma \in \{\Sym^0 \otimes \det^a, \Sym^{p-1} \otimes \det^a\}$.
Then $\cInd_{KZ}^G \sigma/f(T) \cInd_{KZ}^G\sigma$ is irreducible.
\item Let~$f(T)$ be an irreducible divisor of~$T^2-\zeta(p)$. Then the representations 
\[
\cInd_{KZ}^G (\Sym^0 \otimes \det\nolimits^a)/f(T) \cInd_{KZ}^G(\Sym^0 \otimes \det\nolimits^a) 
\]
and 
\[
\cInd_{KZ}^G (\Sym^{p-1} \otimes \det\nolimits^a)/f(T) \cInd_{KZ}^G(\Sym^{p-1} \otimes \det\nolimits^a)
\]
have the same semisimplification, which has length~$2$ and is isomorphic to %
\[
\omega^a\nr_{f(T)} \circ \det \oplus \bigl( (\omega^a\nr_{f(T)} \circ \det) \otimes \St \bigr)
\]
where~$\nr_{f(T)}$ is the unramified character defined in Section~\ref{subsec: notation and conventions}.
\item If~$\pi$ is an irreducible object of~$\cA_A$, then there exist a Serre weight~$\sigma = \Sym^r \otimes \det^s$ with $0 \leq r \leq p-2$, and an irreducible $f(T) \in k[T]$,
such that~$\pi$ is an irreducible subquotient of $\cInd_{KZ}^G \sigma/ f(T) \cInd_{KZ}^G \sigma$.
\end{enumerate}
\end{theorem}
\begin{proof}
We begin by proving part~(1).
To ease notation,
write $\pi \coloneq  \cInd_{KZ}^G \sigma%
$, 
and $f := f(T)$.
If~$f$ is a linear polynomial, then~$\pi/f\pi$ is absolutely irreducible by Theorem~\ref{classification}.
Accordingly, we assume that~$\deg(f) > 1$. 
Let~$\pi'$ be a nonzero $k[G]_\zeta$-submodule of~$\pi/f\pi$.
We claim that~$\pi'$ contains the image of a nonzero element of $\Hom_{\kbase[G]}(\pi, \pi/f\pi)$.
Assuming the claim, we conclude the proof of part~(1) as follows. 
By Lemma~\ref{lem:mapstoquotient} we see that $\Hom_{\kbase[G]_\zeta}(\pi, \pi/f\pi) =\Hom_{\kbase[G]_\zeta}(\pi, \pi)/f$. %
Since $k[T]/f k[T] \cong \Hom_{\kbase[G]}(\pi, \pi)/f$ is a field, we obtain that every nonzero element of $\Hom_{\kbase[G]}(\pi, \pi/f\pi)$ is surjective, and so~$\pi' = \pi/f\pi$, which concludes the proof.

We now prove the claim. 
Since~$\pi'$ is nonzero, by Lemma~\ref{quotient by polynomial} it contains the image of a nonzero morphism
\[
\cInd_{KZ}^G(\sigma') \to \pi/f\pi
\]
for some irreducible $\kbase[KZ]_\zeta$-module~$\sigma'$.
Extending scalars to~$\overline{\kbase}$,
and noting that~$f$ is not divisible by~$T$ (by our assumptions that~$f$ is irreducible and~$\deg(f) > 1$), 
we see that 
\[
(\pi/f\pi) \otimes_\kbase \overline{\kbase} \cong \bigoplus_{i=1}^m \cInd_{KZ}^G(\sigma)/(T-\lambda_i)^{n_i}\cInd_{KZ}^G(\sigma)
\]
for some~$\lambda_i \in \bF^\times$ (where possibly $n_i > 1$ if~$k$ is not perfect).
We conclude from Theorem~\ref{classification}
that if~$\sigma' \ne \sigma$ then $\{\sigma,\sigma'\} = \{\Sym^0 \otimes \det^a, \Sym^{p-1} \otimes \det^a\}$ for some~$a$.
However, in this case, we are assuming that~$f(T)$ is coprime to $T^2-\zeta(p)$, and so
$\pi/f\pi$ is a $k[T, 1/(T^2-\zeta(p))]$-module.
Since Lemma~\ref{lem:cind homs}~(3) shows that
\begin{multline}\label{eqn:change of weight isomorphism}
\cInd_{KZ}^G(\Sym^{p-1} \otimes \det{}^s)[1/(T^2-\zeta(p))] \\
\cong \cInd_{KZ}^G(\Sym^0\otimes \det{}^s)[1/(T^2-\zeta(p))],
\end{multline}
we can thus replace~$\sigma'$ by~$\sigma$.
This concludes the proof of part~(1).

Part~(2) is an immediate consequence of Lemma~\ref{lem: quotient giving St and trivial}. %
Given part~(2), part~(3) is an immediate consequence of Lemma~\ref{quotient by polynomial}~(2) and \eqref{eqn:change of weight isomorphism}.
\end{proof}

The following corollaries complete our classification results.

\begin{cor}\label{irreducibleHecke}
Let~$\sigma_1, \sigma_2$ be Serre weights and let~$f_i \in \cH(\sigma_i)$ be irreducible monic polynomials.
Then $\cInd_{KZ}^G(\sigma_1)/f_1$ and $\cInd_{KZ}^G(\sigma_2)/f_2$ have a common irreducible subquotient if and only if one of the following is true:
\begin{enumerate}
\item $\sigma_1 = \sigma_2$ and $f_1 = f_2$, 
\item $\{\sigma_1, \sigma_2\} = \{\Sym^0 \otimes \det^s, \Sym^{p-1} \otimes \det^s\}$ for some~$0 \leq s \leq p-2$, and $f_1 = f_2$,
\item $\{\sigma_1, \sigma_2\} = \{\Sym^r \otimes \det^s, \Sym^{p-1-r} \otimes \det^{r+s}\}, f_ 1= f_2 = T$, for some $0 \leq r \leq p-1, 0 \leq s \leq p-2$.
\end{enumerate}
\end{cor}
\begin{proof}
Assume that the $\cInd_{KZ}^G \sigma_i/f_i$ have a common irreducible subquotient.
Then the same is true after extending scalars to~$\lbar k$, and so there exist roots~$\mu_i$ of~$f_i$ in~$\lbar k$ such that
$\cInd_{KZ}^G(\sigma_1)/(T-\mu_1)$ and~$\cInd_{KZ}^G(\sigma_2)/(T-\mu_2)$
have a common irreducible subquotient.

Assume first that~$\sigma_1 \ne \sigma_2$.
Then Theorem~\ref{classification}, shows that 
either $\{\sigma_1, \sigma_2\} = \{\Sym^r \otimes \det^s, \Sym^{p-1-r} \otimes \det^{r+s}\}$ and~$\mu_1 = \mu_2 = 0$, or
$\{\sigma_1, \sigma_2\} = \{\Sym^0 \otimes \det^s, \Sym^{p-1} \otimes \det^s\}$ and~$\mu_1 = \mu_2$.
Bearing in mind that~$f_1, f_2$ are monic irreducible, we conclude that~(3) holds, respectively~(2) holds.

Finally, if~$\sigma_1 = \sigma_2$ then Theorem~\ref{classification} shows that~$\mu_1 = \mu_2$, hence again~$f_1 = f_2$, and~(1) holds.
This concludes the proof of the forward implication in the corollary.
The converse implication is immediate from~\eqref{eqn:change of weight isomorphism} and~\eqref{eqn:second exception} (the latter being combined with a flat descent to pass from $\overline{k}$ to~$k$).
\end{proof}

\begin{cor}
\label{cor:cofinite length}
Let~$\sigma$ be a Serre weight.
If $\pi'$ is a non-zero subobject of $\cInd_{KZ}^G\sigma$ in~$\cA_A$,
then $(\cInd_{KZ}^G\sigma)/\pi'$ is of finite length.
\end{cor}
\begin{proof}
By Lemma~\ref{quotient by polynomial}, there exists a nonzero polynomial~$f(T)$ such that $f(T) \cInd_{KZ}^G(\sigma) \subset \pi'$.
So there is a surjection
\[
\cInd_{KZ}^G\sigma/f(T)\cInd_{KZ}^G\sigma \to (\cInd_{KZ}^G\sigma)/\pi'
\]
and it suffices to prove that the left-hand side has finite length.
By factoring~$f(T)$, we reduce to
the case that~$f(T)$ is irreducible, in which case  
the corollary follows
from Theorem~\ref{classifyirreducibles}. %
\end{proof}

We will use the following result in the proof of Lemma~\ref{lem:R1Gamma on fg}.
\begin{lemma}
\label{lem:Artinian}
Let $\sigma$ be a Serre weight,
and $g$ be an irreducible element of $\cH(\sigma)$.
Then $(\cInd_{KZ}^G \sigma)[1/g]/(\cInd_{KZ}^G \sigma)$ is an essential extension of
$\frac{1}{g}\cInd_{KZ}^G \sigma/\cInd_{KZ}^G \sigma.$ 
\end{lemma}
\begin{proof}%
  Write $\pi\coloneq \cInd_{KZ}^G \sigma$. To prove that $\frac{1}{g}\pi/\pi \to \pi[1/g]/\pi$ is essential it is
  enough to show that for each $n\ge 1$, \[\frac{1}{g}\pi/\pi\to
    \frac{1}{g^n}\pi/\pi\] is an essential extension. 
Taking into account the isomorphism \[\frac{1}{g^n}\pi/\pi\isoto
  \pi/g^n\pi\]given by multiplication by~$g^n$, it is equivalent to show that $g^{n-1}\pi/g^n\pi \to \pi/g^n\pi$ is an essential extension.
This means that, for every irreducible submodule $\Theta \subset \pi/g^n\pi$, we need to prove that $\Theta \subset g^{n-1}\pi/g^n \pi$.

Since~$\Theta$ is irreducible, it is the image of a morphism $\alpha: \cInd_{KZ}^G(\sigma') \to \pi/g^n\pi$ for some Serre weight~$\sigma'$.
Let~$\pi' = \cInd_{KZ}^G(\sigma')$.
Since $\Hom_{\cA_A}(\pi', \pi/g\pi) \ne 0$, and~$g$ is irreducible, Corollary~\ref{irreducibleHecke} shows that 
we can choose~$\sigma'$ in such a way that $\Hom_{\cA_A}(\pi', \pi) \ne 0$.
(In fact, except when $\sigma \in \{\Sym^0 \otimes \det^s, \Sym^{p-1} \otimes \det^s\}$ and~$g$ divides $T^2-\zeta(p)$, we can choose~$\sigma' = \sigma$.)
We now apply Lemma~\ref{lem:mapstoquotient}.
We find that 
\[
\Hom_{\cA_A}(\pi', \pi/g^n\pi) = \Hom_{\cA_A}(\pi'/g^n\pi', \pi/g^n\pi).
\]
hence~$\alpha$ factors through~$\pi'/g^n \pi'$.
Furthermore, we see that the two  actions of $k[T]$ on  $\Hom_{\cA_A}(\pi'/g^n\pi', \pi/g^n\pi)$ coincide, hence
\[
\Hom_{\cA_A}(\pi'/g\pi', \pi/g^n\pi) = \Hom_{\cA_A}(\pi'/g^n\pi', g^{n-1}\pi/g^n\pi),
\]
because both sides coincide with the $g(T)$-torsion subspace of $\Hom_{\cA_A}(\pi'/g^n\pi', \pi/g^n\pi)$.
Since we have to prove that $\Theta = \operatorname{image}(\alpha)$ is contained in $g^{n-1}\pi/g^n\pi$, 
it now suffices to prove that $\alpha : \pi' \to \Theta$ factors through $\pi'/g\pi'$. %
We will do this by proving that
\begin{equation}\label{eqn:to prove equal for essential extension}
\Hom_{\cA_A}(\pi', \Theta) = \Hom_{\cA_A}(\pi'/g \pi', \Theta).
\end{equation}
Since~$\Theta$ is an irreducible quotient of $\pi'/g^n\pi'$, it is also a quotient of $\pi'/g\pi'$.
If~$\pi' / g \pi'$ is irreducible, it is therefore isomorphic to~$\Theta$, and so~\eqref{eqn:to prove equal for essential extension}
is a consequence of Lemma~\ref{lem:mapstoquotient}~(2).
If~$\pi'/g\pi'$ is reducible, then we are in the case described in Theorem~\ref{classifyirreducibles}~(2). 
After extending scalars to a quadratic extension of~$k$, we can therefore assume that~$\Theta$ is a character, or a twist of the Steinberg representation.
In either case
the left-hand side of~\eqref{eqn:to prove equal for essential extension} has dimension~$1$, which immediately implies that~\eqref{eqn:to prove equal for essential extension}
is an equality, as desired.
\end{proof}

We will also find the following refinement of Lemma~\ref{lem:Artinian} useful.
\begin{lem}
  \label{lem: uniserial}Let $\sigma$ be a Serre weight,
and $g \in \cH(\sigma)$ be a non-zero irreducible element.
Then $(\cInd_{KZ}^G \sigma)[1/g]/(\cInd_{KZ}^G \sigma)$ is uniserial, hence Artinian.
\end{lem}
\begin{proof}
Except in the case when $\sigma$ 
is a twist of $\Sym^0$ or $\Sym^{p-1}$ and $g$ divides $T^2-\zeta(p),$
we know by Theorem~\ref{classifyirreducibles} that $\cInd_{KZ}^G \sigma/ g(\cInd_{KZ}^G \sigma)$
is irreducible, in which case the lemma is an easy consequence of Lemma~\ref{lem:Artinian}.
If we are in one of the remaining cases, we may extend scalars to a splitting field of~$g$,
and then twist so that $\zeta(p) = 1$, $\sigma = \Sym^0$
or $\Sym^{p-1}$ and $g = T -1$,
and so we assume this from now on.

It suffices to prove that $\pi_n = \cInd_{KZ}^G \sigma/(T-1)^n\cInd_{KZ}^G(\sigma)$ is uniserial for every~$n \geq 1$, which we do by induction on~$n$.
The case~$n = 1$ is true because $\Hom_{k[KZ]}(\sigma, \pi_1)$ is one-dimensional, and the image of every nonzero element generates~$\pi_1$ as a $G$-module.
For the case~$n = 2$ we localize
each of $\cInd_{KZ}^G \Sym^0$  and $\cInd_{KZ}^G\Sym^{p-1}$
at the prime ideal $(T-1)$ of $k[T]$,
so as to obtain by Lemma~\ref{lem:cind homs} a chain
of inclusions
\begin{multline*}
(\cInd_{KZ}^G \Sym^0)_{(T-1)}\supset(\cInd_{KZ}^G \Sym^{p-1})_{(T-1)}
\\
\supset
(T-1)(\cInd_{KZ}^G \Sym^0)_{(T-1)}\supset(T-1)(\cInd_{KZ}^G \Sym^{p-1})_{(T-1)}
\end{multline*}
with successive quotients equal to $1_G$ (the trivial representation of~$G$),
$\St$, and $1_G$ again.
This together with the case~$n = 1$ for both $\sigma \in \{\Sym^0, \Sym^{p-1}\}$ implies the case~$n = 2$.
In general, we have a commutative diagram with exact rows
\[
\xymatrix{
0 \ar[r] & \pi_1 \ar[r] \ar@{=}[d] & \pi_2 \ar[r] \ar[d] & \pi_1 \ar[r]\ar[d] & 0\\
0 \ar[r] & \pi_1 \ar[r] & \pi_{n+1} \ar[r] & \pi_n \ar[r] & 0} 
\]
such that the first row is the pullback of the second by the injection $\pi_1 \to \pi_n$.
Since $\pi_n$ is uniserial by inductive assumption, it suffices to prove that $\soc_G^i(\pi_{n+1}) = \soc_G^i(\pi_1)$ for~$i = 0, 1$.
(Recall that~$\soc_G^1(\pi_{n+1})$ is the preimage of $\soc_G(\pi_{n+1}/\soc_G(\pi_{n+1}))$ in~$\pi_{n+1}$, and $\soc_G^0 = \soc_G$.)

For $i = 0, 1$, the image of $\soc^i_G(\pi_{n+1})$ in $\pi_n$ is necessarily contained in $\soc_G^1(\pi_n)$.
Since~$\pi_n$ is uniserial, $\soc_G^1(\pi_n) = \pi_1 \subset \pi_n$.
Hence $\soc_G^i(\pi_{n+1}) \subset \pi_2$, and so it coincides with~$\soc^i_G(\pi_2) = \soc^i_G(\pi_1)$, because~$\pi_2$ is uniserial.
This concludes the proof.
\end{proof}

Finally, we use the previous results to classify the $G$-submodules of compact inductions of Serre weights.

\begin{prop}
\label{cind subobjects}\leavevmode
\begin{enumerate}
\item
If  $\sigma$ is not a twist of $\Sym^0$ or~$\Sym^{p-1}$,
then any subobject of $\cInd_{KZ}^G \sigma$
is of the form $f(T)\cInd_{KZ}^G \sigma$ for   some $f(T)\in k[T]$. 
\item
If $\sigma$  is a twist  of $\Sym^0$ or~$\Sym^{p-1}$,
then any subobject $\pi$ of $\cInd_{KZ}^G\sigma$ 
satisfies inclusions
$$(T^2-\zeta(p))f(T)\cInd_{KZ}^G\sigma \subseteq  \pi\subseteq f(T)\cInd_{KZ}^G\sigma$$
for some $f(T)\in k[T].$
Furthermore, $\pi$ is a $k[T]$-submodule of $\cInd_{KZ}^G\sigma$.
\end{enumerate}
\end{prop}
\begin{proof}
Let~$\pi$ be a nonzero proper subobject of~$\cInd_{KZ}^G(\sigma)$.
By Lemma~\ref{quotient by polynomial}, there exists a monic polynomial~$f(T)$ such that $\pi$ contains $f(T)\cInd_{KZ}^G(\sigma)$.
Factoring~$f(T)$ we obtain that~$\pi$ is the preimage of a $G$-submodule~$\lbar \pi$ of
\[
\bigoplus_{i=1}^m \cInd_{KZ}^G\sigma/g_i(T)^{n_i}\cInd_{KZ}^G\sigma
\]
for pairwise different monic irreducible polynomials~$g_i$, and positive integers~$m, n_i$.
By Corollary~\ref{irreducibleHecke}, the sets of isomorphism classes of irreducible subquotients of the representations~$\cInd_{KZ}^G\sigma/g_i(T)$ are pairwise disjoint. 
So there exist subobjects~$\lbar \pi_i \subset \cInd_{KZ}^G\sigma/g_i(T)^{n_i}$ such that $\lbar \pi = \bigoplus_i \lbar \pi_i$.
Then the proposition follows from the facts that $\cInd_{KZ}^G\sigma/g_i(T)^{n_i}$ is uniserial, by Lemma~\ref{lem: uniserial}, and 
$\cInd_{KZ}^G\sigma/g_i(T)$ is irreducible, by Theorem~\ref{classifyirreducibles}, except when~$\sigma$ is a twist of~$\Sym^0$ or~$\Sym^{p-1}$ and~$g_i(T)$ divides~$T^2-\zeta(p)$, in which case it has length two.
\end{proof}

\begin{rem}
When~$\sigma$ is a twist of~$\Sym^0$ we can still classify all submodules in terms of morphisms from compact inductions, but we may have to work with another compact-mod-centre subgroup.
In fact, we have an exact sequence~\eqref{tree}
\begin{gather*}
0 \to \cInd_{N}^G(\delta) \to \cInd_{KZ}^G \Sym^0 \to 1_G \to 0
\end{gather*}
where~$N$ is the normalizer in~$G$ of the upper-triangular Iwahori subgroup~$I$, and~$\delta$ is the nontrivial character of~$N/IZ$.
However, this submodule of $\cInd_{KZ}^G \Sym^0$ is not the image of any morphism $\cInd_{KZ}^G(\sigma') \to \cInd_{KZ}^G(\Sym^0)$.
Indeed, it properly contains~$(T-1)\cInd_{KZ}^G(\Sym^0)$, and so it cannot be the image of any morphism $\cInd_{KZ}^G(\Sym^0) \to \cInd_{KZ}^G(\Sym^0)$.
On the other hand, by Lemma~\ref{lem:cind homs}, every morphism~$\gamma: \cInd_{KZ}^G(\Sym^{p-1}) \to \cInd_{KZ}^G(\Sym^0)$ is a $k[T]$-multiple of a morphism~$\beta$ whose cokernel has length two.
Hence the cokernel of~$\gamma$ cannot be isomorphic to the trivial representation of~$G$.
\end{rem}

\subsection{Categories of representations}%

We now establish some basic properties of our categories of representations; we refer to Appendix~\ref{sec: cat theory background} 
for some background material in category theory (e.g.\ the notion of a locally Noetherian category).

We continue to work with coefficients in a complete Noetherian local $\cO$-algebra~$A$, with residue field~$k$ of characteristic~$p$.
We say that a representation of $G$ on an $A$-module 
is finitely generated if it is finitely generated as an $A[G]$-module.
We note that a smooth representation of $G$ on an $A$-module, admitting a central character, is finitely 
generated if and only if it is a quotient of a compactly induced
representation $\cInd_{KZ}^G V$, for some finite length $A$-module $V$ 
endowed with a smooth action of~$KZ$.
Any object of $\cA_A$ is thus a filtered colimit of quotients of objects of the form $\cInd_{KZ}^G V$,
where $V$ is a finite length $A$-module endowed with a smooth
$KZ$-action. %

We frequently use the following result. For example, together with the paragraph above, it implies that any object of $\cA_A$ is a (not necessarily filtered) colimit of objects of the form $\cInd_{KZ}^G V$, for~$V$ of finite $A$-length.
\begin{thm}%
\label{thm:noetherian}
If~$A$ is a complete Noetherian local $\cO$-algebra,
and $\zeta: Z \to \cO^\times \subset A^\times$ is a continuous character, then
any finitely generated 
object of~$\cA_A$ is Noetherian, and every finitely generated object of~$\cC_A$ has finite length
{\em (}and so is Noetherian{\em )}.
\end{thm}
\begin{proof}
Let~$V$ be a finitely generated object of~$\cC_A$, i.e.\ a smooth $A[KZ]_\zeta$-module which is locally $\fm_A$-torsion and finitely generated over~$A[KZ]_\zeta$.
Since~$KZ$ is compact module centre, $V$ is therefore finitely generated over~$A$, and so it is annihilated by~$\fm_A^m$ for some~$m$.
Since~$A/\fm_A^m$ is Artinian, we conclude that~$V$ has finite $A$-length, and so it is a finite length object of~$\cC_A$, as desired.

Now let~$\pi$ be a finitely generated smooth $A[G]_\zeta$-module.
By definition, $\pi$ is a quotient of $\cInd_{KZ}^G(V)$ for some
smooth $A[KZ]_\zeta$-module~$V$ of finite type over~$A$.
As proved in the previous paragraph, $V$ has finite length in~$\cC_A$.
Since
quotients, subobjects and extensions of Noetherian objects are Noetherian, it
thus suffices to prove %
that~$\cInd_{KZ}^G(V)$ is Noetherian,
and it further suffices to prove it in the case that
~$V$ is an irreducible object of~$\cC_A$. %
In this case, the result follows from Proposition~\ref{cind subobjects}, which shows that every $\cA_A$-subobject of $\cInd_{KZ}^G(V)$ is finitely generated over~$A[G]_\zeta$.
This concludes the proof. %
\end{proof}

\begin{remark}\label{rem: not noetherian in general}
Theorem~\ref{thm:noetherian} does not hold in general for non-abelian $p$-adic reductive groups.  
For example, by~\cite{MR3365778, MR4280498} it fails for $\GL_2(F)$ for any finite extension~$F/\bQ_p$ different from~$\bQ_p$. 
Presumably, it fails any time we're outside 
the case of $\GL_2(\Q_p)$, or some essentially equivalent context.
\end{remark}

In contrast with Remark~\ref{rem: not noetherian in general}, the following lemma (and its proof) goes through essentially
unchanged for general $p$-adic reductive groups, and even for general locally profinite groups. 
\begin{lem}\label{lem:Grothendieckcategory}%
If~$A$ is a complete Noetherian local $\cO$-algebra, then $\cA_{A}$ and~$\cC_A$ are Grothendieck categories. 
\end{lem}
\begin{proof}
We begin by proving the lemma for~$\cA_A$.
Since the category of $A$-modules is a Grothendieck category, and
since the formation of colimits in $\cA_A$ is compatible with
the formation of colimits on the underlying~$A$-modules,
it is
enough to show that $\cA_A$ has a generator. %
To form a generator of~$\cA_A$,
choose a cofinite %
system~$K_n \subset K$ of open normal subgroups
of the compact open subgroup~$K$ of~$G$.
For any~$m \geq 1$,
the composite $Z \buildrel \zeta \over \longrightarrow A^{\times} \to
(A/\fm_A^m)^{\times}$ is a continuous
character with values in a discrete group, and hence factors through
the quotient $Z/(K_n \cap Z)$ if $n$ is sufficiently large (depending on $m$),
i.e.\ if $n \geq N(m)$ for some $N(m) \geq 1$.
Via the isomorphism $K_n Z/ K_n \iso Z/(K_n \cap Z)$, 
the preceding composite induces a continuous character
$K_nZ \to K_n Z/K_n \to (A/\fm_A^m)^{\times}$,
which (by abuse of notation) we again denote simply by~$\zeta$.

Now, for $n \geq N(m)$, consider the compact induction $\cInd_{K_nZ}^G (A/\fm_A^m)$, where~$K_nZ$ acts on~$A/\fm_A^m$ via the character~$\zeta$.
If~$V$ is an object of~$\cA_A$, then evaluation on the coset
$1 \mod \fm_A^m \in (A/\fm_A^m)\subset \cInd_{K_n Z}^G (A/\fm_A^m)$
induces a natural identification
$$V^{K_n}[\fm_A^m] = \Hom_{\cA_A}\bigl(\cInd_{K_n Z}^G (A/\fm_A^m), V\bigr)$$ 
Since~$V = \varinjlim_{m, n \geq N(m)} V^{K_n}[\fm_A^m]$,
by definition of~$\cA_A$, 
we conclude that
$$\bigoplus_{m, n\geq N(m)} \cInd_{K_nZ}^G (A/\fm_A^m)$$
is a generator of~$\cA_A$.

This concludes the proof for~$\cA_A$.
A very similar argument works for~$\cC_A$,
and shows that the modules $\cInd_{K_n Z}^{KZ}(A/\fm_A^m)$ are generators. %
\end{proof}

\begin{cor}%
\label{cor:noetherian}%

If~$A$ is a complete Noetherian local $\cO$-algebra,
the abelian categories $\cA_{A}$ and~$\cC_A$ are locally Noetherian. The Noetherian
objects of~$\cA_A$, resp.\ $\cC_A$, are precisely the finitely generated objects. %
\end{cor}
\begin{proof}
By Lemma~\ref{lem:Grothendieckcategory}, filtered colimits are exact
in~$\cA_A$; %
so in order to see that $\cA_{A}$ is locally Noetherian, we only need to exhibit a set of Noetherian generators of~$\cA_{A}$.
By Theorem~\ref{thm:noetherian}, the compact induction $\cInd_{KZ}^G(V)$ is Noetherian whenever~$V$ is a finite length object of~$\cC_A$. %
Hence the system of generators 
\[
\cInd_{KZ}^G\left ( \cInd_{K_nZ}^{KZ} (A/\fm_A^m) \right ) 
\]
constructed in the proof of Lemma~\ref{lem:Grothendieckcategory} consists of Noetherian objects,
since $\cInd_{K_nZ}^{KZ} (A/\fm_A^m)$ is a finite length object of~$\cC_A$.
This concludes the proof that~$\cA_A$ is locally Noetherian.
The proof for~$\cC_A$ is the same, using the fact that $\cInd_{K_nZ}^{KZ} (A/\fm_A^m)$ is Noetherian in~$\cC_A$ (being of finite length), and in fact shows that~$\cC_A$ 
is locally finite. %

Finally, by Theorem~\ref{thm:noetherian}, the finitely generated objects
of~$\cA_A$, resp.\ $\cC_A$ are Noetherian, and the converse is immediate from the definition of a Noetherian object.
\end{proof}

We now discuss certain full subcategories of $\cA_A$, %
obtained by imposing finiteness condition on its objects. %
Recall~\cite[Defn.\ 2.2.15]{MR2667882} that a
representation~$\pi\in \cA_A$ is {\em locally admissible}, resp.\! {\em locally finite}, if for every vector~$v \in \pi$ the $G$-representation~$\langle G \cdot v \rangle$ generated by~$v$ is admissible, resp.\! has finite length. 
On the other hand, Corollary~\ref{cor:noetherian} and Proposition~\ref{prop: properties of locally Noetherian categories} imply that an object of~$\cA_A$
is finitely generated if and only if it is Noetherian, if and only if it is a compact object of~$\cA_A$.

\begin{defn}\label{defn: A^ladm and A^fg}
We let $\cA_A^{\ladm}$, resp.\ $\cA_A^{\fg}$, resp.\ $\cA_A^{\fl}$ denote
the full subcategory of~$\cA_A$ consisting of locally admissible objects, resp.\ finitely generated objects
(equivalently, Noetherian objects, or equivalently, compact objects), 
resp.\ objects of finite length.
\end{defn}

As for~$\cA_A$, when~$A = \cO$ we will often drop the symbol~$A$ from the notation for~$\cA^{\ladm}_A$, $\cA^{\fg}_A$, and $\cA_{A}^{\fl}$.
When the residue field of~$A$ is finite, it follows from~\cite[Thm.\ 2.3.8]{MR2667882} that an object of~$\cA_A$ is locally finite if and only if it is locally admissible. 
For completeness, we provide a proof of this fact in the generality of this paper.

\begin{lemma}\label{locally admissible and locally finite}
Let~$\pi \in \cA_A$.
Then~$\pi$ is locally admissible if and only if it is locally finite.
\end{lemma}
\begin{proof}
It suffices to prove that if~$\pi$ is a finitely generated 
object of~$\cA_A$,
then~$\pi$ is admissible if and only if it has finite length.
Since irreducible objects of~$\cA_A$ are classified in Theorem~\ref{classifyirreducibles}, and are all admissible by inspection, 
there remains to prove that if~$\pi$ is admissible and finitely generated then it has finite length.
Since~$\pi$ is a quotient of~$\cInd_{KZ}^G(V)$ for some finite length 
object of~$\cC_A$,
this is a consequence of the fact that every admissible quotient of~$\cInd_{KZ}^G(\sigma)$ has finite length, which is itself a consequence of Corollary~\ref{cor:cofinite length} and the fact that~$\cInd_{KZ}^G(\sigma)$ is not itself admissible (e.g.\ because $\Hom_{KZ}(\sigma,\cInd_{KZ}^G(\sigma))=\cH(\sigma)$ is infinite-dimensional).
\end{proof}

\subsection{Generalities about $\Ext$ groups in $\cA_{A}$}\label{subsec:
  generalities on Exts}

The inclusion of $\cA_{A}$ into the category of all $A[G]_\zeta$-modules
is exact and fully faithful, and admits a right
adjoint, given by passage to the submodule consisting of smooth vectors annihilated by
some power of~$\fm_A$.  
This right adjoint preserves injectives, and so, since the
latter category admits enough injectives (being the category of modules over
a ring), so does the category~$\cA_{A}$.  Of course, since $\cA_{A}$
is a Grothendieck category, it also admits enough injectives for abstract reasons.

Recall that the functor $\cInd_{KZ}^G(-)$ is an exact functor from~$\cC_A$ to $\cA_{A}$.
It is left adjoint to the forgetful functor (i.e.\ restriction) from $\cA_{A}$ to $\cC_{A}$, and so this latter functor preserves injectives.
Thus we have natural isomorphisms
\begin{equation}
\label{eqn:ext adjunction}
\Ext^i_{\cA_{A}}(\cInd_{KZ}^G V, \pi) \iso \Ext^i_{\cC_{A}}(V,\pi)
\end{equation}
whenever $V$ is a object of $\cC_{A}$ and $\pi$ is an object
of~$\cA_{A}$. We will use~\eqref{eqn:ext adjunction} to reduce some
questions about $\Ext$ groups in $\cA_A$ to $\Ext$ groups in
$\cC_A$. We begin with the following basic result about these
$\Ext$ groups, whose statement and proof extend in an obvious way to
arbitrary compact $p$-adic analytic
    groups.

\begin{lemma}\label{lem:extKZ}%
Let~$A$ be a complete Noetherian local $\cO$-algebra and
let~$V, W$ be objects of~$\cC_A$ which are finitely generated over~$A$.
Then
\[
\Ext^i_{\cC_{A}}(V,W)
\] 
is finitely generated over~$A$.
\end{lemma}
\begin{proof}
By induction on~$\operatorname{length}_{\cC_A}(V)$, it suffices to prove the lemma when~$V$ is irreducible.
We begin by proving that $\Ext^i_{\cC_{k}}(V,W)$ is finite-dimensional over~$k$ whenever $W$ is an object of~$\cC_k$ of finite dimension over~$k$.
By induction on~$\operatorname{length}_{\cC_A}(W)$, it suffices to prove this claim when~$W$ is irreducible.
Let~$K_1$ be the first congruence subgroup of~$K$, which acts trivially on~$V$ and~$W$.
Since
\[
\Hom_{\cC_{\kbase}}(V, -) = \Hom_{\kbase[KZ/K_1]_\zeta}(V, (-)^{K_1})
\]
there is a spectral sequence
\begin{equation}\label{changecoefficientsI}
\Ext^i_{\kbase[KZ/K_1]_\zeta}(V, H^j(K_1, W)) \Rightarrow \Ext^{i+j}_{\cC_{\kbase}}(V, W).
\end{equation}
The group~$K_1$ acts trivially on~$W$, which is a finite-dimensional $\kbase$-vector space, and so $H^j(K_1, W)$ is also finite-dimensional for all~$j$: 
see for example~\cite[Cor.\ 4.2.5, Thm.\ 5.1.2]{MR1765127} for a discussion of this fact, which follows from Lazard's results on the structure of~$k[[K_1]]$.
Finally, since $k[KZ/K_1]_\zeta$ is a finite $k$-algebra, we see that the groups on the $E_2$-page are also finite-dimensional, completing the proof of the claim.

Now let~$W$ be an $A$-finite object of~$\cC_A$.
Since~$A$ is Noetherian we know that $\Ext^j_A(k, W)$ is a finite $k$-vector space for all~$j$, hence the lemma follows from the claim proved before and the spectral sequence~\eqref{changecoefficients2} to follow.
\end{proof}

\begin{lemma}%
Let~$(A, \fm,  k)$ be a complete Noetherian local $\cO$-algebra, and let $V \in \cC_k, W \in \cC_A$.
Then there is a spectral sequence 
\begin{equation}\label{changecoefficients2}
\Ext^i_{\cC_k}(V, \Ext^j_A(k, W)) \Rightarrow \Ext^{i+j}_{\cC_A}(V, W). 
\end{equation}
Similarly, if $V \in \cA_k, W \in \cA_A$,
then there is a spectral sequence 
\begin{equation}\label{changecoefficients2-A-version}
\Ext^i_{\cA_k}(V, \Ext^j_A(k, W)) \Rightarrow \Ext^{i+j}_{\cA_A}(V, W). 
\end{equation}

\end{lemma}
\begin{proof}
Note that
\[
\Hom_{\cC_A}(V, -) = \Hom_{\cC_k}(V, (-)[\fm])
\]
as functors from~$\cC_A$ to $k$-vector spaces.
If~$I$ is an injective object of~$\cC_A$, then $I[\fm]$ is injective in~$\cC_k$,
since $(\text{--})[\fm]$ is right  adjoint to  the inclusion of  $\cC_k$  into~$\cC_A$,
which is exact.
Furthermore,
$I$ is also injective as an $A$-module (as we will show below),
and so the right derived functors  
of $(\text{--})[\fm]$ on~$\cC_A$ coincide
with (the composite of the forgetful functor to $A$-modules and) the right
derived functors 
of $(\text{--})[\fm]$ on  the category  of~$A$-modules.

To see the $A$-module injectivity of~$I$,
note that if $\mathfrak{a} \to A$ is the inclusion of an ideal, and $\varphi: \mathfrak{a} \to I$ is $A$-linear, 
then~$\varphi(\mathfrak{a})$ is contained in~$I^{K_0}$ for some open normal subgroup~$K_0\subset K$, since~$I$ is finitely generated ($A$~being Noetherian). 
Furthermore, $\varphi$ factors through $\mathfrak{a}/\fm_A^n\mathfrak{a}$ for some~$n>0$, since~$I$ is smooth.
By the Artin--Rees lemma there exists~$m>0$ such that~$\fa\cap \fm_A^m \subset \fm_A^n \fa$.
Viewing~$\mathfrak{a}/\fa \cap \fm_A^m$ and~$A/\fm_A^m$ as smooth $K_0 Z$-modules by letting~$K_0 Z$ act by~$\zeta$, and using the fact that~$I$ is an injective object of~$\cC_A$, 
and that restriction to an open subgroup preserves injectives,
it follows that~$\varphi$ extends to a map $A \to I$.

Putting together the observations of the preceding paragraphs
gives a Grothendieck spectral sequence as in the statement of the lemma.
The same proof works in the case  of~$\cA_A$.
\end{proof}

\begin{lemma}
\label{lem:ext and colimits} %
Let~$A$ be a complete Noetherian local $\cO$-algebra.
If $\pi$ is a finitely generated object of $\cA_{A}$,
then each $\Ext^i_{\cA_{A}}(\pi,\text{--})$ 
commutes with filtered colimits.
The same statement is true with~$\cA_A$ replaced by~$\cC_A$.
\end{lemma}
\begin{proof}
  This follows from Corollary~\ref{cor:noetherian} and Proposition~\ref{prop: properties of locally Noetherian categories}~\eqref{item: Exts commute with colimits in the right hand
      variable}.
\end{proof}

\begin{lemma}
\label{lem:countable dim}
Let~$A$ be a complete Noetherian local $\cO$-algebra.
If $\pi$ and $\pi'$ are objects of $\cA_{A}$, with $\pi$ being finitely generated
and $\pi'$ being countably generated, then each
$\Ext^i_{\cA_{A}}(\pi,\pi')$ is a countably generated $A$-module.
\end{lemma}
\begin{proof}
Since $\pi$ is finitely generated, and thus finitely
presented by Theorem~\ref{thm:noetherian},
a standard dimension-shifting
argument (using the fact that since~$A$ is Noetherian, every submodule of a countably generated $A$-module is countably generated) reduces us to checking the claim in the case when
$\pi = \cInd_{KZ}^G V,$ for some finitely generated $A$-module $V$ endowed with
a smooth $KZ$-action.  We then consider the isomorphisms
$$\Ext^i_{\cA_{A}}(\cInd_{KZ}^G V, \pi') \iso 
\Ext^i_{\cC_{A}}(V, \pi')
\iso \varinjlim_W \Ext^i_{\cC_{A}}(V,W),$$
where $W$ runs over the finitely generated $A[KZ]$-submodules of~$\pi'$.
(The first isomorphism is~\eqref{eqn:ext adjunction}, and the last isomorphism is Lemma~\ref{lem:ext and colimits}.)
The directed set of such $W$ contains a countable cofinal subset,
since $\pi'$ is countably generated, and thus the lemma follows from the fact
that each
$\Ext^i_{\cC_{A}}(V,W)$
is finitely generated over~$A$, by Lemma~\ref{lem:extKZ}. %
\end{proof}
We will sometimes make use of the following comparison (due to Pa\v{s}k\={u}nas when~$A = \cO$) between Ext groups in
the locally admissible and smooth categories.

\begin{lem}\label{lem: smooth Ext equals admissible Ext}%
If~$A$ is a complete Noetherian local $\cO$-algebra, and ~$\pi,\pi'$
  are objects of~$\cA_A^{\ladm}$, then $\Ext^i_{\cA_A^{\ladm}}(\pi,\pi')=\Ext^i_{\cA_A}(\pi,\pi')$.  
\end{lem}
\begin{proof}
It suffices to prove that $\cA^{\ladm}_A \to \cA_A$ preserves injective objects; in the case $A=\cO$ this is 
   ~\cite[ Cor.\ 5.18]{MR3150248}.
Examining the proof of that result, we see that to prove it in general,
it suffices to treat the case that~ $A=k$ (a general field of characteristic~$p$). %
When~$k$ is a finite field, this is \cite[Prop.\ 5.16]{MR3150248}.
Examining the proof of that result next, and bearing in mind Corollary~\ref{cor:cofinite length}, we see that for general~$k$ it suffices to prove that 
$\Res^G_{KZ} : \cA^{\ladm}_k \to \cC_k$ preserve injectives.
When~$k$ is finite this is~\cite[Cor.\ 3.10]{MR2667892}, and the proof given there %
only uses the finiteness of~$k$ in the proof of~\cite[Lem.\ 3.3]{MR2667892}; 
that result remains true in our context, because the statement descends to~$\Fp$ and then follows in general by extension of scalars along~$\bF_p \to k$.
\end{proof}

A different (more elementary) comparison of Ext groups occurs if we consider $\cA^{\fg}$ inside~$\cA$.
The former category does  not have enough injectives, but (as noted
in~\ref{subsubsec:fg cat})
we can define $\Ext^i$ in
$\cA^{\fg}$ via Yoneda extensions. 
We then have the following result.

\begin{lemma}
\label{lem:Yoneda}
Let~$A$ be a complete Noetherian local $\cO$-algebra.
If $\pi$ and $\pi'$ are objects of $\cA_A^{\fg}$, 
then
$\Ext^i_{\cA_A^{\fg}}(\pi,\pi') \iso  \Ext^i_{\cA_A}(\pi,\pi')$.
\end{lemma}
\begin{proof}
  This follows from Corollary~\ref{cor:noetherian} and
Lemma~\ref{lem:f.g. Ext}.
\end{proof}

We next establish some base-change results about Ext groups.  

\begin{lemma}%
\label{lem:injective scalar extension}
If $A$ is a finite-dimensional associative algebra over a field~$\kbase$,
if $I$ is an injective $A$-module, and if $\lbase$ is any extension
of~$\kbase$, then $\lbase\otimes_{\kbase} I$ is an injective $\lbase\otimes_{\kbase} A$-module.
\end{lemma}
\begin{proof}
We begin with the case that~$I$ is finite-dimensional over~$\kbase$.
It follows from Baer's criterion that $I$ is injective if and only if $I^\vee \coloneq \Hom_\kbase(I, \kbase)$
is projective over $A^{\op}$.  
Similarly $\lbase\otimes_{\kbase} I$ is injective
over $\lbase\otimes_{\kbase} A$ if and only if $(\lbase\otimes_{\kbase} I)^\vee \cong \lbase\otimes_{\kbase} I^{\vee}$
is projective over $\lbase\otimes_{\kbase} A^{\op}$.   Thus it suffices to prove 
the analogue of the lemma for projective modules, and thus for free modules (using
the characterization of projective modules as direct summands of free modules).
The case of free modules is clear, and thus the lemma is proved in this case.

Since $A$ and $\lbase\otimes_{\kbase} A$ are Noetherian, the property of being
injective over these rings is preserved under the formation of filtered 
colimits.  
Since $A^{\op}$ is finite dimensional, we know that every simple $A^\op$-module has a finite-dimensional projective envelope.
Thus any simple $A$-module has a finite-dimensional injective envelope.
It follows that any injective $A$-module is a filtered colimit of finite dimensional
injective $A$-modules, since it contains an injective envelope of any finite-dimensional submodule.
Hence the lemma follows from the finite-dimensional case.\end{proof}

\begin{lemma}\label{injective via subgroups}
Let~$I$ be a smooth $k[KZ]_\zeta$-module.
Then~$I$ is injective in~$\cC_k$ if and only if $I^H$ is $k[KZ/H]_\zeta$-injective for every open normal subgroup~$H \subset K$.
\end{lemma}
\begin{proof}
The forward direction is immediate, since~$(-)^H$ preserves injectives (being right adjoint to an exact functor).
Assume now that $I^H$ is $k[KZ/H]_\zeta$-injective for all~$H$.
The usual argument for proving Baer's criterion shows that $I$ is injective if and only 
$\Ext^1_{\cC_{\kbase}}(V,I) = 0$ for each finite-dimensional object
$V$ of $\cC_{\kbase}$.  
Note that $\Ext^1_{\cC_{\kbase}}(V,I) = \varinjlim_H \Ext^1_{\cC_{\kbase}}(V, I^H)$ by Lemma~\ref{lem:ext and colimits}, %
and that $\Ext^1_{\cC_{\kbase}}(V, I^H) = \Ext^1_{k[KZ/H^p]}(V, I^H)$ whenever~$H$ is a uniform pro-$p$ group normal in~$KZ$ and acting trivially on~$V$. 
In fact, if~$E$ is an extension of~$V$ by~$I^H$ then $(h-1)^2V = 0$ for all~$h \in H$, and so $(h^p-1)V = 0$ for all~$h \in H$; and since~$H$ is uniform, 
the set of $p$-th powers of elements of~$H$ is an open subgroup of~$H$.
Since~$I^{H^p}$ is assumed to be injective as a $k[KZ/H^p]$-module with fixed central character, we deduce that the transition map 
$\Ext^1_{\cC_{\kbase}}(V, I^H) \to \Ext^1_{\cC_{\kbase}}(V, I^{H^p})$ is zero. 
Hence the colimit $\Ext^1_{\cC_{\kbase}}(V,I) = 0$.
\end{proof}

\begin{cor}%
\label{cor:compact group ext base-change}
If $k \subseteq l$ is an extension of fields,
and if $V$ is a finite-dimensional object of $\cC_{\kbase}$,
then for any object $W$ of $\cC_{\kbase}$,
the base-change map
$$\lbase\otimes_{\kbase} \Ext^i_{\cC_{\kbase}}(V,W) \to \Ext^i_{\cC_l}(\lbase\otimes_{\kbase} V,
\lbase\otimes_{\kbase} W)$$
is a natural isomorphism.
\end{cor}
\begin{proof}The case~$i = 0$ is true by~\cite[Lemma~5.1]{MR3150248}, since~$V$ is finitely generated over a Noetherian quotient of~$\kbase[KZ]$.
It then suffices to prove %
that $\lbase\otimes_{\kbase} I$ is injective if~$I$ is injective in~$\cC_{\kbase}$. 
The case~$i = 0$ implies that $(\lbase\otimes_{\kbase} I)^H = \lbase\otimes_{\kbase} I^H$.
Lemma~\ref{injective via subgroups} shows that $I^H$ is an injective $k[KZ/H]_\zeta$-module for every open normal subgroup $H \subset K$.
Lemma~\ref{lem:injective scalar
extension} then shows that $\lbase\otimes_{\kbase} I^H$ is an injective $l[KZ/H]_\zeta$-module.
Hence another application of Lemma~\ref{injective via subgroups}
implies that $\lbase\otimes_{\kbase} I$
is an injective object of~$\cC_l$, as claimed.
\end{proof}

\begin{cor}\label{cor:compact group ext base-change II}
If $\cO \subset \cO'$ is a finite unramified extension,
and if $V$ is a finite dimensional object of $\cC_{\bF}$,
then for any object $W$ of $\cC_{\bF}$,
the base-change map
$$\F'\otimes_{\F} \Ext^i_{\cC_{\cO}}(V,W) \to \Ext^i_{\cC_{\cO'}}(\F'\otimes_{\F} V,
\F'\otimes_{\F} W)$$
is a natural isomorphism.
\end{cor}
\begin{proof}
For any $i$  and~$j$,
there  is  a base-change isomorphism 
\begin{multline*}
\F'\otimes_{\F}\Ext^i_{\cC_{\F}}\bigl(V,\Ext^j_{\cO}(\F,W)\bigr)
\iso
\Ext^i_{\cC_{\F'}}\bigl(\F'\otimes_{\F}  V,  \F'\otimes_{\cO}\Ext^j_{\cO}(\F,W)\bigr)
\\
=
\Ext^i_{\cC_{\F'}}\bigl(\F'\otimes_{\F}  V,  \Ext^j_{\cO'}(\F',\F'\otimes_{\F} W)\bigr),
\end{multline*}
the first isomorphism following from Corollary~\ref{cor:compact group ext base-change},
and the equality holding because $\cO'$ is unramified over~$\cO$,
so that the  derived functors $\RHom_{\cO'}(\F', \cO'\otimes_{\cO} \text{--})$
and $\cO'\otimes_{\cO}\RHom_{\cO}(\F,\text{--})$ coincide on the category of  $\cO$-modules;
both are computed by the complex
$$ \cO'\otimes_{\cO} (\text{--})
\buildrel \varpi \cdot  \over \longrightarrow \cO'\otimes_{\cO}(\text{--})
$$
(where $\varpi$ 
is a uniformizer of~$\cO$).

If we now consider
the spectral sequence~\eqref{changecoefficients2}
for each of~$\cO$ and~$\cO'$, 
we find that these base-change isomorphisms abut to the base-change
morphism in the statement of the present corollary, showing that
this morphism is also an isomorphism, as claimed. 
\end{proof}

\begin{prop}
\label{prop:ext and base-change}

If $\lbase$ is an extension of $\kbase$,
and
if $\pi$ and $\pi'$ are objects of $\cA_{\kbase}$, with $\pi$ being finitely generated,
then there is a natural isomorphism
$$\lbase\otimes_{\kbase} \Ext^i_{\cA_{\kbase}}(\pi,\pi') \iso
\Ext^i_{\cA_{l}}(\lbase\otimes_{\kbase}\pi, \lbase\otimes_{\kbase} \pi'),$$
for each $i \geq 0$.

Similarly, if~$\pi$ and~$\pi'$ are objects of~$\cA_\bF$ with~$\pi$ finitely generated, and $\cO \subset \cO'$ is a finite unramified extension, then there is a natural isomorphism
\[
\F'\otimes_{\F} \Ext^i_{\cA_{\cO}}(\pi,\pi') \iso
\Ext^i_{\cA_{\cO'}}(\F'\otimes_{\F}\pi, \F'\otimes_{\F} \pi')
\]
\end{prop}

\begin{remark}
If we were working in the category of all
$\kbase[[G]]_\zeta$-modules 
then this result would be straightforward, since we would be
able to compute the Ext's using a resolution of $\pi$ by
finite rank free $\kbase[[G]]_\zeta$-modules.
It is the fact 
that we are working in the category $\cA_{\kbase}$,
i.e.\ that we have imposed smoothness, which makes the result
less obvious.

We also remark that in the case when~$k$ is finite and $\pi$ is assumed
to be of finite length, the result has been proved by Pa{\v s}k{\= u}nas
\cite[Proposition~5.33]{MR3150248}.
\end{remark}

\begin{proof}[Proof of Proposition~\ref{prop:ext and base-change}]
Using Theorem~\ref{thm:noetherian} and dimension shifting,
one easily reduces to the case when $\pi = \cInd_{KZ}^G V$
for some finitely generated smooth $KZ$-representation~$V$.
Since compact induction is compatible with extension of scalars,
the proposition follows from Corollaries~\ref{cor:compact group ext
  base-change} and~\ref{cor:compact group ext
  base-change II}, together with~\eqref{eqn:ext adjunction}.
\end{proof}

\begin{remark}
Our proof of Proposition~\ref{prop:ext and base-change}
uses as input the fact that finitely generated objects of $\cA_{\kbase}$ and~$\cA_\cO$
are Noetherian (i.e.\ Theorem~\ref{thm:noetherian}). Bearing in mind
Remark~\ref{rem: not noetherian in general}, we see that the argument, and perhaps also the result,
won't extend as stated to more general $p$-adic Lie groups. 
It seems plausible that it should at least hold in general with
``finitely generated'' replaced by ``finitely presented'', and it's likely
that Shotton's results~\cite{MR4106887} about smooth finitely
presented 
$\GL_2(F)$-representations can be used to show this for~$\GL_2(F)$,
for arbitrary $p$-adic fields~$F$. 
(Note that for~$\GL_2(\Qp)$ all finitely generated representations are
finitely presented by Theorem~\ref{thm:noetherian}.)
\end{remark}%

\subsection{Blocks of~$\cA^{\ladm}_A$}\label{subsec:blocks}
By Lemma~\ref{locally admissible and locally finite}, the category $\cA_A^{\ladm}$ is locally finite, unlike~$\cA_A$, and so admits a decomposition into blocks (see  Lemma~\ref{lem: block decomposition}).
When~$A = \cO$, these blocks are studied intensively in~\cite{MR3150248}. %
Recall that,
by definition, a block of~$\cA_A^{\ladm}$ 
is an equivalence class of (isomorphism classes of) irreducible objects
 under the equivalence relation generated by
\begin{equation}\label{definition of blocks}
\pi_1 \sim \pi_2 \text{ if } \Ext^1_{\cA_A^{\ladm}}(\pi_1, \pi_2) \ne 0 \text{ or } \Ext^1_{\cA_A^{\ladm}}(\pi_2, \pi_1) \ne 0.
\end{equation}
We note that it follows from the results of \cite{BarthelLivneDuke} and~\cite{BreuilGL2I} when~$A = \cO$, and 
Theorem~\ref{classifyirreducibles} in general, that the irreducible objects of $\cA_A$ are automatically admissible, 
and hence lie in $\cA_A^{\ladm}.$  Thus we can equally well regard this as an equivalence relation on the irreducible objects of~$\cA_A$.
Similarly, by Lemma~\ref{lem: smooth Ext equals admissible Ext}, we can replace $\Ext^1_{\cA_A^{\ladm}}$ with~$\Ext^1_{\cA_A}$ in~\eqref{definition of blocks} without changing the relation~$\sim$.
The next lemma shows that we can also replace~$\cA_A^{\ladm}$ by~$\cA_k^{\ladm}$
(and hence then also by~$\cA_k$).

\begin{lemma}\label{same blocks}%
Let~$A$ be a complete Noetherian local $\cO$-algebra. The blocks of $\cA^{\ladm}_A$ coincide with the blocks of~$\cA^{\ladm}_k$.
\end{lemma}
\begin{proof}
Let $\pi_1, \pi_2 \in \cA^{\ladm}_k$ be irreducible. 
Since $\Ext^1_{\cA_k}(\pi_1, \pi_2) \to \Ext^1_{\cA_A}(\pi_1, \pi_2)$ is injective, it suffices to prove that if $\Ext^1_{\cA_A}(\pi_1, \pi_2) \ne 0$ then
$\Ext^1_{\cA_k}(\pi_1, \pi_2) \ne 0$ or $\pi_1 \cong \pi_2$.
But if $\pi_1\neq \pi_2$, then evidently any extension of $\pi_1$ by $\pi_2$ is annihilated
by~$\mathfrak m_A$ (since each $\pi_i$, and since they are irreducible by assumption).
Thus such an extension arises from an element of~$\Ext^1_{\cA_k}(\pi_1,\pi_2)$,
as claimed.
\end{proof}

\begin{rem}\label{rem:how to prove that something is a union of blocks}
If~$S$ is a set of isomorphism classes of irreducible objects of~$\cA_k^{\ladm}$, then~$S$ is a union of blocks if and only if, for all
irreducible $\pi \in S, \tau \not \in S$, we have $\Ext^1_{\cA_k^{\ladm}}(\pi, \tau) = 0$ and
$\Ext^1_{\cA_k^{\ladm}}(\tau, \pi) = 0$.
This is an immediate consequence of the fact that the relation ``being in the same block'' is the transitive closure of the relation~$\sim$
defined in~\eqref{definition of blocks}.
\end{rem}

Our goal in the rest of this section is to extend the results of~\cite{MR3150248} to a classification of blocks of~$\cA^{\ladm}_A$. %
We do this in Proposition~\ref{prop: blocks over perfect fields}.
By Lemma~\ref{same blocks}, it suffices to classify the blocks of~$\cA_k^{\ladm}$ when~$A = k$ is a %
field extension of~$\bF$.
To do so, we will make use of Theorem~\ref{classifyirreducibles}, which classifies the irreducible objects of~$\cA_k^{\ladm}$, and Proposition~\ref{prop:ext and base-change}, 
which asserts that~$\Ext^i_{\cA_k}$ commutes with extensions of~$k$ for finitely generated representations.
We begin with the case of an algebraically closed field.

\begin{prop}\label{ordinary parts}
Let~$A = k$ be an algebraically closed field of characteristic~$p$. 
Then the blocks of~$\cA_k^{\ladm}$ are as follows: 
\begin{enumerate}
\item $\fB = \{\pi\}$ for an irreducible supersingular representation~$\pi$,
\item $\fB = \{\Ind_B^G(\chi_1 \otimes \omega^{-1}\chi_2), \Ind_B^G(\chi_2 \otimes \omega^{-1} \chi_1)\}$ for smooth characters~$\chi_1, \chi_2 : \bQ_p^\times \to k^\times$ such that $\chi_1\chi_2^{-1} \ne 1, \omega^{\pm 1}$
and $\chi_1\chi_2\omega^{-1} = \zeta$, %
\item $\fB = \{\Ind_B^G(\chi \otimes \omega^{-1} \chi)\}$ for a smooth character $\chi: \bQ_p^\times \to k^\times$ such that $\chi^2\omega^{-1} = \zeta$, and
\item $\fB = \{\chi \circ \det, \chi \otimes \St, \Ind_B^G(\omega\chi \otimes \omega^{-1} \chi)\}$ for a smooth character $\chi : \bQ_p^\times \to k^\times$ such that $\chi^2 = \zeta$.
\end{enumerate}

\end{prop}
\begin{proof}
Without loss of generality, we can twist and assume that $\zeta(p) = 1$.
By Theorem~\ref{classification}, the sets in the statement of the proposition form a partition of the set of irreducible objects of~$\cA_k^{\ladm}$.
We first prove that the sets~$\fB$ are unions of blocks.
By Remark~\ref{rem:how to prove that something is a union of blocks}, it
suffices to prove that, for all $\pi \in \fB$ and irreducible $\tau \not \in \fB$,
we have $\Ext^1_{\cA_k}(\tau, \pi) = 0$ and $\Ext^1_{\cA_k}(\pi, \tau) = 0$.

Assume first that~$\fB$ has type~(2).
That $\Ext^1_{\cA_k}(\tau, \pi) = 0$ follows from 
Corollary~4.3.1 and Table~1 of Corollary~5.3.4 of \cite{Heyerderived}, %
while the statement that $\Ext^1_{\cA_k}(\pi, \tau) = 0$ follows from ~\cite[Theorem~4.2.12 and~(3.7.6)]{MR2667883} (or rather their proofs; \cite{MR2667883} assumes that the coefficient field is finite, but the same methods prove this result).\footnote{Alternatively,
one can apply the results of~\cite{hoff2025rightderivedfunctorsordinary},
which incorporate the results of~\cite{MR2667883}, but with no restrictions on
the coefficient field.  In fact, Thm.~C of the former reference
shows that the computations giving \cite[Cor.~5.3.4, Table~1]{Heyerderived}
and the computations of derived ordinary parts in~\cite{MR2667883} are
at base the same computation.}
These results also imply that there is a nontrivial extension between the elements of~$\fB$, hence $\fB$ is a block of~$\cA_k^{\ladm}$.

We now drop the assumption that~$\fB$ has type~(2).
Then, the result we have just proved shows that, if~$\tau$ is contained in a block of type~(2), then 
$\Ext^1_{\cA_k}(\tau, \pi) = 0$ and $\Ext^1_{\cA_k}(\pi, \tau) = 0$, as desired.
If~$\tau$ is not contained in a block of type~(2), then $\pi$ and~$\tau$ both descend to a finite subfield of~$k$,
and so the vanishing $\Ext^1_{\cA_k}(\tau, \pi) = 0$ and $\Ext^1_{\cA_k}(\pi, \tau) = 0$ follows from the classification of blocks in~\cite{MR3150248}.

At this point, we know that the sets~$\fB$ in the statement of the proposition are unions of blocks, and they are blocks if~$\fB$ has type~(2).
This is also true if~$\fB$ has type~(1) or~(3), since in these cases~$\fB$ is a singleton.
So there remains to prove that
\[
\{1_G, \St, \Ind_B^G(\omega \otimes \omega^{-1})\} 
\]
is a block, which is immediate from
the fact that it is a block of $\cA_{\bF_p}^{\ladm}$.
\end{proof}

We now consider the case of a 
general field of coefficients of characteristic~$p$.
The following definition will be useful in stating our results.

\begin{defn}\label{defn:companion weights}\leavevmode
\begin{enumerate}
  \item If $\sigma = \Sym^r \otimes \det^s$ is a Serre weight, and
  $0\le r\le p-3$, we define
  $\sigmacomp\coloneq \Sym^{p-3-r} \otimes \det^{a+b+1}$. 
  If~$\sigma = \Sym^{p-2} \otimes \det^a$ we define $\sigmacomp\coloneq  \sigma$. 
  If~$\sigma = \Sym^{p-1} \otimes \det^a$, we do not define~$\sigmacomp$.
  \item If~$f \ne T \in k[T]$ is an irreducible monic polynomial, then we define
  $f^*(T) \coloneq  f(0)^{-1}T^{\deg f}f(\zeta(p)/T)$.
  \item Let~$\sigma = \Sym^r \otimes \det^s$ be a Serre weight, 
assume that~$0 \leq r \leq p-2$,
and let~$f \ne T \in \cH(\sigma)$ be an irreducible polynomial.
Then we define~$\fB(\sigma, f)$ to be the set of irreducible $\cA_k$-subquotients of
\[
\cInd_{KZ}^G(\sigma)/f(T) \oplus \cInd_{KZ}^G(\sigmacomp)/f^*(T).
\]
\end{enumerate}
\end{defn}

Note that if~$f \ne T \in k[T]$ 
is irreducible and monic, then
$f^*$ is also irreducible and monic, and $\lambda \in \lbar k^\times$ is a root of~$f^*$ if and only if $\zeta(p)\lambda^{-1}$ is a root of~$f$.
The map
\begin{equation}\label{eqn:block involution}
(\sigma, f) \mapsto (\sigmacomp, f^*)
\end{equation}
is an involution on the set of pairs consisting of a Serre weight~$\sigma = \Sym^r \otimes \det^s$ with~$0 \leq r \leq p-2$, and an irreducible monic polynomial $f \ne T \in k[T]$.
Furthermore, $\fB(\sigma, f) = \fB(\sigmacomp, f^*)$.

\begin{prop}\label{type5blocks}
Let~$A = k$ be a %
field of characteristic~$p$.
Let~$\sigma = \Sym^r \otimes \det^s$ be a Serre weight, 
assume that~$0 \leq r \leq p-2$,
and let~$f \ne T \in \cH(\sigma)$ be an irreducible polynomial. %
Then~$\fB \coloneq  \fB(\sigma, f)$ is a block of $\cA_k^{\ladm}$.
\end{prop}
\begin{proof} %
We first prove that~$\fB$ is a union of blocks. We will do so by applying Remark~\ref{rem:how to prove that something is a union of blocks},
and so we let~$\pi, \tau$ be irreducible objects of~$\cA_k^{\ladm}$, and we assume that~$\pi \in \fB$, $\tau \not \in \fB$.
We need to prove that $\Ext^1_{\cA_k}(\pi, \tau) = 0$ and $\Ext^1_{\cA_k}(\tau, \pi) = 0$.
By Proposition~\ref{prop:ext and base-change}, we can verify this after extending scalars to
an algebraic closure~$\lbar k/k$.
It thus suffices to prove that if~$\pi', \tau'$ are irreducible subquotients of~$\pi \otimes_k \lbar k, \tau \otimes_k \lbar k$
then $\Ext^1_{\cA_{\lbar k}}(\pi', \tau') = 0$ and $\Ext^1_{\cA_{\lbar k}}(\tau', \pi') = 0$.

Assume this is not the case.
Choose a Serre weight~$\sigma'$ and an irreducible polynomial~$g(T) \in k[T]$ such that~$\tau$ is an irreducible subquotient of $\cInd_{KZ}^G(\sigma)/g(T)$.
Since we are assuming that $\tau \not \in \fB$, we know that $(\sigma', g) \not \in \{(\sigma, f), (\sigmacomp, f^*)\}$
Note that there exist $\lambda \in \lbar k^\times, \mu \in \lbar k$ such that $f(\lambda) = 0$ and $\pi'$ is an irreducible subquotient of $\cInd_{KZ}^G(\sigma)/(T-\lambda)$, resp.\
$g(\mu) = 0$ and $\tau'$ is an irreducible subquotient of $\cInd_{KZ}^G(\sigma')/(T-\mu)$.
Proposition~\ref{ordinary parts} then implies that $(\sigma', \mu) = (\sigma, \lambda)$ or $(\sigmacomp, \zeta(p)\lambda^{-1})$,
bearing in mind the isomorphism~\eqref{eqn:Hecke eigenvalue of parabolic induction}.
The first case implies that~$g = f$, since $g$ and~$f$ are irreducible monic polynomials in~$k[T]$ with a common root in~$\lbar k$.
The second case implies that $g = f^*$, for the same reason.
This contradiction concludes the proof that~$\fB$ is a union of blocks.

There remains to prove that~$\fB$ is a block.
Since $\pi\coloneq  \cInd_{KZ}^G(\sigma)/f(T)$ and $\pi'\coloneq  \cInd_{KZ}^G(\sigmacomp)/f^*(T)$ are always indecomposable of length~$\leq 2$, 
it suffices to prove that $\Ext^1_{\cA_k}(\pi, \pi') \ne 0$ and~$\Ext^1_{\cA_k}(\pi', \pi) \ne 0$.
By Proposition~\ref{prop:ext and base-change}, we can verify this after extending scalars to an algebraic closure $\lbar k/k$.
It then suffices to prove that
\[
\Ext^1_{\cA_{\lbar k}}(\cInd_{KZ}^G(\sigma)/(T-\lambda)^i, \cInd_{KZ}^G(\sigmacomp)/(T-\zeta(p)\lambda^{-1})^j) \ne 0
\]
for all~$i, j > 0$.
This is true when~$i = j = 1$ by Proposition~\ref{ordinary parts}, and it then follows in general 
from the long exact sequences in~$\Ext$ associated to
{\small 
\[
0 \to \cInd_{KZ}^G(\sigmacomp)/(T-\zeta(p)\lambda^{-1})^{j-1} \to \cInd_{KZ}^G(\sigmacomp)/(T-\zeta(p)\lambda^{-1})^j \to \cInd_{KZ}^G(\sigmacomp)/(T-\zeta(p)\lambda^{-1}) \to 0
\]}
and 
{\small
\[
0 \to \cInd_{KZ}^G(\sigma)/(T-\lambda) \to \cInd_{KZ}^G(\sigma)/(T-\lambda)^i \to \cInd_{KZ}^G(\sigma)/(T-\lambda)^{i-1} \to 0. \qedhere
\]
}%
\end{proof}

\begin{prop}\label{prop: blocks over perfect fields}
Let~$A$ be a complete Noetherian local $\cO$-algebra with %
residue field~$k$. %
Let~$\fB$ be a block of irreducible objects of~$\cA_A^{\ladm}$.
Then exactly one of the following is true:
\begin{itemize}
\item there exists a unique irreducible supersingular representation~$\pi$ with central character~$\zeta$ such that $\fB = \{\pi\}$; or
\item there exists a unique orbit~$\{(\sigma, f), (\sigmacomp, f^*)\}$ of the involution~\eqref{eqn:block involution} such that $\fB = \fB(\sigma, f) = \fB(\sigmacomp, f^*)$. 
\end{itemize}
\end{prop}
\begin{proof}
By Proposition~\ref{type5blocks}, $\fB(\sigma, f)$ is a block of~$\cA^{\ladm}$.
By Theorem~\ref{classifyirreducibles}, every non-supersingular irreducible object of~$\cA^{\ladm}$ is contained in a block of the form~$\fB(\sigma, f)$.
By Corollary~\ref{irreducibleHecke}, if the intersection
$\fB(\sigma, f) \cap \fB(\sigma', f')$ is not empty, then $(\sigma', f') \in \{(\sigma, f), (\sigmacomp, f^*)\}$.
To conclude the proof, it thus suffices to prove that if~$\pi$ is irreducible and supersingular, then~$\{\pi\}$ is a block of~$\cA^{\ladm}$.
By Remark~\ref{rem:how to prove that something is a union of blocks}, 
it is enough to prove that if $\tau \ne \pi$ is irreducible, then $\Ext^i_{\cA^{\ladm}}(\pi, \tau) = 0$ and $\Ext^i_{\cA^{\ladm}}(\tau, \pi) = 0$.
Passing to an algebraic closure of~$k$, this is a consequence of Proposition~\ref{ordinary parts}.
\end{proof}

\begin{rem}\label{rem:explicit description of blocks}
Applying~\eqref{eqn:Hecke eigenvalue of parabolic induction}, one sees that Proposition~\ref{prop: blocks over perfect fields} generalizes Proposition~\ref{ordinary parts}.
In more detail, if~$\fB$ is a block of~$\cA^{\ladm}_k$, then we have the following cases: %
\begin{enumerate}
\item $\fB = \{\pi\}$ for irreducible supersingular~$\pi$;
\item $\fB = \fB(\Sym^r \otimes \det^s, f)$, and $f = T-\lambda$ for some~$\lambda \in k^\times$, such that~$\lambda^2 \ne \zeta(p)$ if~$r = 0, p-2$. Then
\[
\fB = \{\Ind_B^G(\chi_1 \otimes \chi_2\omega^{-1}), \Ind_B^G(\chi_2 \otimes \chi_1 \omega^{-1})\},
\]
with $\chi_1 = \nr_{\lambda^{-1}\zeta(p)}\omega^s, \chi_2 = \nr_\lambda \omega^{r+s+1}$.
\item $\fB = \fB(\Sym^{p-2} \otimes \det^s, f)$ and~$f = T-\lambda$ for some~$\lambda \in k^\times$ such that~$\lambda^2 = \zeta(p)$. Then
\[
\fB = \{\Ind_B^G (\chi \otimes \chi \omega^{-1})\},
\]
with~$\chi = \nr_\lambda\omega^s$.
\item $\fB = \fB(\Sym^{0} \otimes \det^s, f)$ and~$f = T-\lambda$ for some~$\lambda \in k^\times$ such that~$\lambda^2 = \zeta(p)$. Then
\[
\fB = \{\chi \circ \det, \chi \otimes \St, \Ind_B^G(\omega \chi \otimes \omega^{-1} \chi)\},
\]
with~$\chi = \nr_\lambda\omega^s$.
\item $\fB = \fB(\Sym^r \otimes \det^s, f)$ and $\deg(f) > 1$. Then~$\fB$ is the set of irreducible subquotients of
\[
\Ind_B^G(\nr_{f^*(T)}\omega^s \otimes \nr_{f(T)}\omega^{r+s}) \oplus \Ind_B^G(\nr_{f(T)}\omega^{r+s+1} \otimes \nr_{f^*(T)}\omega^{s-1}),
\]
where~$\nr_{f(T)}$ denotes the unramified $\bQ_p^\times$-module~$k[T]/f(T)$ on which~$p$ acts as multiplication by~$T$.
\end{enumerate}
\end{rem}

\begin{rem}\label{rem:blocks after base extension}
If~$\fB$ is a block of type~(5), and $\fB \otimes_k \lbar k \coloneq  \{\pi \otimes_k \lbar k : \pi \in \fB \}$, then~$\fB \otimes_k \lbar k$ is almost always a union of blocks of type~(2).
The only exceptions occur when~$\zeta(p)$ is not a square in~$k$, in which case 
\[
\{\cInd_{KZ}^G\Sym^{p-2} \otimes \det\nolimits^a/(T^2-\zeta(p))\} \otimes_k \lbar k
\]
is a union of two blocks of type~(3), and
\[
\{\cInd_{KZ}^G\Sym^{0} \otimes \det\nolimits^a/(T^2-\zeta(p)), \cInd_{KZ}^G \Sym^{p-3} \otimes \det\nolimits^{a+1}/(T^2-\zeta(p))\} \otimes_k \lbar k
\]
is a union of two blocks of type~(4).
\end{rem}

When~$A = \cO$, the endomorphisms rings of projective generators of blocks of~$\cA^{\ladm}$ are also computed in~\cite{MR3150248}, and we point out the following consequence of this computation.

\begin{prop}\label{injectiveArtinian}%
Let~$\pi$ be an object of~$\cA^{\ladm}$. 
Then~$\pi$ is Artinian if and only if~$\soc_G(\pi)$ has finite length.
\end{prop}
\begin{proof}
One direction is immediate, so we assume that~$\soc_{\cA^{\ladm}}(\pi)$ has finite length and we prove that~$\pi$ is Artinian.
If~$\pi$ is not Artinian, then $\pi \otimes_\bF l$ is not Artinian for every finite extension~$l/\bF$.
On the other hand, Theorem~\ref{classifyirreducibles} and~\cite[Lemma~5.1]{MR3150248} imply that
\[
\soc_{\cA^{\ladm}_l}(\pi \otimes_\bF l) = \soc_{\cA^{\ladm}}(\pi) \otimes_\bF l
\]
has finite $\cA^{\ladm}_l$-length.
Replacing~$\bF$ with~$l$, we can therefore assume without loss of generality that $\soc_{\cA^{\ladm}}(\pi)$ is a direct sum of finitely many absolutely irreducible objects of~$\cA^{\ladm}$.

Since~$\pi$ injects into the $\cA^{\ladm}$-injective envelope of $\soc_{\cA^{\ladm}}(\pi)$, it suffices to prove that, if~$J$ is an injective object of $\cA^{\ladm}$ with absolutely irreducible
socle, then~$J$ is Artinian.
Let~$\fB$ be the block containing~$\soc(J)$, and let $J_\fB$ be the direct sum of injective envelopes of the objects of~$\fB$.
Since $J$ injects in~$J_\fB$, it suffices to prove that~$J_\fB$ is Artinian, or equivalently, that $J_\fB$ is Noetherian as an object of the opposite category to~$\cA^{\ladm}$.
By~\cite[Section~IV]{Gabrielthesis} (see also the discussion immediately preceding the statement of \cite[Thm.\ 1.5]{MR3150248}), 
it thus suffices to prove that $E_\fB \coloneq  \End_{\cA^{\ladm}}(J_\fB)$ is a Noetherian ring.
By~\cite[Cor.\ 6.5, Cor.\ 8.11, Cor.\ 9.33, Lem.\ 10.90]{MR3150248}, the ring~$E_\fB$ is finitely generated over its centre, which is a Noetherian local ring.
Hence~$E_\fB$ is indeed Noetherian, as desired. 
\end{proof}

\subsection{A chain of projective lines.}
\label{subsec:chain}
We now let $X$ denote a chain of projective lines over $\F$ with ordinary
double points, of length $(p \pm 1)/2$, where the sign is positive if
and only if~$\zeta$ is odd. %
We choose coordinates on each irreducible component of~$X$ in such a way that each singular point corresponds to~$0$ on one intersecting component and~$\infty$ on the other. We will refer to the points~$0$ and~$\infty$ as \emph{marked points}.
We are going to
label the components of~$X$ by cuspidal types~$\Theta(\chi)$ with central
character~$\zeta=\chi^{p+1}$; %
note that since~$\Theta(\chi)\cong\Theta(\chi')$ if and only if~$\chi'=\chi$ or~$\chi'=\chi^p$, %
there are as many
cuspidal types with central character~$\zeta$ as irreducible
components of~$X$.  In order to construct the labelling, it will be
useful to have the following definition.

\begin{defn}\label{adjacenttypes}
We will say that two cuspidal types~$\tau_1, \tau_2$ are \emph{adjacent} if there exist $\sigma_i \in \JH(\overline{\tau}_i)$ such that~$\{\sigma_1, \sigma_2\}$ is the set of Jordan--H\"older factors of a principal series representation of~$\F[\GL_2(\bF_p)]$.
\end{defn}

Recall from~\cite[Corollary~5.6]{BreuilPaskunas} that the extensions
amongst irreducible $\F[\GL_2(\bF_p)]$-representations are classified
by the following proposition.%

\begin{prop}\label{finiteextensions}
Let~$\sigma , \sigma'$ be Serre weights. Then $\dim_\bF\Ext_{\bF[\GL_2(\bF_p)]}^1(\sigma, \sigma') \leq 1$, and $\dim_\bF\Ext_{\bF[\GL_2(\bF_q)]}^1(\sigma, \sigma') = \dim_\bF\Ext_{\bF[\GL_2(\bF_p)]}^1(\sigma', \sigma)$. Furthermore, writing $\sigma = \Sym^b\bF^2 \otimes \operatorname{det}^a$, one of the following is true:
\begin{enumerate}
\item there exist exactly two Serre weights~$\sigma_1, \sigma_2$ such that $\Ext_{\bF[\GL_2(\bF_p)]}^1(\sigma, \sigma_i)$ is nonzero. One of these extensions is the mod~$p$ reduction of a lattice in a cuspidal type, and the other is the mod~$p$ reduction of a lattice in a principal series type.
\item $b = p-2$, there exists a cuspidal type~$\tau$ with~$\JH(\overline{\tau}) = \{\sigma\}$, and there exists a unique Serre weight~$\sigma_1$ such that $\Ext_{\bF[\GL_2(\bF_p)]}^1(\sigma, \sigma_1) \ne 0$. This nonsplit extension is the mod~$p$ reduction of a lattice in a principal series type.
\item $b = 0$, and there exists a unique Serre weight~$\sigma_1$ such that $\Ext^1_{\bF[\GL_2(\bF_p)]}(\sigma, \sigma_1) \ne 0$. This nonsplit extension is the mod~$p$ reduction of a lattice in a cuspidal type.
\item $b = p-1$, and~$\sigma$ is a projective $\F[\GL_2(\bF_p)]$-module.
\end{enumerate}
\end{prop}

\begin{cor}\label{adjacents}
Let~$\tau$ be a cuspidal type. Then one of the following is true:
\begin{enumerate}
\item there exist exactly two cuspidal types adjacent to~$\tau$, or
\item $\overline{\tau}$ contains a twist of~$\Sym^0$ or~$\Sym^{p-2}$, and there exists exactly one cuspidal type adjacent to~$\tau$.
\end{enumerate}
\end{cor}
\begin{proof}
  This is immediate from Proposition~\ref{finiteextensions}.
\end{proof}

\begin{prop}\label{labelcomponents}
There exist exactly two bijections~$\tau \mapsto X(\tau)$, from the set of cuspidal types with central character~$\zeta$ to the set of irreducible components of~$X$, with the following property: $X(\tau_1) \cap X(\tau_2)$ is not empty if and only if~$\tau_1$ and~$\tau_2$ are adjacent.
\end{prop}
\begin{proof}
Assume first that~$\zeta$ is odd.
Twisting by the determinant, we may assume
that~$\zeta|_{\mu_{p-1}(\bQ_p)} = \omega^{-1}$. 
Then there exists a cuspidal type~$\tau$ with central character~$\zeta$ such that~$\overline{\tau} \cong \Sym^{p-2}$, which therefore needs to be sent to one of the components with only one singular point.
By Corollary~\ref{adjacents}, the bijection is determined by choosing which one.

Similarly, if~$\zeta$ is even we can twist and assume that $\zeta|_{\mu_{p-1}(\bQ_p)} = 1$.
In this case, the bijection is determined by choosing the image of the cuspidal type whose reduction contains the trivial $\F[\GL_2(\bF_p)]$-module.
\end{proof}

In what follows we will make an arbitrary choice amongst the two
bijections constructed in Proposition~\ref{labelcomponents}.
Next, we record an additional property of the map $\tau \mapsto X(\tau)$.

\begin{lemma}
If~$\tau_1, \tau_2$ are adjacent cuspidal types, the pair~$(\sigma_1, \sigma_2)$ in Definition~{\em \ref{adjacenttypes}} is unique.
\end{lemma}
\begin{proof}
Let~$\sigma_i'$ be the other Jordan--H\"older factor of~$\overline{\tau}_i$, if any exist.
Let~$x = \dim \sigma_1$. 
Then $\dim \sigma_2 = p+1-x$, $\dim \sigma_1' = p-1-x$,  and $\dim \sigma_2' = x-2$.
Since $\sigma_2$ is not isomorphic to~$\sigma_2'$, and there is a principal series extension between~$\sigma_1$ and~$\sigma_2$, there is no principal series extension between~$\sigma_1$ and~$\sigma_2'$.
Similarly, there is no principal series extension between~$\sigma_1'$ and~$\sigma_2$.
Finally, $\dim \sigma_1' + \dim \sigma_2' = p-3$, so there is no principal series extension between~$\sigma_1'$ and~$\sigma_2'$.
\end{proof}
The following corollary now is immediate (recalling that by definition, a principal series type is irreducible).
\begin{cor}\label{labelintersections}
The map $\tau \mapsto X(\tau)$ induces a bijection from singular points of~$X$ to isomorphism classes of principal series types.
\end{cor}

\begin{rem}\label{rem:sigmasigmacomp}
Although we have chosen to label components of~$X$ by cuspidal types, another parametrization is possible, in terms of Definition~\ref{defn:companion weights}.
If we partition the set of Serre weights~$\Sym^r \otimes \det^s$ such that~$r \ne p-1$, according to orbits of the involution $\sigma \mapsto \sigmacomp$,
the semisimplified mod~$p$ reduction map $\tau \mapsto \JH(\lbar \tau)$ induces a bijection from isomorphism classes of cuspidal types,
to (unordered) pairs $\{\sigma, \sigmacomp\}$.
We can then define $X(\sigmasigmacomp) \coloneq  X(\tau)$, where~$\tau$ is such that $\JH(\lbar \tau) = \{\sigma, \sigmacomp\}$.
This is the parametrization we will use in~\cite{DEGcategoricalLanglands}. 
\end{rem}

\subsection{Coordinates on~$X$.}%
\label{coordinates}
Choose a Serre weight $\sigma$ (as usual with central
character compatible with~$\zeta$, as usual). %
In Section~\ref{subsec:chain} we have chosen coordinates on each irreducible component of~$X$ in such a way that each singular point corresponds to~$0$ on one intersecting component and~$\infty$ on the other.
We will use the isomorphism~\eqref{eqn: Hecke algebra of Serre weight} and the map $\tau \mapsto X(\tau)$ to
construct morphisms from the spectra of various Hecke algebras~$\cH(\sigma)$ to~$X$.
Thus we will be able to regard $\cInd_{KZ}^G \sigma$ as ``lying over'' a 
copy of $\A^1 = \Spec \F[T] = \Spec \cH(\sigma)$, which will form the basis of our localization theory for~$\cA_A$.

\begin{defn}\label{f_sigma}
Let~$\sigma$ be a Serre weight, and assume $\sigma$ is not a twist of~$\Sym^{p-1}$. Then~$\sigma$ is contained in a unique cuspidal type~$\tau$.
Define a morphism $f_\sigma: \Spec \cH(\sigma) \to X$ in the following way.
\begin{enumerate}
\item If~$\sigma$ is not a twist of~$\Sym^0, \Sym^{p-2}$, then it is contained in a unique principal series type~$\tau'$.
Then $f_\sigma: \Spec \cH(\sigma) \cong \bA^1 \to X(\tau)$ is either the map $x \mapsto x$ or the map $x \mapsto x^{-1}\zeta(p)$, where the sign is chosen so that $f_\sigma(0)$ is
the point of~$X(\tau)$ corresponding to~$\tau'$ under
Corollary~\ref{labelintersections}. 
\item If $\sigma \cong \Sym^0 \otimes \det^s$, then~$f_\sigma: \Spec \cH(\sigma) \cong \bA^1 \to X(\tau)$, 
is either the map $x \mapsto x$ or the map $x \mapsto x^{-1}\zeta(p)$,
where the sign is chosen so that $f_\sigma(0)$ is the nonsingular marked point of~$X(\tau)$.
\item If~$\sigma \cong \Sym^{p-2} \otimes \det^s$, then~$f_\sigma : \bA^1 \to X(\tau)$ 
is either the map
$x \mapsto x + x^{-1}\zeta(p)$ or the map $x \mapsto (x+x^{-1}\zeta(p))^{-1}$, 
where the sign is chosen so that $f_\sigma(0)$ is the singular point of~$X(\tau)$.
\end{enumerate}
We extend this definition to the case $\sigma \cong \Sym^{p-1} \otimes \det^{a}$ by letting~$f_{\sigma}$ be the same map as in Case~(2) for 
$\sigma \cong \Sym^0 \otimes \det^s$. Then~$f_\sigma$ is an open immersion whenever~$\sigma$ is not a twist of~$\Sym^{p-2}$.
\end{defn}

\begin{rem}\label{Steinbergdefinition}
The definition in the case of twists of~$\Sym^{p-2}$ is motivated by the existence of extensions between non-isomorphic irreducible quotients 
of $\cInd_{KZ}^G(\Sym^{p-2} \otimes \det^s)$, which does not occur for any other weights.
The definition for twists of~$\Sym^{p-1}$ and~$\Sym^0$ is motivated by 
the fact that, for every polynomial~$f \in k[T]$, 
the representations
\[
\cInd_{KZ}^G(\Sym^{0} \otimes \det\nolimits^s)/f(T), \cInd_{KZ}^G(\Sym^{p-1} \otimes \det\nolimits^s)/f(T)
\]
have the same semisimplification, by
Lemma~\ref{lem:cind homs} and Lemma~\ref{lem: quotient giving St and trivial}.
\end{rem}

The key property of the maps $f_\sigma$ is given in the following proposition. 

\begin{prop}\label{bijection with blocks}
Let~$A$ be a complete Noetherian local $\cO$-algebra,
with %
residue field~$k$.
Let~$x$ be a closed point of~$X \otimes_\bF k$.
Define
\[
\fB_x =
\bigcup_\sigma\bigcup_{\substack{y \,\, \mathrm{s.t.}\\ f_\sigma(y) = x}} \JH\left ( \cInd_{KZ}^G(\sigma)
\otimes_{\cH(\sigma)} y \right ). 
\]

Then~$\fB_x$ is a block of~$\cA^{\ladm}_A$, and the map $x \mapsto \fB_x$ is a bijection from the set of closed points of~$X \otimes_\bF k$ 
to the set of blocks of~$\cA^{\ladm}_{A}$.
\end{prop}
\begin{proof}
The maps~$f_\sigma$ have been defined so as to set up a bijection between the regular closed points of~$X$, and the set of orbits of the involution~\eqref{eqn:block involution}; and 
a bijection between the singular closed points of~$X$, and the Serre weights of irreducible supersingular representations. 
The proposition then follows from Proposition~\ref{prop: blocks over perfect fields}.

In more detail, assume first that~$\zeta$ is even and~$x$ is not a marked point.
There are two possibilities for the set of weights~$\sigma$ such that~$f_\sigma(\bA^1)$ contains~$x$: it either has the form
\[
\{\Sym^r \otimes \det\nolimits^s, \Sym^{p-3-r} \otimes \operatorname{det}^{r+s+1}\}
\]
for some even~$r \ne 0, p-3$ (which are the irreducible subquotients of a cuspidal type), or
\[
\{\Sym^0 \otimes \det\nolimits^s, \Sym^{p-1} \otimes \det\nolimits^s, \Sym^{p-3} \otimes \operatorname{det}^{s+1}\}
\]
(which are the irreducible subquotients of a cuspidal type, together with a Steinberg weight).
Bearing in mind Remark~\ref{Steinbergdefinition}, it follows that in either case there exist~$0 \leq r \leq p-3$ and 
an irreducible monic polynomial $f(T) \ne T \in k[T]$
such that %
\[
\fB_x = \JH \left ( \cInd_{KZ}^G(\Sym^r \otimes \det\nolimits^s)/f(T) \right ) \cup \JH \left (\cInd_{KZ}^G(\Sym^{p-3-r}\otimes\operatorname{det}^{r+s+1})/f^*(T) \right ),
\]
which in the terminology of Remark~\ref{rem:explicit description of blocks} is a block of type~(4) 
if~$r \in \{0, p-3\}$ and $f$ is a linear factor of $T^2-\zeta(p)$,
and a block of type~(2) or type~(5) otherwise (depending on whether~$f$ is linear or not).

If~$\zeta$ is even and~$x$ is a marked point,
the set of~$\sigma$ such that~$f_\sigma(\bA^1)$ contains~$x$ has the form
\[
\{\Sym^r \otimes \det\nolimits^s, \Sym^{p-1-r}\otimes \operatorname{det}^{r+s}\}
\]
which are the irreducible subquotients of a principal series type.
Thus
\[
\fB_x = \JH \left (\cInd_{KZ}^G(\Sym^r \otimes \det\nolimits^s)/T \right ) \cup \JH \left (\cInd_{KZ}^G(\Sym^{p-1-r}\otimes\operatorname{det}^{r+s})/T \right ),
\]
and these are two isomorphic supersingular irreducible representations, so $\fB_x$ is a block of type~(1).

The above discussion, together with Proposition~\ref{prop: blocks over perfect fields}, shows that $x \mapsto \fB_x$ is a bijection in the case that~$\zeta$ is even.
The analysis in the case that~$\zeta$ is odd is similar.
The only difference occurs when~$x$ is a point of a component indexed by a cuspidal type whose reduction is isomorphic to $\sigma = \Sym^{p-2} \otimes \det^s$.

In this case, if~$x$ is a singular point of~$X$, then~$f_\sigma^{-1}(x) = \{0\}$, so~$\fB_x$ is again a block of type~(1).
On the other hand, if~$x$ is a regular point of~$X$, 
then there exists an irreducible monic~$f \ne T \in k[T]$
such that $f_\sigma^{-1}(x)$ is the union of the vanishing sets of~$f$ and~$f^*$.
If~$f \ne f^*$, then
$\fB_x$ is a block of type~(2) (if~$f$ is linear) or a block of type~(5) (if~$f$ is not linear). %
If~$f = f^*$, then $\fB_x$ is a block of type~(3) (if~$f$ is linear) or a block of type~(5) (if~$f$ is not linear).
An application of Proposition~\ref{prop: blocks over perfect fields} then concludes the proof.
\end{proof}

\begin{rem}\label{rem: ambiguity of choice}
  In light of Proposition~\ref{bijection with blocks}, the existence of 
  two bijections with the property described in Proposition~\ref{labelcomponents} 
  is related to the fact that the
  group~$\bZ/2 \times \bZ/2$ acts on the category~$\cA$ by
  autoequivalences arising from twisting by the quadratic characters
  $\bQ_p^\times \to \F^\times$.
For while twisting by the unramified quadratic character gives rise to an isomorphism
\[
(\cInd_{KZ}^G \Sym^r \otimes \operatorname{det}^s)/(T+\lambda) \cong (\nr_{-1} \circ \operatorname{det})\otimes (\cInd_{KZ}^G \Sym^r \otimes \operatorname{det}^s)/(T-\lambda),
\]
twisting by ramified characters changes the $K$-socle of an irreducible object of~$\cA$.
\end{rem}

\subsection{Bernstein centres of blocks.}\label{subsec: Bernstein
  centres of blocks}
The results in this section might be of independent interest, but they are quite technical in nature and will only be applied in Section~\ref{subsec: BL gluing}.
We let~$A = \cO$ and~$k = \bF$.
Let~$x$ denote a block of absolutely irreducible objects
of~$\cA^{\ladm}$, and write $\cA_x$ for the full subcategory of~$\cA$ such that~$\pi \in \cA_x$ if and only if every irreducible subquotient of~$\pi$ is contained in~$x$.
Since every proper quotient of~$\cInd_{KZ}^G(\sigma)$ has finite length, by Corollary~\ref{cor:cofinite length}, we see that every object of~$\cA_x$ is locally admissible.
(This is a special case of the statement, and proof, of Lemma~\ref{lem:finite sets} below.)

The paper~\cite{MR3150248} describes an equivalence of~$\cA_x$ with a category of modules over a ring~$\tld E_x$ defined as follows. 
Let $\pi_x = \bigoplus_{\pi \in x} \pi$ and choose an injective envelope $\pi_x \to J_x$, so that~$P_x = J_x^\vee$ is a projective envelope of~$\pi_x^\vee$.
Let~$\tld E_x$ be the endomorphism ring of~$J_x$, which is naturally a topological ring.
Then the functor $\Hom_{\cA_x}(-, J_x)$ is an equivalence of~$\cA_x$ with the category of compact left $\tld E_x$-modules, and so it defines an isomorphism between the Bernstein centre of~$\cA_x$ and the centre~$Z_x$ of~$\tld E_x$.

\begin{remark}\label{Pontrjagin}
If~$\tau \in \cA_x$ and~$J$ is an injective object in~$\cA_x$, and~$E = \End_G(J)$, then~$E$ is a compact ring and~$\Hom_{\cA_x}(\tau, J)$ is a compact $E$-module: this follows from~\cite[Section~IV.4]{Gabrielthesis}.
Hence its Pontrjagin dual~$\Hom_{\cA_x}(\tau, J)^\vee$ is a discrete $E$-module.

On the other hand, $\Hom_{\cA}(\tau, J)$ still makes sense for general~$\tau \in \cA$.
In particular, it follows from~\cite[Lemma~2.1]{PaskunasBM} that if~$\lambda$ is  a finite length $KZ$-representation, then the action of $E$ on $\Hom_{\cA}(\cInd_{KZ}^{G}\lambda, J)$ is continuous for the discrete topology on this module.
Hence $\Hom_{\cA}(\cInd_{KZ}^{G}\lambda, J)^\vee$ is a compact $E$-module.
\end{remark}

If~$\pi \in x$ is an irreducible object we will write~$\pi \to J_\pi$ for an injective envelope of~$\pi$ in~$\cA_x$. 
It is often (but not always) the case that~$Z_x$ is isomorphic to~$\End(J_\pi)$ for an irreducible object~$\pi \in x$.
More precisely, we have the following result of Pa{\v s}k{\= u}nas.

\begin{thm}\label{centrethm}
Let~$x$ be a block of absolutely irreducible objects of~$\cA$, and choose~$\pi \in x$.
Then the natural map~$Z_x \to \End_G(J_\pi)$ is an isomorphism if~$x$ is a block of type~(1), (2) or~(4).
\end{thm}
\begin{proof}
This is~\cite[Proposition~6.3]{MR3150248} for type~(1), \cite[Corollary~8.7]{MR3150248} for type~(2), and~\cite[Corollary~10.78, Theorem~10.87]{MR3150248} for type~(4).
\end{proof}

We will often apply this result together with the following proposition.

\begin{prop}\label{cyclicity}
Let~$\pi$ be an absolutely irreducible object of~$\cA$.
Let~$x$ be the block of~$\cA$ containing~$\pi$, and choose an injective envelope $\pi \to J_\pi$ in~$\cA_x$.
If~$\sigma$ is a Serre weight, then
\[
\Hom_\cA(\cInd_{KZ}^G\sigma, J_\pi)^\vee
\]
is a cyclic module over the endomorphism ring~$E_\pi = \End_{\cA_x}(J_\pi)$.
\end{prop}
\begin{proof}
This is proved in most cases in~\cite[Proposition~2.9]{MR3429471}.
We give a different, uniform proof as follows.
Let~$\fm_\pi$ be the Jacobson radical of~$E_\pi$.
By Nakayama's lemma for compact modules, it suffices to prove that $\Hom_\cA(\cInd_{KZ}^G\sigma, J_\pi)^\vee/\fm_\pi$ is at most one-dimensional.
We know that
\[
\Hom_\cA(\cInd_{KZ}^G\sigma, J_\pi)^\vee/\fm_\pi \cong \Hom_\cA(\cInd_{KZ}^G\sigma, J_\pi[\fm_\pi])^\vee
\]
and the representation~$J_\pi[\fm_\pi]$ contains~$\pi$ as a subquotient with multiplicity one.
It therefore suffices to show that $\cInd_{KZ}^G\sigma$ has at most one locally admissible quotient~$\Pi$ with socle~$\pi$ appearing with multiplicity one in~$\Pi$.

To see this, observe that  by Lemma~\ref{quotient by polynomial}, $\Pi$ is necessarily a quotient of~$\cInd_{KZ}^G\sigma/f(T)$ for some Hecke operator~$f(T)$.
Since the socle of~$\Pi$ is absolutely irreducible we may assume that~$f(T) = (T-\lambda)^n$ for some~$n$.
Since~$\Pi$ contains~$\pi$ as a subquotient with multiplicity one, we may furthermore assume that~$n = 1$.
Hence~$\Pi$ is irreducible except possibly in the case that~$\sigma$ is a twist of~$\Sym^0$ or~$\Sym^{p-1}$ and~$\lambda = \pm \zeta(p)^{1/2}$.
However, in this case there are only two possibilities for~$\Pi$, and they are distinguished by their $G$-socle (indeed they non-isomorphic absolutely irreducible $G$-representations).
\end{proof}

Now let~$\sigma$ be a Serre weight and choose~$\lambda \in \bF$.
In Section~\ref{subsec: BL gluing} we will be working with the inverse system~$\{\tau_n\}$ of finite length $G$-representations defined by
\[
\tau_n = \cInd_{KZ}^G\sigma/(T-\lambda)^n.
\]
These are all contained in the same block, which we will denote by~$x$.
There is an action of the Hecke algebra~$\cH(\sigma)$ on~$\tau_n$ for all~$n$, and the transition maps are equivariant for this action.
In fact, the action extends to the completion of~$\cH(\sigma)$ at the maximal ideal generated by~$(T-\lambda)$, which is isomorphic to a power series ring in one variable.
We will show that in most cases this action is induced by the Bernstein centre of the block~$\cA_x$.

\begin{prop}\label{centreelementsI}
In the notation of the previous paragraph, if~$x$ has type~(1), (2) or~(4) then there exists an element of the Bernstein centre~$Z_x$ of~$\cA_x$ inducing the Hecke operator~$T$ on~$\tau_n$ for all~$n$.
\end{prop}
\begin{proof}
Assume first that~$\pi = \tau_1$ is irreducible, so that every composition factor of~$\tau_i$ is isomorphic to~$\pi$.
If~$\pi'$ is another irreducible object of~$\cA_x$ and $\pi' \to J_{\pi'}$ is an injective envelope in~$\cA_{x}$ then the space $\Hom_{\cA}(\tau_i, J_{\pi'})$ vanishes for all~$i$, since $\operatorname{soc}_G(J_{\pi'}) \cong \tau_i$ is not a composition factor of~$\tau_i$.
Hence the functor
\[
M_\pi(-) = \Hom_{\cA}(-, J_\pi)^\vee
\]
is fully faithful on the full subcategory of~$\cA_x$ generated by the representations~$\{\tau_n\}$.
By Proposition~\ref{cyclicity} the $E_\pi$-module~$M_\pi(\cInd_{KZ}^G\sigma)$ is cyclic,
while by Theorem~\ref{centrethm} the natural map~$Z_x \to E_\pi$ is an isomorphism.
Thus there exists~$T_\pi \in Z_x$ such that~$T_\pi$ induces the action of $T$ as an endomorphism of~$M_\pi(\cInd_{KZ}^G\sigma)$,
which is to say (given the full faithfulness of $M_\pi(-)$) that the element $T_\pi$ acts on $\cInd_{KZ}^G\sigma$ itself via~$T$.
Since~$\cInd_{KZ}^G\sigma$ surjects onto~$\tau_i$ equivariantly for~$T$ (by definition)
and~$T_\pi$ (since it is an element of the Bernstein centre),
the actions of $T$ and $T_{\pi}$ coincide on each~$\tau_i$.
The proposition thus follows in the case that~$\tau_1$ is irreducible.

After twisting (if necessary), there remains to prove the proposition in the case that~$\sigma = \Sym^0, \Sym^{p-1}$ and~$\lambda = 1$. 
Hence~$\tau_1$ is a nonsplit extension of~$1$ by~$\St$ (or vice versa) and the block containing~$\tau_1$ is $x = \{1, \St, \pi\}$, where~$\pi = \Ind_B^G(\omega \otimes \omega^{-1})$.
We will follow~\cite[Section~10.4]{MR3150248} and work in a quotient category of~$\cA_x$.
By~\cite[Lemma~10.84]{MR3150248} the kernel of the functor
\[
\Hom_{\cA_x}(-, J_{\pi} \oplus J_{\St})^\vee
\] 
is the category~$\mathfrak{T}(k)$ of modules with trivial $G$-action. 
By~\cite[IV.4, Th\'eor\`eme~4]{Gabrielthesis} this functor defines an equivalence of the quotient category $\fQ(\kbase) = \cA_x/\fT(\kbase)$ with the category of left discrete $\End_G( J_{\pi} \oplus J_{\St})$-modules. (Compare the proof of~\cite[Corollary~10.85]{MR3150248}, and note that our categories are opposite to those denoted in~\cite{MR3150248} by the same symbols.)
More precisely, while it is the functor
\[
\Hom_{\fQ(\kbase)}(-, J_{\pi} \oplus J_{\St})
\]
that gives an equivalence, by the equivalence of~(b) and~(c) in~\cite[III.2, Lemme~1]{Gabrielthesis}, we see that the natural map
\[
\Hom_{\cA_x}(X, J_{\pi} \oplus J_{\St}) \to \Hom_{\fQ(\kbase)}(X, J_{\pi} \oplus J_{\St})
\]
is an isomorphism for every object~$X$ of~$\cA_x$.

Since $\Hom_{\cA}(\tau_i, J_\pi) = 0$ for all~$i$, the functor $M_\St(-) = \Hom_{\cA}(-, J_\St)^\vee$ induces a fully faithful functor on the full subcategory of~$\fQ(\kbase)$ generated by the~$\tau_i$.
By Proposition~\ref{cyclicity} we know that $M_\St(\cInd_{KZ}^G\sigma)$ is cyclic over the endomorphism ring~$E_\St$ of~$J_\St$.
Arguing as before, it follows from Theorem~\ref{centrethm} that there exists~$T_x \in E_\St \cong Z_x$ such that $T = T_x$ as endomorphisms of~$\Hom_{\cA_x}(\tau_i, J_\St)$.
The faithfulness of~$M_\St(-)$ implies that $T = T_x$ as elements of~$\End_{\fQ(k)}(\tau_i)$.
This implies that the image in~$\cA_x$ of
\[
(T-T_x) : \tau_i \to \tau_i
\]
has trivial $G$-action.
Considering this statement for~$\tau_{i+1}$, and using that (as an immediate consequence of Lemma~\ref{lem: uniserial})  the only submodule of~$\tau_{i+1}$ with trivial $G$-action has length at most one, together with the fact that~$T-T_x$ commutes with the projection $\tau_{i+1} \to \tau_i$, it follows that~$T = T_x$ as elements of~$\End_{\cA_x}(\tau_i)$, as required.
\end{proof}

Blocks of type~(3) are an exception to Theorem~\ref{centrethm} and Proposition~\ref{centreelementsI} because the rings~$\End_G(J_\pi)$ are no longer commutative.
The following weaker result will suffice for our purposes; by a more detailed analysis of the endomorphism rings~$E_x$, taking into account their relationship with Galois pseudodeformation rings, it should actually be possible to compute explicit integral dependence equations for Hecke operators over the centre of the block. 

\begin{prop}\label{centreelementsII}
Let~$\sigma$ be a twist of~$\Sym^{p-2}$, let~$\lambda = \pm 1$, and let
\[
\tau_i = \cInd_{KZ}^G\sigma/(T-\lambda)^i.
\]
Let~$x$ be the block of type~(3) containing the irreducible representation~$\pi = \tau_1$.
Then there exist a positive integer~$n$ and central elements~$z_0, z_1, \ldots, z_n \in Z_x$ contained in the unique maximal ideal~$\fm_x \subset Z_x$ such that
\[
(T-\lambda)^{n+1}+z_n(T-\lambda)^n + \cdots + z_1(T-\lambda) + z_0 = 0
\]
holds in $\End_G(\tau_i)$ for all~$i$.
\end{prop}
\begin{proof}
Fix an injective envelope~$\pi \to J_\pi$ in~$\cA_x$ and let~$E_\pi = \End_G(J_\pi)$.
The functor 
\[
M_\pi(-) = \Hom_{\cO[G]}(-, J_\pi)^\vee 
\]
is defined on~$\cA$, and 
it restricts to an equivalence of~$\cA_x$ with the category of discrete left $E_\pi$-modules.
Hence the natural map~$Z_x \to E_\pi$ induces an isomorphism~$Z_x \to Z(E_\pi)$.
By~\cite[Corollary~9.25]{MR3150248}, the ring~$E_\pi$ is a free module of rank~4 over its centre, which is a local ring.
By Proposition~\ref{cyclicity}, the module
\[
M = M_\pi(\cInd_{KZ}^G\sigma)
\]
is cyclic over~$E_\pi$.
It follows that~$M$ is finite over~$Z_x$.
(See the proof of~\cite[Lemma~2.6]{MR3556448} for more details.)

The Hecke algebra~$\cH(\sigma)$ acts on~$M$ by $Z_x$-linear endomorphisms.
The endomorphism induced by~$(T-\lambda)$ on~$M^\vee$ is locally nilpotent, because every $G$-linear map $\cInd_{KZ}^G\sigma \to J_\pi$ factors through a representation of finite length, in fact through one of the~$\tau_i$.
The quotient~$M/\fm_x M$ is a finite-dimensional vector space over the residue field of~$Z_x$,  and it is dual to~$M^\vee[\fm_x]$, hence~$(T-\lambda)$ induces a nilpotent endomorphism of~$M/\fm_x M$. 
It follows that there exists~$m>0$ such that~$(T-\lambda)^m(M) \subset \fm_x M$.
This implies that there exist $n > 0$ and $z_0, z_1 , \ldots, z_n \in \fm_x$ such that %
\[
(T-\lambda)^{n+1}+z_n(T-\lambda)^n + \cdots + z_1(T-\lambda) + z_0 = 0
\]
as endomorphisms of~$M$.
Hence the same relation holds as endomorphisms of~$M_\pi(\tau_i)$ for all~$i$, since they are quotients of~$M$.
Since the functor~$M_\pi(-)$ is fully faithful on~$\cA_x$, this implies the proposition.
\end{proof}

\section{Localization of smooth
  \texorpdfstring{$\GL_2(\Q_p)$}{GL2(Qp)}-representations}\label{sec: localization of G reps}

\subsection{Localization}\label{subsec: localization}
In this subsection we begin to explain how we localize $\cA$ 
over~$X$. 
The fundamental input is that the closed points of $X$ are in bijection with
the blocks of $\cA^{\ladm}$, which was established in Proposition~\ref{bijection with blocks}. We work in the same generality as that result, and accordingly %
we let~$A$ denote a complete Noetherian local $\cO$-algebra, with %
residue field~$k$, 
and write~$X_A \coloneq  X \otimes_\bF k$.

\begin{df}\label{localizingsubcategory}
If $Y$ %
is a closed subset
of~$X_A$,
then we let $\cA_{A, Y}$, resp.\ $\cA_{A, Y}^{\ladm}$, denote the subcategory of %
$\cA_A$, resp.\ $\cA_A^{\ladm}$, 
consisting of those representations each of whose irreducible
subquotients lies in the union of the blocks corresponding under Proposition~\ref{bijection with blocks} to the closed points of~$Y$.
We write ~$\cA_Y$ for~$\cA_{\cO,Y}$.
We write~$i_{Y, *}: \cA_{A, Y} \to \cA_A$ for the inclusion functor.
\end{df}

We now establish that~~$\cA_{A, Y}$ is localizing; we refer to Appendix~\ref{app: localizing cats} for the basic
definitions and properties of localizing categories.

\begin{lemma}\label{lem: A_Y is localizing}%
The subcategories~$\cA_{A, Y} \to \cA_A$ and $\cA_{A, Y}^{\ladm} \to \cA_A^{\ladm}$ are localizing.
\end{lemma}
\begin{proof}
We give the proof in the smooth case; the same argument works in the locally admissible case.
To see that~$\cA_{A, Y}$ is a Serre subcategory, it suffices to check that given an exact sequence
\[
0 \to \pi_1 \to \pi \to \pi_2 \to 0
\]
in~$\cA$ the set~$\JH(\pi)$ of isomorphism classes of irreducible subquotients of~$\pi$ coincides with~$\JH(\pi_1) \cup \JH(\pi_2)$.
Since it is immediate that $\JH(\pi_1) \cup \JH(\pi_2) \subseteq \JH(\pi)$, let~$V$ be an irreducible subquotient of~$\pi$.
Then there are subspaces $W_1 \subset W_2 \subset \pi$ such that~$V \cong W_2/W_1$, hence~$V$ is a subquotient of~$\pi_2$ unless $W_1 + \pi_1 = W_2 + \pi_1$.
Similarly, $V$ is a subquotient of~$\pi_1$ unless~$W_1 \cap \pi_1 = W_2 \cap \pi_1$.
Since~$W_1 \ne W_2$, these two equalities cannot hold at the same time.
This concludes the proof that~$\cA_{A, Y}$ is a Serre subcategory.

To see that~$\cA_{A, Y}$ is localizing it thus suffices to check that it is
closed under direct sums. %
Assume that~$\pi_i \in \cA_{A, Y}$ and~$V$ is an irreducible submodule of a
quotient of $\bigoplus_{i \in I}\pi_i$.
Since~$V$ is cyclic over~$A[G]$, there exists a finite subset~$J \subset I$ such that~$V$ is a submodule of a quotient of~$\bigoplus_{i \in J} \pi_i$.
By the previous discussion, $V \in \bigcup_{i \in J} \JH(\pi_i) \subset \cA_Y$, implying the claim.
\end{proof}

\begin{rem}\label{rem: referee asked to put this in}
If $l/k$ is an extension of fields of characteristic~$p$, $Y \subset X_k$ is closed, and $\pi \in \cA_{k, Y}$, then
$\pi \otimes_k l \in \cA_{l, Y_l}$, where~$Y_l$ is the preimage of~$Y$ under the natural map $X_l \to X_k$.
In fact, by Lemma~\ref{lem: A_Y is localizing}, it suffices to prove that $\pi \otimes_k l$ is a filtered colimit of objects of $\cA_{l, Y_l}$;
hence, writing~$\pi$ as a filtered colimit of finitely generated subobjects, 
it suffices to prove that $\pi \otimes_k l \in \cA_{l, Y_l}$ when $\pi \in \cA_{k, Y}\cap \cA^{\fg}$.
If~$\pi$ is furthermore irreducible, this is immediate from the definition of the bijection in~\ref{bijection with blocks}.
If now $\pi \in \cA^{\fg}$, then it follows from Proposition~\ref{cind subobjects}
and a d\'evissage on~$\pi$
that
\[
\JH(\pi \otimes_k l) \subset \bigcup_{\tau \in \JH(\pi)}\JH(\tau \otimes_k l),
\]
which reduces the claim to the case where~$\pi = \tau$ is irreducible, as desired.
We will often make implicit use of this remark.
\end{rem}

The following lemma gives the precise sense in which~$\cA_A^{\ladm}$ decomposes as a product of blocks; it goes back to~\cite[Section~IV]{Gabrielthesis},
but we give a proof for convenience, and to establish notation.
If~$x \in X_A$ is closed,
the inclusion $\cA_{A, \{x\}}^{\ladm} \to \cA^{\ladm}$ has a right adjoint, which we will denote by $\pi \mapsto \pi_x$.
More precisely, $\pi_x$ is the maximal subobject of~$\pi$ that is contained in $\cA_{A, \{x\}}^{\ladm}$.
Since the inclusion preserves compact objects, the functor~$\pi \mapsto \pi_x$ commutes with filtered colimits.

\begin{lemma}\label{lem: block decomposition}\leavevmode
\begin{enumerate}
\item The sum of counits $\bigoplus_{x \in X_A} (-)_x \to \id_{\cA_A^{\ladm}}$ is a natural isomorphism.
\item The functor
\[
\prod_{x \in X_A} \cA_{A, \{x\}}^{\ladm} \to \cA^{\ladm}_A, (\pi_x)_{x \in X_A} \mapsto \bigoplus_{x \in X_A} \pi_x
\]
is an equivalence.
\end{enumerate}
\end{lemma}
\begin{proof}
Since both sides of $\bigoplus_{x \in X_A} (-)_x \to \id_{\cA_A^{\ladm}}$ preserve filtered colimits, and~$\cA_A^{\ladm}$ is locally finite,
it suffices to prove that this natural transformation is an isomorphism when evaluated at an object~$\pi$ of finite length.
This follows by induction on $\operatorname{length}_{\cA_A^{\ladm}}\pi$, using the fact that $\Ext^i_{\cA^{\ladm}_A}(\pi_x, \pi_{x'}) = 0$ if~$x \ne x'$.

For the second statement, the full faithfulness is immediate from the fact that
$\Hom_{\cA^{\ladm}_A}(\pi_x, \pi_{x'}) = 0$ if~$x \ne x'$, and the essential surjectivity is immediate from part~(1).
\end{proof}

\begin{lemma}
\label{lem:finite sets}
If $Y = \{y_1,\ldots,y_n\}$ is a finite closed subset of~$X_A$,
then the functor
\[
\prod_{i=1}^n \cA_{A, \{y_i\}}^\ladm \to \cA_{A, Y}, (\pi_{y_i})_{i=1}^n \mapsto \bigoplus_{i=1}^n \pi_{y_i}
\]
is an equivalence.
In particular, $\cA_{A, Y}=\cA_{A, Y}^{\ladm}$, and 
its subcategory~$\cA_{A, Y}^{\fg}$ of finitely generated objects is an Artinian category {\em (}in fact, all its objects have finite length{\em )}. 
\end{lemma}
\begin{proof}
By definition, the objects of $\cA_{A, Y}$ are the objects of $\cA_{A}$ whose irreducible subquotients  lie in the blocks corresponding
to the points~$y_i$.  
If  $\cInd_{KZ}^G \sigma \to \pi$ is a morphism where $\pi$ is  an object of $\cA_{A, Y}$,
and $\sigma$ is a Serre weight,
then since $\cInd_{KZ}^G \sigma$ has subquotients lying
in infinitely many different blocks, this morphism must factor through a non-trivial quotient of~$\cInd_{KZ}^G(\sigma)$. 
Such a quotient has finite length, by Corollary~\ref{cor:cofinite length}. Thus any finitely generated
subrepresentation of~$\pi$ is of finite length, and so admissible,
and so $\cA_{A, Y} \subseteq \cA_{A, Y}^{\ladm}$. 
The lemma then follows from 
Lemma~\ref{lem: block decomposition}.
\end{proof}

Note that if $A = \cO$ and~$x$ is a block of~$\cA^{\ladm}$, the category~$\cA_{\{x\}}$ coincides with the category~$\cA_x$ of Section~\ref{subsec: Bernstein
centres of blocks}.
We will typically write~$\cA_x$ for this category (which coincides with $\cA_x^{\ladm}$, by Lemma~\ref{lem:finite sets}, and as was already noted in Section~\ref{subsec: Bernstein centres of blocks}).

\begin{df}\label{defn: A_U}
If $U$ is an open subset of $X_A$, then we write $\cA_{A, U} \coloneq  \cA_A/\cA_{A, Y},$
where $Y \coloneq  X_A \setminus U$. %
\end{df}

Since~$\cA_{A, Y}$ is localizing, there is a fully faithful right adjoint $(j_U)_*: \cA_{A, U} \to \cA_A$
to the quotient functor $(j_U)^*:\cA_{A} \to
\cA_{A, U}$. 
Where~$U$ is understood, we write~$j_*$ for~$(j_U)_*$,
and~$j^*$ for~$(j_U)^*$. We often refer to~$j^{*}$ as the localization functor.

\begin{lem}
\label{lem:adjoint and colimits}\leavevmode
Let~$Y \subset X_A$ be a closed subset with complement~$U$.
\begin{enumerate}
\item
The right adjoint $j_{U*}$ commutes with filtered colimits.
\item The categories~$\cA_{A, Y}$ and~$\cA_{A, U}$ are locally Noetherian.
\item
An object of~$\cA_{A, Y}$ is Noetherian if and only if it is Noetherian in~$\cA_A$, if and only if it is finitely generated.
Furthermore, if we let $\cA_{A, U}^{\fg}$ denote the essential image of $\cA_A^{\fg}$ under
the localization functor~$j_U^{*}$,
then $\cA_{A, U}^{\fg}$ is precisely the subcategory of Noetherian objects in~$\cA_{A, U}$.  
\item If~$\pi_1, \pi_2 \in \cA_{A, Y}^{\fg}$, resp.\ $\pi_1, \pi_2 \in \cA_{A, U}^{\fg}$, then
$\Ext^i_{\cA^{\fg}_{A, Y}}(\pi_1, \pi_2) \isoto \Ext^i_{\cA_{A, Y}}(\pi_1, \pi_2)$, resp.\
$\Ext^i_{\cA^{\fg}_{A, U}}(\pi_1, \pi_2) \isoto \Ext^i_{\cA_{A, U}}(\pi_1, \pi_2)$.
\item If~$\pi \in \cA_{A, Y}^{\fg}$, resp.\ $\pi \in \cA_{A, U}^{\fg}$, then $\Ext^i_{\cA_{A, Y}}(\pi, -)$ commutes with filtered colimits, 
resp.\ $\Ext^i_{\cA_{A, U}}(\pi, -)$ commutes with filtered colimits.
\end{enumerate}
\end{lem}
\begin{proof}
The first three statements are immediate from Lemma~\ref{lem:adjoint and colimits abstract version} (using
Corollary~\ref{cor:noetherian} to identify the Noetherian objects of~$\cA_A$ with the finitely generated objects).
The fourth, resp.\ fifth statement
follows from these and Lemma~\ref{lem:f.g. Ext},
resp.\ Proposition~\ref{prop: properties of locally Noetherian categories}~\eqref{item: Exts commute with colimits in the right hand variable}.
\end{proof}

Our next goal is to compute the localizations of objects in~$\cA_A^{\ladm}$, 
and of generators of~$\cA_A$, 
i.e.\ their images under~$j_*j^*$.
This is done in
Lemma~\ref{localization of locally admissible objects}
and
Proposition~\ref{prop:localizing compact inductions} respectively.

We begin by recording a further property of the bijection in Proposition~\ref{bijection with blocks}.

\begin{lemma}\label{f_sigma preimages}
Let~$\sigma$ be a Serre weight, and let~$Y \subset X_A$ be a closed subset.
Let~$g \in \cH(\sigma)$ be a polynomial such that the vanishing set~$V(g)$ of~$g$ in $\Spec\cH(\sigma)$ equals $f_\sigma^{-1}(Y)$.
Then:
\begin{enumerate}
\item If~$h \in \cH(\sigma)$ is coprime to~$g$, then $(\cInd_{KZ}^G\sigma)/h(\cInd_{KZ}^G \sigma)$ has no irreducible subquotients in~$\cA_{A, Y}$.
\item If~$g$ is square-free, then $(\cInd_{KZ}^G\sigma)/g(\cInd_{KZ}^G \sigma)$ is the maximal multiplicity-free quotient of~$\cInd_{KZ}^G\sigma$ contained in~$\cA_{A, Y}$.
\end{enumerate}
\end{lemma}
\begin{proof}
  Since~$h$ is coprime to~$g$, and~$V(g) = f_\sigma^{-1}(Y)$, the image $f_\sigma(V(h))$ is disjoint from~$Y$.
The first part then follows from Proposition~\ref{bijection with blocks}. %

For the second part, note that  $(\cInd_{KZ}^G\sigma)/g(\cInd_{KZ}^G \sigma)$ %
is an element of~$\cA_{A, Y}$, by Proposition~\ref{bijection with blocks}.
It is multiplicity-free, by Corollary~\ref{irreducibleHecke}.
Finally, the maximality claim follows from the fact that, by part~(1), every quotient of~$\cInd_{KZ}^G \sigma$ contained in~$\cA_{A, Y}$ must be a quotient of $(\cInd_{KZ}^G \sigma)/g^n(\cInd_{KZ}^G \sigma)$ for some~$n$.\qedhere
\end{proof}

\begin{defn}\label{f_Y}
Let~$Y \subset X_A$ be a closed subset, and let~$\sigma$ be a Serre weight. 
If~$f_\sigma^{-1}(Y)$ is finite, 
we define~$f_Y \in \cH(\sigma)$ to be the monic squarefree generator of the ideal of $f_\sigma^{-1}(Y)$.
  Otherwise ~$f_\sigma^{-1}(Y)=\Spec \cH(\sigma)$, and 
we set~$f_Y \coloneq  0$.
\end{defn}

By Lemma~\ref{f_sigma preimages}, if~$f_Y \ne 0$, then 
$(\cInd_{KZ}^G \sigma)/f_Y(\cInd_{KZ}^G \sigma)$ is the maximal multiplicity-free quotient of~$\cInd_{KZ}^G\sigma$ contained in~$\cA_{A, Y}$.
The next two results, namely Lemma~\ref{lem:ext vanishing} and Lemma~\ref{lem:ext vanishing II}, %
are the main input in computing the localization of generators.

\begin{lemma}
\label{lem:ext vanishing}
Let~$\sigma$ be a Serre weight, let~$Y$ be a closed subset of $X_A$, and choose $g \in \cH(\sigma)$ such that $f_{\sigma}^{-1}(Y) = V(g)$
{\em (}a closed subset of $\Spec \cH(\sigma)${\em )}.
Then for
any object $\pi$ of $\cA_{A, Y}$,
we have $\Ext^i_{\cA_A}\bigl(\pi,(\cInd_{KZ}^G \sigma)[1/g]\bigr) =
0$ %
for all~$i$.
\end{lemma}
\begin{proof}Applying Lemma~\ref{lem:checking on generators},
we see that it suffices to prove the claimed vanishing in the case when
$\pi$ is finitely generated, which we assume from now on. 
By d\'evissage we can furthermore assume that the maximal ideal of~$A$ acts trivially on~$\pi$. 
Using the spectral sequence~\eqref{changecoefficients2-A-version} 
it therefore suffices to prove that $\Ext^i_{\cA_{k}}\bigl(\pi,(\cInd_{KZ}^G \sigma)[1/g]\bigr) = 0$ for all~$i$.
(In more detail, arguing with a minimal free resolution of~$k$ over~$A$, we see that the terms $\Ext^j_A(k, \cInd_{KZ}^G(\sigma)[1/g])$ that appear in $E_2^{ij}$ 
are isomorphic to finite direct sums of copies of~$\cInd_{KZ}^G(\sigma)[1/g]$.)
The formation of this $\Ext^i_{\cA_k}$-module is compatible
with extension of scalars, by Proposition~\ref{prop:ext and base-change},
so we can replace $k$ by some uncountable algebraically closed extension.
We thus assume that $k$ is uncountable and algebraically closed for the remainder of the argument.
Also, if $g = 0 $ then 
$(\cInd_{KZ}^G \sigma)[1/g] = 0$,
and the claimed vanishing follows immediately.  Thus we assume
for the remainder of the proof that $g \neq 0$.

We first suppose that $Y$ is a finite union of closed points. %
It will be useful in this case to prove a slightly stronger statement,
namely that
$\Ext^i_{\cA_k}\bigl(\pi,(\cInd_{KZ}^G \sigma)[1/h]\bigr) = 0$
for any non-zero multiple $h$ of~$g$.
We first note that $\pi \in \cA^{\ladm}_{k}$ by Lemma~\ref{lem:finite sets}. 
On the other hand, %
if $f \in \cH(\sigma)$ is coprime to $h$ (and so in particular
coprime to~$g$), then Lemma~\ref{f_sigma preimages}(1) shows that the finite length module
$(\cInd_{KZ}^G \sigma)/f(\cInd_{KZ}^G \sigma)$ has no
irreducible subquotient lying in~$\cA_{k, Y}$. 
Lemma~\ref{lem: block decomposition} thus implies that
\[
\Ext^i_{\cA^{\ladm}_k}\bigl(\pi, (\cInd_{KZ}^G \sigma)/f(\cInd_{KZ}^G \sigma)\bigr) = 0
\]
for all~$i$, and by Lemma~\ref{lem: smooth Ext equals admissible Ext} we can replace~$\cA_k^{\ladm}$ with~$\cA_k$.
A consideration of the long exact sequence of $\Ext^i_{\cA_k}$ arising from the
short exact sequence
$$0 \to (\cInd_{KZ}^G \sigma)[1/h] \buildrel f\cdot \over \longrightarrow
(\cInd_{KZ}^G \sigma)[1/h] \to (\cInd_{KZ}^G \sigma)/f (\cInd_{KZ}^G
\sigma)\to 0$$
now shows that $f$ acts invertibly on 
$\Ext^i_{\cA_k}\bigl(\pi,(\cInd_{KZ}^G \sigma)[1/h]\bigr).$  Certainly $h$ also
acts invertibly,
and so this $\Ext$ module is a vector space over the fraction field
of~$\cH(\sigma)$, which is isomorphic to $k(T)$, a field of uncountable
dimension over $k$.  On the other hand, Lemma~\ref{lem:countable dim}
shows that this $\Ext$ module is of countable dimension over~$k$.
Thus it must vanish, as claimed.

We now consider the general case.
Since we are assuming that~$\pi$ is finitely generated, it is a quotient of a representation
of the form $\cInd_{KZ}^G V$, for some finite dimensional $KZ$-representation~$V$.
A d\'evissage using the long exact Ext sequence reduces to the 
case when $\pi$ is a quotient of $\cInd_{KZ}^G\tau$, for some Serre weight~$\tau$.
If $\pi$ is a proper quotient of $\cInd_{KZ}^G \tau$, then $\pi$ is of finite length, and so lies in $\cA_{k, Y_0}$
for some finite closed subset $Y_0$ of $Y$.
The claimed vanishing %
then follows from the special case already proved,
so we  assume that $\pi$ is equal to $\cInd_{KZ}^G \tau$.

Then the assumption that $\cInd_{KZ}^G \tau$ is an object of $\cA_{k, Y}$
implies that for each non-zero element $q \in \cH(\tau)$,
the (finite length) quotient $(\cInd_{KZ}^G \tau)/q (\cInd_{KZ}^G \tau)$
lies in $\cA_{k, Y_0}$ for some finite closed subset~$Y_0$ of $Y$ (depending on~$q$).  
The result we already proved then shows that
$\Ext^i_{\cA_k}\bigl((\cInd_{KZ}^G \tau)/q(\cInd_{KZ}^G \tau) ,
(\cInd_{KZ}^G \sigma)[1/g]\bigr) = 0$
for each $i$,
and a consideration of the long exact Ext sequence
arising from the short exact sequence
$$0 \to \cInd_{KZ}^G \tau \buildrel q \cdot \over \longrightarrow
\cInd_{KZ}^G \tau \to (\cInd_{KZ}^G \tau)/q(\cInd_{KZ}^G\tau) \to 0$$
shows that $q$ %
acts invertibly on $\Ext^i_{\cA_k}\bigl(\cInd_{KZ}^G \tau, (\cInd_{KZ}^G \sigma)
[1/g] \bigr).$
As above, this Ext module is thus a vector space over the fraction
field of $\cH(\tau)$, while also being a countable dimensional $k$-vector space, and is thus equal to zero, as required.
\end{proof}

\begin{lemma}\label{lem:ext vanishing II}
Let~$Y$ be a closed subset of $X_A$, and let~$\tau$ be a finite length object of~$\cA_A$ with no irreducible subquotients in~$\cA_{A, Y}$.
Then for
any object $\pi$ of $\cA_{A, Y}$, and any~$i \geq 0$,
we have $\Ext^i_{\cA_A}\bigl(\pi,\tau\bigr) =
0$.
\end{lemma}
\begin{proof}
When~$Y$ is finite, this is an immediate consequence of Lemma~\ref{lem:finite sets}. %
In general, by d\'evissage we can assume that~$\tau$ is irreducible. 
Theorem~\ref{classifyirreducibles} shows that there is an exact sequence
\[
0 \to \cInd_{KZ}^G(\sigma) \xrightarrow{h} \cInd_{KZ}^G(\sigma) \to \tau' \to 0,
\]
for some Serre weight~$\sigma$ and some irreducible $h \in \cH(\sigma)$, such that~$\tau$ is an irreducible subquotient of~$\tau'$.

Assume first that~$\tau' = \tau$ is irreducible.
Writing~$f_\sigma^{-1}(Y) = V(g)$, we see that $h$ is coprime to~$g$, since otherwise $\tau$ would be a subquotient of $\cInd_{KZ}^G(\sigma)/g(T)$, and
so would be an object of~$\cA_{A, Y}$. 
Hence Lemma~\ref{lem:ext vanishing} implies the requisite vanishing, because of the exact sequence
\[
0 \to \cInd_{KZ}^G(\sigma)[1/g] \to \cInd_{KZ}^G(\sigma)[1/g] \to \tau \to 0.
\]
If~$\tau'$ is not irreducible, then Lemma~\ref{lem: quotient giving St and trivial}
produces an exact sequence
\[
0 \to \cInd_{KZ}^G(\sigma')[1/g] \to \cInd_{KZ}^G(\sigma)[1/g] \to \tau \to 0
\]
for a pair of Serre weights~$\sigma, \sigma'$, so we can conclude in the same way.
\end{proof}

We now compute the localization of locally admissible objects of~$\cA_A$.
If~$x \in X_A$ is closed, recall from Lemma~\ref{lem: block decomposition} that
we write~$(-)_x$ for the right adjoint to the inclusion $\cA_{A, \{x\}} = \cA_{A, \{x\}}^{\ladm} \to \cA_A^{\ladm}$, and that $\bigoplus \pi_x \iso \pi.$

\begin{lemma}\label{localization of locally admissible objects}
Let~$\pi \in \cA_A^{\ladm}$ and let~$U \subset X_A$ be open.
Then
\begin{equation*} %
j_{U*}j_U^*\pi \cong \bigoplus_{x \in U} \pi_x.
\end{equation*}
\end{lemma}
\begin{proof}%

Write $Y \coloneq  X_A \setminus U$, $j_* \coloneq  j_{U*}$ and~$j^*\coloneq  j_U^*$. 
Let $\pi_U \coloneq  \bigoplus_{x \in U} \pi_x$.
Then the kernel of the surjection $\pi \to \pi_U$, which equals $\bigoplus_{x \in Y} \pi_x$,
is an object of $\cA_{A, Y}^{\ladm}$, and so $\pi \to \pi_U$ becomes an 
isomorphism after applying $j_{*}j^*$.
Hence it suffices to prove that the unit of the $(j^*, j_{*})$-adjunction, evaluated at $\pi_U$, is an isomorphism.

Each of the functors $(-)_U, j^*$ and~$j_*$ commutes with filtered colimits (applying Lemma~\ref{lem:adjoint and colimits} (1) for~$j_*$).
Hence it suffices to prove that $\pi_U \isoto j_* j^* \pi_U$ when~$\pi$ has finite length.
Now the kernel~$\cK$ and cokernel~$\cC$ of this map are objects of~$\cA_{A, Y}$.
Lemma~\ref{lem:ext vanishing II}, applied to~$\tau \coloneq  \pi_U$, implies that $\cK = 0$.
The same lemma, applied to $\tau \coloneq  \operatorname{im}(\pi_U \to j_* j^* \pi_U)$,
shows that~$\cC$ is a direct summand of~$j_*j^* \pi_U$, and since
\[
\Hom_{\cA_A}(\cC, j_*j^* \pi_U) = \Hom_{\cA_{A, U}}(j^*\cC, j^* \pi_U) = 0
\]
we conclude that~$\cC = 0$, as desired.
\end{proof}

Lemma~\ref{localization of locally admissible objects} shows that our localization theory generalizes the block decomposition of~$\cA_A^{\ladm}$, as intended.
We now compute the localization of compact inductions of Serre weights.

\begin{prop}
\label{prop:localizing compact inductions} 
Let $Y$ be a closed subset of $X_A$, and write 
$U \coloneq  X_A\setminus Y$.  

\begin{enumerate}
\item If $\sigma$ is a Serre weight, 
and if $f_{\sigma}^{-1}(Y) = V(g)$
{\em (}a closed subset of $\Spec \cH(\sigma)${\em )}
for some $g \in \cH(\sigma)$,
then the natural map
$$\cInd_{KZ}^G \sigma \rightarrow
(\cInd_{KZ}^G \sigma)[1/g]$$
can be identified with the unit morphism
$$\cInd_{KZ}^G \sigma \to
j_{U*}j_U^* \cInd_{KZ}^G \sigma.$$
\item The functor $j_{U*}$ is exact.
\end{enumerate}
\end{prop}
\begin{proof}As usual, we write~$j$ for~$j_{U}$. 
We first put ourselves in the situation of~(1).
If $g = 0$ then $\cInd_{KZ}^G \sigma$ is an object of $\cA_{A, Y}$, 
because
Proposition~\ref{cind subobjects} implies that every irreducible
subquotient of~$\cInd_{KZ}^G \sigma$ is contained in a block in~$Y$. 
Hence both $(\cInd_{KZ}^G \sigma)[1/g]$
and
$j_*j^* (\cInd_{KZ}^G \sigma) $
vanish, and~(1) follows immediately. 
If $g \neq 0$,
then the embedding
\begin{equation}
\label{eqn:inclusion}
\cInd_{KZ}^G \sigma \hookrightarrow (\cInd_{KZ}^G \sigma)[1/g]
\end{equation}
is the filtered colimit of the embeddings
$$\cInd_{KZ}^G \sigma \to \dfrac{1}{g^n}(\cInd_{KZ}^G \sigma).$$
The cokernel of this latter morphism is isomorphic to
$\bigl(\cInd_{KZ}^G \sigma\bigr)/g^n \bigl(\cInd_{KZ}^G \sigma\bigr)$,
which is of finite length, and lies in $\cA_{A, Y}$.  
Thus each of these inclusions
induces an isomorphism after applying $j_*j^*,$ 
and hence the colimiting inclusion~\eqref{eqn:inclusion} also induces an isomorphism
after applying this functor (which preserves filtered colimits, by Lemma~\ref{lem:adjoint and colimits}). 
Thus it suffices to show that
$(\cInd_{KZ}^G \sigma)[1/g]$ is in the image of $j_*$,
or equivalently, by Lemma~\ref{lem:adjoint image}, that 
$\Ext^i_{\cA_A}\bigl(\pi, (\cInd_{KZ}^G \sigma)[1/g]\bigr) = 0$
for $i = 0,1$,
whenever $\pi$ is an object of $\cA_{A, Y}$.
This follows from Lemma~\ref{lem:ext vanishing}.

To prove~(2) it suffices to show that $R^ij_*(j^*\pi) = 0$ for all~$i \geq 1$ and all~$\pi \in \cA_A$.
By Lemma~\ref{Rj_* commutes with filtered colimits}, we know that~$R^ij_*$ preserves filtered colimits for all~$i \geq 0$.
Since~$\pi$ is a filtered colimit of finitely generated objects, and~$j^*$ preserves all colimits, it now suffices to prove that~$R^ij_*(j^*\pi) = 0$ for all $\pi \in \cA_A^{\fg}$ and all~$i \geq 1$.
Choosing a surjection $\cInd_{KZ}^G(V) \to \pi$ for some finite length smooth $A[KZ]_\zeta$-representation~$V$, we obtain a finite filtration on~$\pi$ whose graded pieces~$\{\tau_k\}_{k=1}^n$
are either irreducible, or of the form $\tau_k = \cInd_{KZ}^G(\sigma_k)$ for some Serre weight~$\sigma_k$.
Since~$j^*$ is exact, by d\'evissage it suffices to prove that $R^ij_*(j^*\tau) = 0$ for all~$i \geq 1$ whenever~$\tau$ is irreducible, 
or~$\tau \cong \cInd_{KZ}^G(\sigma)$ for a Serre weight~$\sigma$.
By Lemma~\ref{lem: vanishing Ext vanish Ri}, we equivalently need to prove that
\[
\Ext^i_{\cA_A}(\Pi, j_*j^*\tau) = 0
\]
for all~$\Pi \in \cA_{A, Y}$ and all~$i$.

When~$\tau = \cInd_{KZ}^G(\sigma)$, this is a consequence of part~(1) and Lemma~\ref{lem:ext vanishing}.
When~$\tau$ is irreducible, this is a consequence of Lemma~\ref{localization of locally admissible objects} and Lemma~\ref{lem:ext vanishing II}.
\end{proof}

\begin{cor}
\label{cor:ext vanishing}
Let $Y$ be a closed subset of $X_A$, and write $U \coloneq  X_A \setminus Y$.
If $\pi$ is an object of $\cA_{A, Y}$, and $\pi'$ is any object of $\cA_A$,
then $\Ext^i_{\cA_A}(\pi, j_{U*}j_U^*\pi') = 0$ for all~$i$.
\end{cor}
\begin{proof}
This follows directly from
Proposition~\ref{prop:localizing compact inductions}~(2),
together with Corollary~\ref{cor:exactness criterion}.
In fact, it was essentially already proved in the course of proving
Proposition~\ref{prop:localizing compact inductions}~(2).
\end{proof}

\begin{lemma}
\label{lem:ext control}
If $\sigma$ is a Serre weight,
and if $\cInd_{KZ}^G \sigma \hookrightarrow \pi$ is an essential embedding in~$\cA_A$, whose cokernel 
is of finite length, then there exists a non-zero $g \in \cH(\sigma)$
such that $\pi$ {\em (}thought of as an overmodule of $\cInd_{KZ}^G \sigma${\em )}
is contained in $\frac{1}{g} \cdot \cInd_{KZ}^G \sigma$.
\end{lemma}
\begin{proof}
By hypothesis, we have a short exact sequence $0 \to \cInd_{KZ}^G \sigma \to \pi \to \pi' \to 0$
where $\pi'$ is of finite length. It follows from  Lemma~\ref{localization of locally admissible objects}  that
there is a finite closed subset $Y \subset X_A$ such that, writing $U \coloneq  X_A \setminus Y$, we have $j_{U*}j^*_U \pi' = 0$, so that
 $j_{U*}j_U^*\cInd_{KZ}^G \sigma \iso j_{U*}j_U^*\pi.$  

Proposition~\ref{prop:localizing compact inductions}~(1) shows
that $j_{U*}j_U^*\cInd_{KZ}^G \sigma = (\cInd_{KZ}^G \sigma )[1/g]$ for some non-zero
element $g$ of $\cH(\sigma)$.  
Consider now the commutative square
$$\xymatrix{ \cInd_{KZ}^G \sigma \ar[r]\ar[d] & \pi \ar[d] \\
(\cInd_{KZ}^G \sigma)[1/g] \ar@{=}[r] & j_{U*}j_U^*\pi}
$$
Since the left-hand vertical arrow is injective, and since the upper horizontal
arrow is an essential embedding, we find that the right-hand vertical arrow is
injective.  Thus we find that $\pi$ embeds as a submodule of $(\cInd_{KZ}^G \sigma)[1/g]$
containing $\cInd_{KZ}^G \sigma$.   Since $\pi$ is finitely generated (being
an extension of a finite length module by a finitely generated module), we see that
if we replace $g$ by a sufficiently large power, then in fact
$\pi \subseteq \frac{1}{g} \cdot \cInd_{KZ}^G \sigma$,
as claimed.
\end{proof}

\begin{remark}\label{actuallyp-torsion}%
An immediate consequence of Lemma~\ref{lem:ext control} is that if $\cInd_{KZ}^G\sigma \to \pi$ is an essential embedding with cokernel of finite length then the maximal ideal~$\fm_A$ of~$A$ acts trivially on~$\pi$.
This can also be seen directly: if~$t \in \fm_A$, it is zero on~$\cInd_{KZ}^G(\sigma)$, and so the image of~$t: \pi \to \pi$ has finite length,
and therefore intersects $\cInd_{KZ}^G\sigma$ trivially.
Since the embedding $\cInd_{KZ}^G\sigma \to \pi$ is essential, this implies that~$t: \pi \to \pi$ is the zero map. 
\end{remark}

\begin{remark}
\label{rem:getting inductions}
Note that if $\sigma$ is not isomorphic to a twist of $\Sym^0$ or $\Sym^{p-1}$, and $h \in \cH(\sigma)$ is irreducible, then $(\cInd_{KZ}^G \sigma)/h (\cInd_{KZ}^G \sigma)$ is irreducible.
In this case, by choosing
$g$ in Lemma~\ref{lem:ext control} appropriately,
we may even assume that $\pi = \frac{1}{g}\cdot \cInd_{KZ}^G \sigma$.
In the exceptional cases, 
we get a counterexample from the exact sequence~\eqref{tree} below.
\end{remark}

\begin{remark}
If~$A = \cO$, another way to prove Lemma~\ref{lem:ext control} is to use the isomorphism
$$\Hom_\cA(\pi',(\cInd_{KZ}^G\sigma)[1/g]/\cInd_{KZ}^G\sigma) \iso \Ext^1_\cA(\pi',\cInd_{KZ}^G\sigma)$$
proved in~\eqref{eqn:simplified spectral sequence} below.
\end{remark}

\subsection{\v{C}ech resolutions}

If $\{U_0,\ldots,U_n\}$ is any finite open cover of $X_A$,
then for any object $\pi$ of $\cA_A$,
we obtain a functorial \v{C}ech resolution  %
\begin{equation}
\label{eqn:Cech resolution again}
0 \to \pi \to \prod_i (j_i)_*(j_i)^* \pi \to \cdots \to
(j_{0,\ldots,n})_*(j_{0,\ldots,n})^* \pi \to 0
\end{equation}
where as usual we have written $U_{i_1, \ldots, i_k}$
for~$U_{i_1}\cap\dots\cap U_{i_k}$, $j_{i_1, \ldots, i_k}$
for $j_{U_{i_1, \ldots, i_k}}$, and the differentials are given by
the usual formulas (see e.g.\
\cite[\href{https://stacks.math.columbia.edu/tag/01FG}{Tag
  01FG}]{stacks-project}). 
More precisely, in the terminology of the Stacks Project, we are working with the \emph{ordered \v{C}ech complex}.
\begin{rem}\label{defining the maps}
In order to define the maps in the \v{C}ech complex we need to specify a natural transformation
\[
j_{U*}j_U^* \to j_{V*}j_V^*
\]
of functors on~$\cA_A$, whenever~$V \subset U$ are open subsets.
Let~$Z_V \supset Z_U$ be the closed complements of~$U$ and~$V$ in~$X_A$.
The universal property of localization induces a functor $j^*_{UV}: \cA_{A, U} \to \cA_{A, V}$ such that $j_{UV}^*j_U^* = j_V^*$,
and $j_{UV}^*$ is the Serre quotient functor with kernel equal to the essential image of $\cA_{A, Z_V}$ in~$\cA_{A, U}$.
Since~$\cA_{A, Z_Y}$ is closed under colimits in~$\cA_A$, and~$j_U^*$ preserves colimits, this essential image is a localizing subcategory of $\cA_{A, U}$.
Hence $j_{UV}^*$ has a right adjoint~$j_{UV*}$, and there is an isomorphism $j_{V*} \cong j_{U*}j_{UV*}$.
The unit of adjunction $\id \to j_{UV*}j_{UV}^*$ then induces a map
\[j_{U*}j_U^* \to j_{U*}j_{UV*}j_{UV}^*j_U^* \cong j_{V*}j_V^*\]
which is the natural transformation we are looking for. 
Note that the composition $\id \to j_{U*}j_U^* \to j_{V*}j_V^*$, where the first arrow is the unit of the adjunction $(j_U^*, j_{U*})$, is the unit of the adjunction $(j_{V}^*, j_{V*})$. 
\end{rem}

We are going to prove that the complex~{\em \eqref{eqn:Cech resolution again}} is always acyclic.
To do so we will use another resolution of smooth~$\GL_2(\bQ_p)$-representations, arising from the Bruhat--Tits tree of~$\PGL_2(\bQ_p)$.
Let~$N$ be the normalizer of the Iwahori subgroup~$\Iw$.
Write $\delta: N \to \cO^\times$ for the orientation character, which is trivial on~$\Iw Z$ and takes value~$-1$ at~$\fourmatrix 0 1 p 0$. 
We identify the representation space of~$\delta$ with~$\cO$.
Then we have an exact sequence of $\cO[G]_\zeta$-representations
\begin{equation}\label{tree}
0 \to \cInd_N^G(\delta) \xrightarrow{\partial} \cInd_{KZ}^G(\triv) \xrightarrow{\mathrm{sum}} \triv \to 0,
\end{equation}
see for example~\cite[Section~2.4.19]{inertialp-adic}.

\begin{prop}
\label{prop:Cech acyclicity}
For any object $\pi$ of~$\cA_A$, and any finite open cover $\{U_i\}$ of~$X_A$,
the complex~{\em \eqref{eqn:Cech resolution again}} is acyclic.
\end{prop}
\begin{proof}
Tensoring~\eqref{tree} with~$\pi$ we obtain a short exact sequence
$$0 \to \cInd_{N}^G \delta\pi \to \cInd_{KZ}^G \pi \to \pi \to 0$$
where~$\delta$ is the nontrivial quadratic character of~$N/\Iw Z$.

Since $p > 2$, %
$\cInd_{N}^G\delta\pi$ is a direct summand of 
$\cInd_{\Iw Z}^G \pi = \cInd_{KZ}^G (\cInd_{\Iw Z}^{KZ} \pi).$ %
Since the formation of each of the terms in~\eqref{eqn:Cech resolution again}
is exact, we reduce to verifying the claim of the proposition
in the case when $\pi$ is of the form $\cInd_{KZ}^G V$ for some smooth $A[KZ]_\zeta$-representation~$V$.  Since 
both compact induction,
and the formation of the terms in~\eqref{eqn:Cech resolution again},
are compatible with passing to filtered colimits (use  Lemma~\ref{lem:adjoint and colimits}~(1)), %
we then reduce to the case when $V$ is finitely generated.  Finally,
since compact induction is exact, we reduce to the case when $\pi$ is of the
form $\cInd_{KZ}^G \sigma$ for some Serre weight~$\sigma$.

The open cover $\{U_i\}$ pulls back via $f_{\sigma}$ 
to an open cover $D(g_i)$ of $\Spec \cH(\sigma)$.
Proposition~\ref{prop:localizing compact inductions}
then implies that~\eqref{eqn:Cech resolution again}
may be identified with the tensor product
\begin{equation}
\label{eqn:Koszul}
(\cInd_{KZ}^G \sigma )\otimes_{\cH(\sigma)} K^{\bullet}(g_0,\ldots,g_n),
\end{equation}
where $K^{\bullet}(g_0,\ldots,g_n)$ denotes 
the usual \v{C}ech complex for a finite open cover 
of $\Spec$ of a ring by distinguished opens
associated to the sequence
$\{g_0,\ldots,g_n\}$.  
Since $\{D(g_i)\}$ is an open cover of the affine scheme $\Spec \cH(\sigma)$,
the complex $K^{\bullet}(g_0,\ldots,g_n)$ is acyclic,
and remains acyclic after tensoring with any $\cH(\sigma)$-module.
We consequently find that~\eqref{eqn:Koszul} is acyclic, and thus
that~\eqref{eqn:Cech resolution again} is acyclic, as claimed.
\end{proof}

\begin{cor}
\label{cor:localizing from inside}
If $Y \subset U$ is an inclusion of a closed subset of $X_A$ in an open subset of~$X_A$,
and if $\pi$ is an object of $\cA_{A, Y}$, then the unit morphism
$\pi \to (j_U)_*(j_U)^*\pi$ is an isomorphism; equivalently, $\pi$ lies in the essential
image of $(j_U)_*$.
\end{cor}
\begin{proof}
This follows from the acyclicity of the \v{C}ech resolution of $\pi$
with respect to the open cover
$\{U, V \coloneq  X_A\setminus Y\}$ of $X_A$, together with the fact that, by definition,
$(j_V)^*\pi = (j_{U \cap V}) ^*\pi=  0$. 
\end{proof}

\begin{cor}%
  \label{cor:disjoint closed localization lemma}Let $Y$ and $W$ be disjoint closed
  subsets of $X_A$.
If $\pi$ is an object of $\cA_{A, Y}$, and $\pi'$ is an object of $\cA_{A, W}$,
then $\Ext^i_{\cA_A}(\pi,\pi') = 0$ for all~$i$.
\end{cor}
\begin{proof} This follows immediately from Corollaries ~\ref{cor:ext vanishing}
  and ~\ref{cor:localizing from inside} (the latter being applied to the inclusion
$W \subseteq X \setminus Y$).
When~$Y, W$ are finite, this is also a consequence of Lemma~\ref{lem:finite sets}. %
\end{proof}

We end this subsection with the following counterpart to Lemma~\ref{lem:ext vanishing}.

\begin{lem}
\label{lem:another ext vanishing}
Let $Y$ be a closed subset of~$X_A$, let $\sigma$ be a Serre weight,
and write $f_{\sigma}^{-1}(Y) = V(g)$ {\em (}a closed subset of 
$\Spec \cH(\sigma)${\em )} %
for some $g \in \cH(\sigma)$.
Then for any object $\pi$ of $\cA_{A, Y}$,
each element of
$\Ext^i_{\cA_A}(\cInd_{KZ}^G \sigma, \pi)$
is annihilated by some power of $g$.
\end{lem}
\begin{proof}
By Lemma~\ref{lem:ext and colimits}, we can assume without loss of generality that~$\pi$ is finitely generated.
By a similar argument as in the proof of Lemma~\ref{lem:ext vanishing}, we then may (and do) assume that~$A=k$, and that~$k$ is uncountable and algebraically closed.

Let $h \in \cH(\sigma)$ be coprime to $g$.
By Lemma~\ref{f_sigma preimages}(1)
we know that $(\cInd_{KZ}^G\sigma)/h(\cInd_{KZ}^G \sigma)$
is of finite length, and lies in $\cA_{W}$
for some finite closed subset $W$ of $X_A \setminus Y$. %
If we set $U \coloneq  X_A \setminus W$,
then $\pi = (j_U)_*(j_U)^*\pi$,
by Corollary~\ref{cor:localizing from inside} 
(note that $Y \subseteq U$).
Thus
$\Ext^i_{\cA_k}\bigl((\cInd_{KZ}^G\sigma)/h(\cInd_{KZ}^G \sigma), \pi\bigr)$
vanishes for every~$i$, 
by Corollary~\ref{cor:ext vanishing},
and so a consideration of the long exact sequence of $\Ext$'s
arising from
the short exact sequence
$$0 \to \cInd_{KZ}^G \sigma \buildrel h \cdot \over\longrightarrow
\cInd_{KZ}^G \sigma \to (\cInd_{KZ}^G \sigma)/h(\cInd_{KZ}^G \sigma) \to 0$$
shows that $h$ acts invertibly on 
$\Ext^i_{\cA_k}(\cInd_{KZ}^G \sigma, \pi)$.
Thus the $\cH(\sigma)$-module structure on
$\Ext^i_{\cA_k}(\cInd_{KZ}^G \sigma, \pi)$
extends to a $\cH(\sigma)_S$-module structure, where $S$ denotes the multiplicative
subset of elements coprime to $g$.  
Since $k$ is uncountable, the ring $\cH(\sigma)_S$ has uncountable $k$-dimension,
while Lemma~\ref{lem:countable dim} shows that
$\Ext^i_{\cA_k}(\cInd_{KZ}^G \sigma, \pi)$ has countable $k$-dimension.
Thus $\Ext^i_{\cA_k}(\cInd_{KZ}^G \sigma, \pi)$ cannot contain any $\cH(\sigma)_S$
torsion-free submodule.  The lemma follows.
\end{proof}

\begin{remark}
\label{rem:alternate}
We explain an alternate proof of Lemma~\ref{lem:another ext vanishing}
in the case when~$Y$ is finite, assuming for simplicity that~$A = \cO$. 
In this case,
Lemma~\ref{lem:finite sets} shows that $\cA_{Y}$ is a product
of blocks of~$\cA^{\ladm}$,
and it follows from~\cite[ Cor.\ 5.18]{MR3150248}
that
any injective resolution of $\pi$ in $\cA_Y$
also provides an injective resolution in $\cA$. 
But if $I^{\bullet}$ is such a resolution in $\cA_Y$, 
then Lemma~\ref{f_sigma preimages}(2) implies that any element $\varphi \in \Hom_{\cA}(\cInd_{KZ}^G \sigma, I^{\bullet})$
factors through $(\cInd_{KZ}^G\sigma)/g^n (\cInd_{KZ}^G \sigma)$,
for some~$n$.
\end{remark}

\subsection{A stack of abelian categories}\label{subsec: stack of
  abelian categories}For each open subset~$U$ of~$X_A$ we have the
localized category $\cA_{A, U}$, and for each open subset $V\subset U$ we
have the natural localization functor $j_{UV}^*: \cA_{A, U}\to\cA_{A, V}$, compare Remark~\ref{defining the maps}.

\begin{thm}
\label{thm: stack of abelian categories}
The collection $\{\cA_{A, U}\}$ together with the localization functors~$j_{UV}^*: \cA_{A, U} \to \cA_{A, V}$ for~$V\subset U$ forms a stack {\em (}of abelian
categories{\em )} over the Zariski site of~$X_A$. 
\end{thm}
\begin{rem}\label{rem: what is a stack of abelian categories}
By a stack of abelian categories over the Zariski site of~$X_A$ we simply mean a stack whose fiber categories are abelian. 
There are further properties one could ask of such an object, such as
exactness of all the pullback functors, and indeed the stack
determined by the~$\{\cA_{A, U}\}$ %
has a lot more structure, such as pushforward functors and acyclic {\v C}ech resolutions.
This is very similar to the formalism underlying cohomological
descent~\cite[\href{https://stacks.math.columbia.edu/tag/0D8D}{Tag~0D8D}]{stacks-project},
as might be expected taking into account the connection
between the categories~$\cA_{A, U}$ %
and sheaves over a stack of $(\varphi,\Gamma)$-modules explained in the introduction.
\end{rem}
\begin{proof}[Proof of Theorem~{\em \ref{thm: stack of abelian categories}}]
Throughout this proof we will drop the symbol~$A$ from the notation.
Since the localization functors $j_{UV}^*: \cA_U \to \cA_V$ are the identity on objects, it is straightforward to check that~$U \mapsto \cA_U$ defines a pseudo-functor in the sense of~\cite[Definition~3.10]{Vistoli}. 
More precisely, the universal property of localization specifies a natural isomorphism
$j_{UW}^* \cong j_{VW}^*j_{UV}^*$
associated to any composite inclusion $W \subset V \subset U$ of open subsets, such that~$j_{UU}^*$ is the identity and Properties~(a) and~(b) in~\cite[Definition~3.10]{Vistoli} are satisfied.
By the procedure explained in~\cite[Section~3.1.3]{Vistoli} we obtain a fibered category~$\cA_\bullet \to X$.

Following~\cite[Definition~4.6]{Vistoli} we need to prove that for any covering~$\{U_i \to U\}_{i\in I}$ the functor from~$\cA_U$ to the category $\cA_\bullet(\{U_i \to U\})$ of objects with descent data is an equivalence.
Since~$X$ is quasicompact, we can assume without loss of generality that the indexing set is equal to~$\{0, 1\ldots, n\}$.
We will deduce the result from Proposition~\ref{prop:Cech
    acyclicity}.
In what follows we write as usual $U_{ij} = U_i \cap U_j$. 

Choose objects~$\pi_i \in \cA_{U_i}$ and isomorphisms
\begin{equation}\label{descentdatum}
j_{U_iU_{ij}}^*\pi_i \isom j_{U_jU_{ij}}^*\pi_j
\end{equation}
in~$\cA_{U_{ij}}$ for all~$i, j$, satisfying the cocycle condition after pullback to~$U_{ijk}$.
Fix indices~$i<j$ and define~$\pi_{ij} = j_{U_iU_{ij}}^*\pi_i$.
We have the unit of adjunction $\pi_i \to j_{U_iU_{ij}*}j_{U_iU_{ij}}^*\pi_i$, and applying~$j_{UU_i*}$ we find a map
\[
u_{ij}: j_{UU_i*}\pi_i \to j_{UU_{ij}*}\pi_{ij}.
\]
Similarly, the unit $\pi_j \to j_{U_jU_{ij}*}j_{U_jU_{ij}}^*\pi_j$ yields a map
\[
u_{ij}^+: j_{UU_j*}\pi_j \to j_{UU_{ij}*}(j_{U_jU_{ij}}^*\pi_j) \isom j_{UU_{ij}*}\pi_{ij}.
\]
where the isomorphism is induced by~\eqref{descentdatum}.
Putting these together we obtain a map
\begin{equation}\label{Cechtruncated}
u_{ij}^+-u_{ij} : \prod_{i} j_{UU_i*}\pi_i \to \prod_{i<j} j_{UU_{ij}*}\pi_{ij}. 
\end{equation}
If we are in the special case that $U_{i_0} = U$ for some index~$i_0$, the cocycle condition implies that the maps 
$j_{U_{i_0}U_{i_0 i}}^*\pi_{i_0} \to \pi_i = j_{U_iU_{i_0 i}}^* \pi_i$ define an isomorphism in~$\cA_\bullet(\{U_i \to U\})$.
Since the formation of~\eqref{Cechtruncated} defines a functor on~$\cA_\bullet(\{U_i \to U\})$, it follows that the complex~\eqref{Cechtruncated} is isomorphic to the first truncation of the {\v C}ech resolution of~$\pi_{i_0}$.
Furthermore, the formation of~\eqref{Cechtruncated} is compatible with pullback for any open inclusion~$V \subset U$.
It follows from this discussion together with Proposition~\ref{prop:Cech acyclicity} and the exactness of pullback functors that if we define~$\pi$ by the exact sequence
\[
0 \to \pi \to \prod_i j_{UU_i*}\pi_i \xrightarrow{\eqref{Cechtruncated}} \prod_{i<j} j_{UU_{ij}*}\pi_{ij}
\]
then~$j_{UU_i}^*(\pi) \cong \pi_i$, proving essential surjectivity of~$\cA_U \to \cA_\bullet(\{U_i \to U\})$.

Now let~$\alpha: \pi_1 \to \pi_2$ be a morphism in~$\cA_U$.
It follows immediately from Proposition~\ref{prop:Cech acyclicity} and functoriality of the {\v C}ech resolutions that if~$j_{UU_i}^*(\alpha) = 0$ for all~$i$ then~$\alpha = 0$, which proves that $\cA_U \to \cA_\bullet(\{U_i \to U\})$ is faithful.
To prove that it is full we begin with representations~$\pi, \tau \in \cA_U$ together with morphisms of descent data~$\alpha_i: j_{UU_i}^*\pi \to j_{UU_i}^*\tau$.
The argument for essential surjectivity implies that the~$\alpha_i$ induce a morphism on the first-truncated {\v C}ech resolutions of~$\pi$ and~$\tau$, and so a morphism~$\alpha: \pi \to \tau$.
Since the formation of the {\v C}ech resolution commutes with~$j_{UU_i}^*$ we deduce that that~$j_{UU_i}^*(\alpha) = \alpha_i$, which concludes the proof.
\end{proof}

\subsection{Further results about \texorpdfstring{$j_*$}{j*}}
From now on until the end of Section~\ref{sec: localization of G reps} we restrict to the case $A=\cO$.

As we can see from part~(1) of Proposition~\ref{prop:localizing compact inductions},
it is not typically the case that $j_*j^*\pi$ is finitely generated over~$\cO[G]_\zeta$,
even if $\pi$ is.  
Indeed we have the following result.

\begin{lemma}
\label{lem:f.g. adjoint}
If $U$ is an open subset of~$X$, and if $j_{U*}j_U^*\pi$ is finitely generated
for some object $\pi$ of~$\cA$, then the unit of adjunction
$\pi \to j_{U*}j_U^*\pi$ is surjective.
\end{lemma}
\begin{proof}
Write~$j$ for~$j_U$. We may write $\pi$ as the filtered colimit of its finitely generated submodules~$\pi_i$,
and then Lemma~\ref{lem:adjoint and colimits} shows that $j_*j^*\pi$ is the filtered
colimit of the~$j_*j^*(\pi_i)$.  Since $j_*j^*\pi$ is finitely generated by assumption, and so 
Noetherian, we find that $j_*j^*(\pi_i) = j_*j^*\pi$ for some
value of~$i$.  Replacing $\pi$ by $\pi_i$, we may thus assume that $\pi$ is finitely
generated.

Choose a surjection $\cInd_{KZ}^G V \to \pi,$
for some finite length $\cO[KZ]_\zeta$-representation~$V$.  We prove that $\pi \to j_*j^*\pi$ is surjective, by induction on
the length of~$V$ (the case~$V = 0$ being trivial).  Let
$\sigma$ be an irreducible subrepresentation of %
$V$, and let $\pi'$ denote the image of $\cInd_{KZ}^G \sigma$
in~$\pi$. If~$\pi'=0$ then we can replace $V$ by $V/\sigma$, and we are done
by induction.
Otherwise, since $j_*$ is exact, we see that $j_*j^*(\pi/\pi')$ is a quotient of
$j_*j^* \pi$, and so is finitely generated. Thus our inductive hypothesis
shows that $(\pi/\pi') \to j_*j^*(\pi/\pi')$ is surjective.  
Also, since $j_*j^* \pi'$ is a subobject of $j_*j^*\pi$, and since finitely generated
$G$-representations are Noetherian, we see that $j_*j^* \pi'$ is finitely
generated.

Now either $\cInd_{KZ}^G \sigma \iso \pi'$, or else $\pi'$ is a proper quotient
of~$\cInd_{KZ}^G\sigma$, in which case it is of finite length.
In the latter case, it follows from Lemma~\ref{localization of locally admissible objects} that $\pi' \to j_*j^*\pi'$ is surjective.
In the remaining case we have~$\cInd_{KZ}^G \sigma \iso \pi'$, and
a consideration of the formula of Proposition~\ref{prop:localizing compact
inductions}~(1) shows that since $j_*j^* \pi' = j_*j^*(\cInd_{KZ}^G \sigma)$ is finitely
generated, the natural morphism $\pi' \to j_*j^*\pi'$ is either a map with zero target, or an isomorphism, depending on whether~$g$ is zero or a unit.
In particular, it is surjective.
Again using the fact that $j_*$ is exact, we see that we have a morphism
of short exact sequences
$$\xymatrix{0 \ar[r] & \pi' \ar[r]\ar[d] & \pi \ar[r]\ar[d] & (\pi/\pi') \ar[r]\ar[d] & 0 \\
 0 \ar[r] & j_*j^*\pi' \ar[r] & j_*j^*\pi \ar[r] & j_*j^*(\pi/\pi') \ar[r] & 0 }
$$
in which the outer two vertical arrows are surjections.
The middle vertical arrow is thus a surjection as well.
\end{proof}

\begin{cor}
\label{cor:adjoint iso}
Let~$U \subset X$ be open.
If $\pi$ is finitely generated and lies in the essential image of~$j_*$,
then the same is true for any subquotient of $\pi$.
\end{cor}
\begin{proof}
By~\cite[III.2, Corollaire]{Gabrielthesis}, $\pi$ is finitely generated and contained in the essential image of~$j_*$ if and only if
$\pi \iso j_*j^*\pi$, and $j_*j^*\pi$ is finitely generated. 
Bearing this in mind, since $j_*$ is exact, it suffices to prove the statement of the lemma
for subobjects of $\pi$, since it then follows for quotients, and so also
for subquotients.
 If $\pi'$ is any subobject
of $\pi$, we see that $j_*j^* \pi'$ is a subobject of $j_*j^*\pi$,
and so finitely generated; thus Lemma~\ref{lem:f.g. adjoint}
implies that the natural morphism
\begin{equation}
\label{eqn:adjoint map}
\pi'\to j_*j^*\pi'
\end{equation}
is surjective.  
The fact that the corresponding map for $\pi$ is an isomorphism
implies that~\eqref{eqn:adjoint map} is also injective.  Thus in fact~\eqref{eqn:adjoint
map} is an isomorphism, and so $\pi'$ also lies in the essential image
of~$j_*$.
\end{proof}

\begin{remark}
The preceding corollary is not true in general without the assumption
of finite generation.  Indeed, if $\pi$ is finitely generated while $j_*j^*\pi$
is not finitely generated (Proposition~\ref{prop:localizing compact inductions}
provides plenty of examples of such~$\pi$), then if we let $\pi'$ denote the image of $\pi$
in $j_*j^*\pi$, we have that $\pi' \subsetneq j_*j^*\pi$,
while $j_*j^*\pi' = j_*j^*\pi$.   Thus $j_*j^*\pi$ is an object in
the essential image of $j_*$, while its subobject $\pi'$ is {\em not} in this
essential image.
\end{remark}

\begin{cor}
\label{cor:another ext vanishing}
Let $Y$ be a closed subset of~$X$, and write $U \coloneq  X \setminus Y$.
If $\pi$ is a finitely generated object lying in the essential image
of~$j_{U*}$,
and if $\pi'$ is an object of~$\cA_Y$,
then $\Ext^i_{\cA}(\pi,\pi') = 0$ for all~$i$.
\end{cor}
\begin{proof}
Arguing by d\'evissage, and taking into account Corollary~\ref{cor:adjoint iso},
we may assume that $\pi$ is a quotient of~$\cInd_{KZ}^G\sigma$,
for some Serre weight~$\sigma$.  If it is a proper quotient, then
it is of finite length, and so lies in $\cA_{W}$, for some 
finite closed subset $W$ of~$U$.
Hence in this case the corollary follows from Corollary~\ref{cor:disjoint closed localization lemma}.
Otherwise, we may assume that $\pi = \cInd_{KZ}^G\sigma,$ in which case the claimed
vanishing follows from Lemma~\ref{lem:another ext vanishing} (note that $g \in \cH(\sigma)^\times$ in this case, e.g.\
by Proposition~\ref{prop:localizing compact inductions}~(1)).
\end{proof}

We may use the preceding result to strengthen Lemma~\ref{lem:f.g. adjoint}.

\begin{cor}
\label{cor:f.g. adjoint}
Let $Y$ be a closed subset of~$X$, and write $U \coloneq  X \setminus Y$.
If $j_{U*}j_U^*\pi$ is finitely generated,
then the unit of adjunction $\pi \to j_{U*}j_U^* \pi$ is a split surjection.
\end{cor}
\begin{proof}
Lemma~\ref{lem:f.g. adjoint} ensures that the morphism $\pi\to j_*j^*\pi$
is surjective.   If we denote its kernel by~$\pi'$, then $\pi'$ lies in~$\cA_Y$.
The short exact sequence
$$0 \to \pi' \to \pi \to j_*j^*\pi \to 0$$
represents an element of~$\Ext^1_{\cA}(j_*j^*\pi,\pi')$,
and Corollary~\ref{cor:another ext vanishing} shows that this Ext module vanishes.
Thus the short exact sequence splits, as claimed.
\end{proof}

\subsection{Completion}\label{subsec:completion}%
Recall that $\cA^{\fg}$ %
denotes the full subcategory of $\cA$ consisting
of finitely generated objects.  Similarly, for any closed subset $Y$ of~$X$,
we write $\cA_Y^{\fg}$ to denote the full subcategory of $\cA_Y$ consisting
of finitely generated objects.
By Lemma~\ref{lem:adjoint and colimits abstract version}(3) ~$\cA_Y^{\fg}$ is a Serre subcategory of~$\cA_Y$, and since~$\cA_Y$ is a Serre subcategory of~$\cA$, we conclude that~$\cA^{\fg}$ and~$\cA_Y^{\fg}$ are Serre subcategories of~$\cA$.
If $Y$ is a finite closed subset of~$X$, then Lemma~\ref{lem:finite sets}
shows that $\cA_Y^{\fg}$ may equally well
be described as the subcategory of $\cA_Y$ consisting of finite length objects.
Applying the discussion of Section~\ref{subsubsec:completion}
we see that the inclusion $i_{Y, *}: \Pro(\cA^{\fg}_Y) \hookrightarrow \Pro(\cA^{\fg})$ 
admits a left adjoint.
\begin{defn}\label{def: completion}
Let~$Y \subset X$ be a closed subspace.
The functor 
\[
\widehat{(-)}_Y : \Pro(\cA^{\fg}) \to \Pro(\cA^{\fg}_Y) 
\]
is defined to be the left adjoint to the inclusion~$i_{Y, *}$.
\end{defn}

Being a left adjoint, the functor~$\widehat{(-)}_Y$ preserves colimits when they exist. 
In particular, it is right exact.
We sometimes write simply $\widehat{(-)}$ if $Y$ is understood from the context.
The following lemma describes $\widehat{\pi}_Y$ explicitly when~$\pi \in \cA^{\fg}$.

\begin{lemma}
\label{lem:completion description}
If~$\pi$ is an object of~$\cA^{\fg}$, there is a natural isomorphism
$$
\widehat{\pi}_Y \iso \quoteslim{} \pi',
$$
where $\pi'$ runs over the cofiltered directed set of quotients of~$\pi$ 
lying in~$\cA_Y^{\fg}$.
\end{lemma}
\begin{proof}
This is a particular instance of Lemma~\ref{lem:abstract completion description}.
\end{proof}

\begin{cor}\label{completion description for Serre weights}
Let~$\sigma$ be a Serre weight and $\pi = \cInd_{KZ}^G(\sigma)$.
Let~$Y \subset X$ be a closed subset, and let~$f_Y$ be as in Definition~{\em \ref{f_Y}}.
Then
\[\widehat \pi_Y \isom \quoteslim{n} \pi/f_Y^n \pi.\]
\end{cor}
\begin{proof}
By Lemma~\ref{f_sigma preimages}, %
the quotients~$\pi/f^n_Y\pi$ are cofinal in the directed set in Lemma~\ref{lem:completion description}.
\end{proof}

Our next goal is to prove Corollary~\ref{cor:completion is exact}, which shows that the completion functor $\pi \mapsto \widehat \pi_Y$ is exact. 

\begin{prop}%
\label{prop:extensions as pullbacks}
Let $Y$ be a closed subset of $X$,
and suppose that 
$$0 \to \pi' \to \pi \to \pi'' \to 0$$ is a short 
exact sequence in~$\cA^{\fg}$, with $\pi'$ being an object
of $\cA_Y^{\fg}$. %
Then we may find a commutative diagram with exact rows
$$\xymatrix{0 \ar[r] & \pi' \ar[r]\ar@{=}[d] & \pi \ar[r] \ar[d] &
\pi'' \ar[r]\ar[d] & 0 \\
0 \ar[r] & \pi' \ar[r] & \widetilde{\pi} \ar[r] & \widetilde{\pi}'' \ar[r] & 0}$$
in which the right-hand {\em (}and hence also the middle{\em )} vertical arrow
is surjective,
and in which $\widetilde{\pi}''$ {\em (}and hence also $\widetilde{\pi}${\em )}
are objects of $\cA_Y^{\fg}$.
\end{prop}
\begin{proof}
By assumption we may find a surjection $\cInd_{KZ}^G \tau \to \pi''$,
for some finite length representation $\tau$ of~$\cO[KZ]_\zeta$,
and we argue by induction on the $KZ$-length of~$\tau$.
Thus, to begin with, we assume that $\tau = \sigma$ is a Serre weight.
If the surjection $\cInd_{KZ}^G \sigma \to \pi''$ is not an isomorphism,
then $\pi''$ is of finite length, 
and so it is an object of $\cA_W$,
for some finite closed subset $W$ of~$X$.  Write $W' = W \cap Y$
and $W'' = W\setminus Y$, so that $W$ is the disjoint union of its
closed subsets $W'$ and~$W''$. Correspondingly, %
we may write $\pi'' = \pi''_1 \oplus \pi''_2,$
where $\pi''_1$ is an object of $\cA_{W'} \subseteq \cA_Y$,
and $\pi''_2$ is an object of $\cA_{W''}$. Since $Y$ and $W''$ are
disjoint, it follows from Corollary~\ref{cor:disjoint closed
  localization lemma} that
$\Ext^1_{\cA}(\pi''_2,\pi') = 0$.  Thus the natural map
$$\Ext^1_{\cA}(\pi''_1,\pi') \to \Ext^1_{\cA}(\pi'',\pi')$$
(given by pullback along the surjection~$\pi''~\to~\pi''_1$)
is an isomorphism, and so $\pi$ is pulled back from an extension
$\pi_1$ of $\pi''_1$ by~$\pi'$.
This concludes the proof in this case.

Suppose, then, that $\cInd_{KZ}^G \sigma \iso \pi''$.
Write $f_{\sigma}^{-1}(Y) = V(g)$, for some $g \in \cH(\sigma)$.
If $g = 0$, then %
$\pi''$ itself is an object
of~$\cA_Y^{\fg}$, and there is nothing to prove.
Suppose instead that $g$ is non-zero.
The extension $\pi$ represents a class
of $\Ext^1_{\cA}(\cInd_{KZ}^G \sigma,\pi')$,  
and Lemma~\ref{lem:another ext vanishing} shows
that some power $g^n$ of $g$ annihilates this class.
A consideration of the long exact Ext sequence associated
to the short exact sequence $$0 \to \cInd_{KZ}^G \sigma
\buildrel g^n \cdot \over \longrightarrow
\cInd_{KZ}^G \sigma \to
(\cInd_{KZ}^G \sigma)/g^n(\cInd_{KZ}^G\sigma)
\to 0 $$
then shows that the class of $\pi$ arises by pullback along 
the surjection 
$\cInd_{KZ}^G \sigma \to (\cInd_{KZ}^G \sigma)/g^n(\cInd_{KZ}^G)$,
whose target is an object of $\cA_Y^{\fg}$.

We now turn to the general case, when $\tau$ is not necessarily irreducible.
We need to construct a subrepresentation~$\pi_1$ of~$\pi$ such that $\pi_1 \cap \pi' = 0$, and $\pi/\pi_1$ is an object of~$\cA_Y^{\fg}$.
Our inductive hypothesis allows us to
choose a subrepresentation $\pi''_0$ of $\pi''$ such that the statement of the
proposition holds for each of $\pi''_0$ and $\pi''/\pi''_0$ in place
of $\pi''.$  Pulling back the extension given by $\pi$ along the inclusion
$\pi''_0\subseteq \pi'',$ we obtain a short exact sequence
$$0 \to \pi' \to \pi_0 \to \pi''_0 \to 0,$$
and by assumption we may find a subrepresentation $\pi_2 \subseteq \pi_0$
such that $\pi' \cap \pi_2 = 0$, and such that $\pi_0/\pi_2$ is an object
of $\cA_Y^{\fg}.$
Next consider the short exact sequence
$$0 \to \pi_0/\pi_2 \to \pi/\pi_2 \to \pi/\pi_0 (\iso \pi''/\pi''_0 ) \to 0.$$
Again by the inductive assumption, we may find a subrepresentation $\pi_1$ of $\pi$
containing $\pi_2$ such that $\pi_0 \cap \pi_1 = \pi_2$,
and such that $\pi/\pi_1$ ($= (\pi/\pi_2)/(\pi_1/\pi_2)$)
is an object of $\cA_Y^{\fg}$.
Then $\pi' \cap \pi_1 = (\pi' \cap \pi_0) \cap \pi_1 = \pi'\cap \pi_2 = 0,$
and so~$\pi_1$ is the subrepresentation of~$\pi$ that we are looking for.
\end{proof}

\begin{cor}
\label{cor:completion is exact}
The completion functor $\cA^{\fg} \to \Pro(\cA_Y^{\fg}), \pi \mapsto \widehat{\pi}_Y$ is exact.
\end{cor}
\begin{proof}
If $\pi' \subseteq \pi$ is an inclusion of objects of~$\cA^{\fg}$,
we must prove that the induced morphism $(\pi')^\wedge_Y \to (\pi)^\wedge_Y$
is a monomorphism. 
By Lemma~\ref{lem:monomorphism criterion} and the definition of
the functor~$\widehat{(\text{--})}_Y$ as an adjoint,
it suffices to show that any morphism
$\pi' \to \overline{\pi}'$ whose target
lies in~$\cA^{\fg}_{Y}$ 
can be placed in a commutative diagram 
$$\xymatrix{ \pi' \ar[r]\ar[d] & \pi \ar[d] \\
\overline{\pi}' \ar[r] & \overline{\pi},}
$$
where $\overline{\pi}$ also lies in $\cA^{\fg}_Y$ and in which the bottom arrow is a monomorphism.
In fact, it suffices to do this with
$\overline{\pi}'$
replaced by the image $\widetilde{\pi}'$ of $\pi'$ in~$\overline{\pi}'.$
(Indeed, if we then find a monomorphism $\widetilde{\pi}' \hookrightarrow \widetilde{\pi}$ with the required properties,
we can then take $\overline{\pi}$ 
to be the coproduct of $\widetilde{\pi}$ and $\overline{\pi}'$ over  $\widetilde{\pi}'$.)
Now, if we push out the short exact sequence
$$0 \to \pi' \to \pi \to \pi/\pi' \to 0 $$
along the surjection $\pi' \to \widetilde{\pi}'$,
and then apply Proposition~\ref{prop:extensions as pullbacks},
we obtain the required quotient $\widetilde{\pi}$ of~$\pi$.
\end{proof}

If $\pi$ is an object of~$\cA^{\fg}$, and $\pi'$ is an object of $\cA^{\fg}_Y$, there is an isomorphism
\begin{equation}\label{eqn:adjunction-ProAfgY}\Hom_{\Pro(\cA^{\fg}_Y)}(\widehat \pi_Y, \pi') \iso \Hom_{\cA^{\fg}}(\pi, \pi')\end{equation}
obtained via the adjunction property of completion and the fact that $\cA^{\fg} \to \Pro(\cA^{\fg})$ is fully faithful.
This isomorphism can also be described as follows.
Lemma~\ref{lem:completion description} shows that $ \widehat \pi_Y \cong \varprojlim_{i \in I} \pi_i$, where~$I$ is the cofiltered set of quotients of~$\pi$ contained in~$\cA^{\fg}_Y$.
If~$i \in I$, the composition of the map $\widehat \pi_Y \to \pi_i$ with the unit of adjunction $\pi \to \widehat \pi_Y$ is the quotient map $\pi \to \pi_i$. 
Then the isomorphism \eqref{eqn:adjunction-ProAfgY} is the composition
$$
\Hom_{\Pro(\cA^{\fg}_Y)}(\widehat{\pi}_Y,\pi')
\inverseiso\varinjlim_I \Hom_{\cA^{\fg}_Y}(\pi_i,\pi') \iso \Hom_{\cA^{\fg}}(\pi,\pi'),
$$
where the first isomorphism is an instance of~\eqref{eqn:first hom formula}.
The following lemma extends this to higher $\Ext$ groups.

\begin{lemma}
\label{lem:ext comparison}
Let $\pi$ be an object of $\cA^{\fg}$,
and write 
\[
\displaystyle \widehat{\pi}_Y = \quoteslim{I} \pi_i,
\]
as in Lemma~{\em \ref{lem:completion description}}.
If $\pi'$ is an object of $\cA^{\fg}_Y$,
then, for any value of~$n$, the natural morphism
\[\Ext^n_{\Pro(\cA^{\fg}_Y)}(\widehat{\pi}_Y,\pi') 
\inverseiso
\varinjlim_{I} \Ext^n_{\cA^{\fg}_Y}(\pi_i,\pi')
\to \Ext^n_{\cA^{\fg}}(\pi,\pi')\] 
is an isomorphism
whose inverse is the completion functor
\[
\Ext^n_{\cA^{\fg}}(\pi, \pi') \xrightarrow{\widehat{(-)}} \Ext^n_{\Pro (\cA_Y^{\fg})}(\widehat \pi_Y, \pi').
\]
\end{lemma}
\begin{proof}
Bearing in mind that completion is exact (by Corollary~\ref{cor:completion is exact}) 
and that the counit $\widehat{\pi}'_Y \to \pi'$ is an isomorphism (by Lemma~\ref{lem:completion description}),
this is a special case of Lemma~\ref{lem:abstract ext comparison}.
\end{proof}

We also have the following corollary,
which in the case of finite $Y$ also follows directly from Lemma~\ref{lem:finite
sets} together with~\cite[ Cor.\ 5.18]{MR3150248}
(as was already noted in Remark~\ref{rem:alternate}).

\begin{cor}
\label{cor:ext comparison}
The inclusion $\cA_Y \hookrightarrow \cA$ preserves injectives.
Consequently, if $\pi$ and $\pi'$ are objects of $\cA_Y$,
then the natural morphism
$$\Ext^n_{\cA_Y}(\pi,\pi') \to \Ext^n_{\cA}(\pi,\pi')$$
is an isomorphism, for every value of~$n$.
\end{cor}
\begin{proof}

Assume that~$\pi'$ is an injective object of~$\cA_Y = \Ind(\cA_Y^{\fg})$, and let $\pi' \to J$ be an injective envelope in $\cA$.
If~$\pi' \ne J$, then we can find a nonzero finitely generated subobject~$\Pi \subset J/\pi'$.
Letting~$\Theta$ be the preimage of~$\Pi$ in~$J$, we find an extension
\[
0 \to \pi' \to \Theta \to \Pi \to 0,
\]
which is not split, because $\pi' \to \Theta$ is essential.
It thus suffices to prove that, for every $\Pi \in \cA^{\fg}$, we have $\Ext^1_{\cA}(\Pi, \pi') = 0$.

Write~$\pi' = \varinjlim_j \pi'_j$ with~$\pi_j' \in \cA_Y^{\fg}$.
If~$\Pi \in \cA^{\fg}$, 
write~$\widehat \Pi_Y = \quoteslim{k}\Pi_k$ with~$\Pi_k \in \cA_Y^{\fg}$ as in Lemma~\ref{lem:completion description}.
Then, for all $i > 0$, we have isomorphisms
\begin{multline*}
\Ext^i_\cA(\Pi, \pi') = \varinjlim_j \Ext^i_{\cA}(\Pi, \pi'_j) %
= \varinjlim_j \Ext^i_{\cA^{\fg}}(\Pi, \pi'_j) \\
= \varinjlim_{j}\varinjlim_k \Ext^i_{\cA_Y^{\fg}}(\Pi_k, \pi'_j) = \varinjlim_k \Ext^i_{\cA_Y}(\Pi_k, \pi') = 0;
\end{multline*}
these are instances respectively of
Lemma~\ref{lem:ext and colimits},
Lemma~\ref{lem:Yoneda},
Lemma~\ref{lem:ext comparison},
and Lemma~\ref{lem:adjoint and colimits},
and the vanishing is due to the fact that~$\pi'$ is injective in~$\cA_Y$.
This concludes the proof.
\end{proof}

\subsubsection{Ind-completion}
\label{subsubsec:ind-completion}
We now consider the $\Ind$-extension of the adjoint pair of exact functors 
$\bigl(\widehat{(-)}_Y, i_{Y, *}\bigr) : \Pro\cA^{\fg} \to \Pro\cA_Y^{\fg}$, which we denote by the same symbols.
By~\cite[Cor.\ 8.6.8]{MR2182076} is an adjoint pair of exact functors
\[
\bigl(\widehat{(-)}_Y, i_{Y, *}\bigr) : \Ind\Pro\cA^{\fg} \to \Ind\Pro\cA_Y^{\fg}.
\]

\begin{rem}\label{rem: completion is Ind extended}
Since $\cA = \Ind \cA^{\fg}$, we can restrict~$\widehat{(-)}_Y$ to~$\cA$. 
This defines a notion of completion for objects of~$\cA$.
Another possibility would be to consider the left adjoint of the inclusion $\Pro\cA_Y \to \Pro\cA$, which may be regarded as a functor
$\Pro \Ind(\cA^{\fg}) \to \Pro \Ind(\cA_Y^{\fg})$; we will not do this in this paper.
\end{rem}

The following is a version of Lemma~\ref{lem:ext comparison} for this Ind Pro version of completion.

\begin{lemma}\label{completion induces an isomorphism on Ext, Ind Pro}
Let~$Y \subset X$ be a closed subset, and let~$\pi \in \cA, \pi' \in \cA_Y$.
Then the $\Ind$-extended completion functor induces an isomorphism
\begin{equation}\label{eqn: ext comparison for Ind Pro}
\Ext^i_{\cA}(\pi, \pi') \isom \Ext^i_{\Ind \Pro \cA_Y^{\fg}}(\widehat \pi_Y, \pi').\end{equation}
\end{lemma}
\begin{proof}
We can write~$\pi$ as a filtered colimit of finitely generated subobjects with injective transition maps.
Since $\widehat{(-)}_Y$ is exact and preserves filtered colimits, an application of Lemma~\ref{Rlim spectral sequence} shows that it suffices to prove
that~\eqref{eqn: ext comparison for Ind Pro} is an isomorphism when~$\pi$ is finitely generated.
In this case, $\widehat \pi_Y$ is an object of $\Pro(\cA_Y^{\fg})$.
By 
Lemma~\ref{lem:limitExt}~(1),
applied to~$\cA$ and to~$\Ind \Pro \cA_Y^{\fg}$, we may furthermore reduce to the case that~$\pi'$ is finitely generated.
(We refer to~\cite[Section~A.5]{DEGcategoricalLanglands} for a discussion of the set-theoretic foundations of $\Ind \Pro$-completions, which guarantee the applicability of Lemma~\ref{lem:limitExt}~(1)
to $\Ind \Pro \cA_Y^{\fg}$.)

By Lemma~\ref{proExtgroups} %
we have
\[
\Ext^i_{\Ind \Pro \cA_Y^{\fg}}(\widehat \pi_Y, \pi') = \Ext^i_{\Pro \cA_Y^{\fg}}(\widehat \pi_Y, \pi'), 
\]
  and  by Lemma~\ref{lem:Yoneda}, we have
\[\Ext^i_\cA(\pi, \pi') = \Ext^i_{\cA^{\fg}}(\pi, \pi'),\]
 so the claim is a consequence of Lemma~\ref{lem:ext comparison}.
\end{proof}

\subsubsection{Further results for finite~\texorpdfstring{$Y$}{Y}.}
We present some additional results 
that hold in the case when $Y$ is finite.

\begin{lemma}\label{lem:finite Y limitExt}
Assume that $Y$ is finite.
If $\pi$ is an object of~$\cA_Y^{\fg}$,
and if $\quoteslim{m}\pi_m$ is a countably indexed object of $\Pro(\cA_Y^{\fg}),$
then the natural map 
$$\Ext^n_{\Pro(\cA_Y^{\fg})}(\pi, \quoteslim{m}\pi_m ) \to
\varprojlim_m \Ext^n_{\cA_Y^{\fg}}(\pi, \pi_m)$$ %
is an isomorphism for every~$n$.
\end{lemma}
\begin{proof}
The natural map of Yoneda $\Ext$-groups 
\[
\Ext^n_{\cA_Y^{\fg}}(\pi, \pi_m) \to \Ext^n_{\Pro(\cA_Y^{\fg})}(\pi, \pi_m)
\]
is an isomorphism for all~$n, m$, by Lemma~\ref{proExtgroups}.
By Lemma~\ref{lem:limitExt 2}, 
it therefore suffices to show that the $\Ext^n_{\cA_Y^{\fg}}(\pi, \pi_m)$ are of finite $\cO$-length.
Since~$Y$ is finite, the objects of $\cA_Y^{\fg}$ have finite $\cA$-length, by Lemma~\ref{lem:finite sets},
and so it suffices to prove that $\Ext^n_{\cA_Y^{\fg}}(\pi, \pi')$ has finite $\cO$-length for irreducible objects~$\pi, \pi'$ of
$\cA_Y^{\fg}$.
This can be verified after a finite extension of coefficients, by Proposition~\ref{prop:ext and base-change}.
Then the finiteness of~$\Ext^n_{\cA}(\pi, \pi')$ was proved in~\cite{MR3150248}.
We can now apply 
Lemma~\ref{lem: smooth Ext equals admissible Ext} to replace $\Ext^n_{\cA}$ with $\Ext^n_{\cA^{\ladm}}$,
Lemma~\ref{lem: block decomposition} to replace $\Ext^n_{\cA^{\ladm}}$ with $\Ext^n_{\cA^{\ladm}_Y}$, and
Lemma~\ref{lem:f.g. Ext} to replace $\Ext^n_{\cA^{\ladm}_Y}$ with $\Ext^n_{\cA^{\fg}_Y}$.
\end{proof}

Our next result has some formal similarity to Lemma~\ref{lem:ext comparison}, but seems to lie deeper: its proof relies on Propositions~\ref{centreelementsI}
and~\ref{centreelementsII} above, 
and these results require Pa\v{s}k\={u}nas' work~\cite{MR3150248}
on the $p$-adic local Langlands correspondence for their proof.

\begin{prop}
\label{prop:Gamma Y and completion}
Let~$Y$ be a finite subset of~$X$.
If $\pi$ is an  object of $\cA^{\fg}_Y$ and $\pi'$ is an object
of~$\cA^{\fg}$,
then passing  to  the completion along $Y$ induces isomorphisms
\[
\Ext^i_{\cA^{\fg}}(\pi, \pi') \isom \Ext^i_{\Pro(\cA_Y^{\fg})}(\pi, \widehat{\pi}'_Y).
\]
\end{prop}
\begin{proof}
Although the category~$\cA^{\fg}$ does not have enough injectives or projectives, the map in the statement of the proposition 
is well-defined by Corollary~\ref{cor:completion is exact}, viewing the $\Ext^i$-groups as groups of Yoneda extensions.

Since~$Y$ is finite, $\cA_Y^{\fg}$ is a finite length category, and so it decomposes as a direct product of blocks.
Writing~$\pi$ as a direct sum according to this decomposition we see that it suffices to prove the proposition when~$Y = \{x\}$ consists of a single closed point of~$X$, possibly not defined over~$\bF$.

On the other hand, every object of~$\cA^{\fg}$ 
has a finite filtration whose graded pieces 
are quotients of compact inductions of Serre weights.
Applying Corollary~\ref{cor:cofinite length},
and passing to a refinement of the filtration, we thus see that the graded pieces
are either irreducible, 
or isomorphic to $\cInd_{KZ}^G(\sigma)$ for a Serre weight~$\sigma$.
Since~$\pi' \mapsto \widehat{\pi}'_Y$ is exact, it follows by
d\'evissage that it suffices to prove the claim when  $\pi$ is
irreducible, and either $\pi' = \cInd_{KZ}^G(\sigma)$ or~$\pi'$ is irreducible.
When~$\pi'$ is irreducible and not contained in~$\cA_Y^{\fg}$, the result is immediate because both sides vanish. 
If~$\pi'$ is an irreducible object of~$\cA_Y^{\fg}$, the result is a consequence of Lemma~\ref{lem:finite Y limitExt} (using Lemma~\ref{lem: smooth Ext equals admissible Ext} and Lemma~\ref{lem:Yoneda}
to identify $\Ext^i_{\cA^{\fg}}(\pi, \pi')$ with $\Ext^i_{\cA_Y}(\pi, \pi')$, and Lemma~\ref{lem:adjoint and colimits}
to identify $\Ext^i_{\cA_Y}(\pi, \pi')$ with $\Ext^i_{\cA^{\fg}_Y}(\pi, \pi')$).

Assume now that $\pi' = \cInd_{KZ}^G(\sigma)$. 
Definition~\ref{f_Y} for~$Y = \{x\}$ introduces a polynomial $f_x(T) \in \cH(\sigma) \cong \bF[T]$, and by Lemma~\ref{f_sigma preimages}(1), the inverse system
\[
\left \{\pi'_n = \left ( \cInd_{KZ}^G \sigma \right )/f_x^n \right \}
\]
is cofinal in the system defining~$\widehat{\pi}'_Y$.
By Lemma~\ref{f_sigma preimages}(2),~$\pi'_1$ is the largest multiplicity-free quotient of~$\cInd_{KZ}^G \sigma$ contained in~$\cA_Y^{\fg}$.
It now follows from Lemma~\ref{lem:finite Y limitExt} that it suffices to prove that the map
\[
\Ext^i_{\cA^{\fg}}(\pi, \cInd_{KZ}^G(\sigma)) \to \varprojlim_n\Ext^i_{\cA_{x}^{\fg}}(\pi, \pi'_n)
\]
induced by the projections $\cInd_{KZ}^G(\sigma) \to \pi'_n$ is an isomorphism.
By Lemma~\ref{lem:Yoneda} and Lemma~\ref{lem:adjoint and colimits} it suffices to prove this with $\cA^{\fg}$ replaced by~$\cA$ and~$\cA_x^{\fg}$ replaced by~$\cA_x$, 
and by Lemma~\ref{lem: smooth Ext equals admissible Ext}, 
we can further replace~$\cA_x$ with~$\cA$.

Since~$f_x$ is a nonzerodivisor on the compact induction, we have short exact sequences
\[
0 \to \pi'_m \to \pi'_{n+m} \to \pi'_n \to 0
\]
for all positive integers~$m, n$.
The connecting homomorphism of the exact sequence
\begin{equation}\label{inductiondirectlimit}
0 \to \cInd_{KZ}^G(\sigma) \to \cInd_{KZ}^G(\sigma)[1/f_x] \to \varinjlim_n \pi'_n \to 0
\end{equation}
induces a map
\[
\varinjlim_n \Ext^{i}_{\cA}(\pi, \pi'_n) \isom \Ext^{i+1}_{\cA}(\pi, \cInd_{KZ}^G(\sigma)),
\]
which is an isomorphism by Proposition~\ref{prop:localizing compact inductions}, Corollary~\ref{cor:ext vanishing},
and the fact (proved in Lemma~\ref{lem:ext and colimits}) that~$\Ext^i_{\cA}$ commutes with filtered colimits.

Pulling back~(\ref{inductiondirectlimit}) along $\pi'_n \to \varinjlim \pi'_m$, and then pushing forward along~$\cInd_{KZ}^G(\sigma) \to \pi'_m$, we rephrase our goal as proving that the map
\[
\varinjlim_n \Ext^i_{\cA}(\pi, \pi'_n) \to \varprojlim_m \Ext_{\cA}^{i+1}(\pi, \pi'_m)
\]
induced by the connecting homomorphisms of the exact sequences
\[
0 \to \pi'_m \to \pi'_{n+m} \to \pi'_n \to 0
\]
is an isomorphism.

Making a finite unramified extension~$\cO'/\cO$ and applying Proposition~\ref{prop:ext and base-change}, and noting that all terms of the inverse limit have finite~$\cO$-length, we can furthermore assume that~$Y = \{x\}$ is a block of absolutely irreducible representations.
Hence the polynomial~$f_x$ can be taken to be linear in almost all cases: the only exception is for blocks of type~(2) consisting of quotients of~$\cInd_{KZ}^G(\sigma)$ when~$\sigma$ is a twist of~$\Sym^{p-2}$.
Indeed, when $\sigma$ is a twist of~$\Sym^{p-2}$ and $\lambda \ne \pm 1$, 
the representation $\cInd_{KZ}^G(\sigma)/(T-\lambda)(T-\lambda^{-1})$ is multiplicity-free and contained in a single block, 
and so we need to take~$f_x = (T-\lambda)(T-\lambda^{-1})$, which is quadratic in~$T$.
However, the representation~$\pi'_m$ is a direct sum
\[
\pi'_m = \cInd_{KZ}^G(\sigma)/(T-\lambda)^m \oplus \cInd_{KZ}^G(\sigma)/(T-\lambda^{-1})^m,
\]
and the exact sequence $0 \to \pi'_m \to \pi'_{n+m} \to \pi'_n \to 0$ is also a direct sum of exact sequences.
Hence the proposition follows from Proposition~\ref{prop:connecting-isomorphisms} below.
\end{proof} %

\begin{prop}\label{prop:connecting-isomorphisms}
Let~$\sigma$ be a Serre weight and~$\lambda \in \bF$.
Define
\[
\tau_n = \cInd_{KZ}^G(\sigma)/(T-\lambda)^n
\]
and let~$\pi$ be an irreducible object of the block~$x$ containing~$\tau_n$.
Then the connecting homomorphisms of the exact sequence
\[
0 \to \tau_m \xrightarrow{(T-\lambda)^n} \tau_{n+m} \to \tau_n \to 0
\]
induce  isomorphisms
\[
\varinjlim_n \Ext_{\cA}^i(\pi, \tau_n) \to \varprojlim_m \Ext^{i+1}_{\cA}(\pi, \tau_m)
\]
for all~$i \geq 0$.
\end{prop}
\begin{proof}
Let $J_x = \bigoplus_{\pi \in x} J_\pi$ be the direct sum of injective envelopes of all the irreducible objects of~$\cA_x$.
Then the functor 
\[
M_x(-) = \Hom_{\cA}(-, J_x)^\vee
\]
is an equivalence of~$\cA_x$ with the category of discrete right modules over~$E_x = \End_{\cA}(J_x)$.
Hence it suffices to prove the proposition after applying~$M_x(-)$.

Define~$M = \varprojlim_n M_x(\tau_n)$, which is a module over~$E_x$ (in fact, since the~$M_x(\tau_n)$ have finite length, this is actually a compact $E_x$-module).
The Hecke operator~$(T-\lambda)$ induces compatible endomorphisms of all the~$M_x(\tau_n)$, hence also of~$M$.
There is an exact sequence
\[
0 \to M \to M[1/(T-\lambda)] \to \varinjlim_n M_x(\tau_n) \to 0.
\]
Since the groups
\[
\Ext^{i+1}_{E_x}(M_x(\pi), M_x(\tau_m)) \cong \Ext^{i+1}_{\cA_x}(\pi, \tau_m)
\]
have finite dimension over~$\bF$, the same argument as in Lemma~\ref{lem:limitExt 2} implies that the natural map
\[
\Ext^{i+1}_{E_x}(M_x(\pi), M) \isom \varprojlim_m \Ext^{i+1}_{E_x}(M_x(\pi), M_x(\tau_m))
\]
is an isomorphism.
So we obtain a connecting homomorphism
\[
\varinjlim_n \Ext^i_{E_x}(M_x(\pi), M_x(\tau_n)) \to \varprojlim_m \Ext^{i+1}_{E_x}(M_x(\pi), M_x(\tau_m)) 
\]
which can be checked to coincide with the map induced by the connecting homomorphisms of the short exact sequences
\[
0 \to M_x(\tau_m) \to M_x(\tau_{n+m}) \to M_x(\tau_n) \to 0
\]

So it suffices to prove that
\begin{equation}\label{Exttovanish}
E^i = \Ext^i_{E_x}(M_x(\pi), M[1/(T-\lambda)]) = 0
\end{equation}
for all~$i$.
We are going to deduce this from Propositions~\ref{centreelementsI} and~\ref{centreelementsII}, using the fact that this $\Ext^i$-group is a module over the centre~$Z_x$ of~$E_x$ in a unique way under the actions on either factor.

Assume first that~$x$ is a block of type~(1), (2) or~(4).
Then it follows from Proposition~\ref{centreelementsI} that there exists an element~$T_x-\lambda \in Z_x$ inducing the operator~$T-\lambda$ on~$M_x(\tau_m)$ for all~$m$, and so on~$M$ too.
Since~$\pi$ is an absolutely irreducible object of~$\cA_x$ and~$T_x-\lambda = 0$ on~$\tau_1$, we also have~$T_x-\lambda = 0$ on~$\pi$.
Hence the vanishing of~\eqref{Exttovanish} follows from the fact that~$T_x-\lambda$ is simultaneously invertible and zero on~$E^i$.

Assume now that~$x$ has type~(3).
Then~$E^i$ has an invertible endomorphism~$T-\lambda$ arising from the second factor.
Furthermore, by Proposition~\ref{centreelementsII} there are central elements~$z_0, \ldots, z_n \in Z_x$ such that
\[
(T-\lambda)^{n+1}+z_n(T-\lambda)^n + \cdots + z_1(T-\lambda) + z_0 = 0
\]
as endomorphisms of~$M_x(\tau_m)$, hence as endomorphisms of~$M$ and~$E^i$.
In addition, the~$z_i$ are contained in the maximal ideal~$\fm_x$, and so they are zero on~$M_x(\pi)$, which is the only irreducible object of the block.
So they are zero on~$E^i$.
It follows that~$(T-\lambda)^{n+1}$ is simultaneously invertible and zero on~$E^i$, and so~$E^i = 0$.
\end{proof}

We next extend our results to $\Ind$-completions.

\begin{prop}\label{prop:Gamma Y and completion II}
Let~$Y$ be a finite subset of~$X$.
If $\pi$ is an object of $\cA_Y$ and $\pi'$ is an object
of~$\cA$,
then passing to the completion along $Y$ induces isomorphisms
\[
\Ext^n_{\cA}(\pi, \pi') \isom \Ext^n_{\Ind\Pro(\cA_Y^{\fg})}(\pi, \widehat{\pi}').
\]
\end{prop}
\begin{proof}
Since $\widehat{(-)}$ is exact, and $\pi \isoto \widehat\pi$, we obtain a natural transformation
\begin{equation}\label{Gamma Y and completion to prove isomorphism}
\Ext^n_{\cA}(\pi, -) \to \Ext^n_{\Ind\Pro(\cA_Y^{\fg})}(\pi, \widehat{(-)})
\end{equation}
which we need to prove is an isomorphism.
Writing~$\pi = \varinjlim_i \pi_i$ as a filtered colimit of finitely generated subobjects with injective transition maps, and using the spectral sequence 
of Lemma~\ref{Rlim spectral sequence}, we can assume without loss of generality that~$\pi \in \cA_Y^{\fg}$.
We now claim that both sides of~$\eqref{Gamma Y and completion to prove isomorphism}$ commute with filtered colimits.
The claim reduces the proposition to proving that
\[
\Ext^n_{\cA^{\fg}}(\pi, \pi') \to \Ext^n_{\Ind\Pro(\cA_Y^{\fg})}(\pi, \widehat{\pi}')
\]
is an isomorphism when~$\pi \in \cA_Y^{\fg}, \pi' \in \cA^{\fg}$.
The dual of Lemma~\ref{proExtgroups} implies that we can replace $\Ind\Pro(\cA_Y^{\fg})$ with~$\Pro(\cA_Y^{\fg})$.
We then conclude the proof with an application of Proposition~\ref{prop:Gamma Y and completion}.

We now prove the claim.
The fact that $\Ext^n_{\cA}(\pi, -)$ preserves filtered colimits is Lemma~\ref{lem:ext and colimits}, since~$\pi$ is finitely generated.
On the other hand, the dual of Lemma~\ref{lem:staysinjective} shows that we can compute $\Ext^n_{\Ind\Pro(\cA_Y^{\fg})}(\pi, \widehat{(-)})$ 
by a projective resolution of~$\pi$ in~$\Pro(\cA_Y^{\fg})$.
This implies the claim because objects of~$\Pro(\cA_Y^{\fg})$ are compact in~$\Ind \Pro (\cA_Y^{\fg})$, and $\widehat{(-)}$ commutes with filtered colimits.
\end{proof}

\subsubsection{Some analogues for \texorpdfstring{$\cA_V$}{AV} (with \texorpdfstring{$V$}{V} open in \texorpdfstring{$X$}{X}).}%

Suppose that %
$Y \subseteq V$ 
for some open subset $V$ of~$X$. 
Write $Z \coloneq  X \setminus V.$  If $\pi$ and $\pi'$ are objects
of $\cA^{\fg}_Z$ and $\cA^{\fg}_Y$ respectively, then 
$$\Hom_{\Pro(\cA^{\fg}_Y)}(\widehat{\pi}_Y, \pi') = \Hom_{\cA^{\fg}}(\pi,\pi') = 0$$ by
Corollary~\ref{cor:disjoint closed localization lemma}. 
Thus $\widehat{(-)}_Y$ vanishes on $\cA^{\fg}_Z$,   
and so, by Lemma~\ref{lem:adjoint and colimits abstract version}~(3),
the functor $\widehat{\pi}_Y$ factors through $\cA^{\fg}_V$, %
and then extends canonically to a
functor $\Pro(\cA^{\fg}_V) \to \Pro(\cA^{\fg}_Y)$ preserving cofiltered limits.
We use the same notation $\widehat{\text{(--)}}_Y$ for this induced functor, so
that by definition we have 
\begin{equation}
\label{eqn:completing localizations}
{(j_V^*\pi)}^\wedge_Y = \widehat{\pi}_Y,
\end{equation}
for objects $\pi$ of $\cA^{\fg}.$
This equality remains true if $\pi \in \cA$, since~$j_V^*$ and~$(-)^\wedge_Y$ preserve filtered colimits.

\begin{lemma}
\label{lem:completing localizations}
The composite functor 
\[
\cA_V^{\fg} \to \Ind\Pro(\cA_Y^{\fg}), \pi_V \mapsto (j_{V*}\pi_V)^\wedge_Y
\]
takes values in $\Pro(\cA_Y^{\fg})$, and is naturally isomorphic to $\pi_V \mapsto (\pi_V)^\wedge_Y$.
\end{lemma}
\begin{proof}
Precomposing the isomorphism~\eqref{eqn:completing localizations} with~$j_{V*}$, we obtain a natural isomorphism 
$(j_V^*j_{V*}\pi_V)^\wedge_Y \isoto (j_{V*}\pi_V)^\wedge_Y$.
Precomposing this with the inverse of the counit $j_V^* j_{V*}\pi_V \isoto \pi_V$ then concludes the proof of the lemma.
\end{proof}

Recall from Corollary~\ref{cor:localizing from inside}
that the unit of adjunction $\id \to j_{V*} j_V^*$ restricts to an isomorphism 
on~$\cA_Y$, so that $j_V^*$ embeds $\cA_Y$ fully faithfully into $\cA_V$.
Thus, if $\pi_V$ is an object of $\cA^{\fg}_V$ and $\pi'$ is an object
of $\cA^{\fg}_Y$, we find that
\begin{multline*}
\Hom_{\Pro(\cA^{\fg}_Y)}({(\pi_V)}^\wedge_Y,\pi') 
\iso \Hom_{\Pro(\cA^{\fg}_Y)}({(j_{V*} \pi_V)}^\wedge_Y,\pi')
\\
= \Hom_{\Ind\Pro(\cA^{\fg})}(j_{V*} \pi_V, \pi') = \Hom_{\cA}(j_{V*}\pi_V, \pi') \iso \Hom_{\cA}(j_{V*}\pi_V, j_{V*}j_V^* \pi')
\\
= \Hom_{\cA_V}(j_V^*j_{V*}\pi_V, j_V^* \pi') = \Hom_{\cA^{\fg}_V}(\pi_V, j_V^*\pi') = \Hom_{\Pro (\cA^{\fg}_V)}(\pi_V, j_V^*\pi').
\end{multline*}
(The first isomorphism is given by Lemma~\ref{lem:completing localizations},
while, as already noted, the second isomorphism follows from Corollary~\ref{cor:localizing from inside}.
The equalities are instances of the various adjunctions that we've established.)
Since
\[
\widehat{\text{(--)}}_Y : \Pro(\cA_V^{\fg}) \to \Pro(\cA_Y^{\fg}) 
\]
preserves cofiltered limits,
we thus see that it is left adjoint to the fully faithful embedding provided by~$j_V^*$.

We now show that
$\widehat{\text{(--)}}_Y$ on~$\cA^{\fg}_V$ 
satisfies analogues of the results
we proved
for $\widehat{\text{(--)}}_Y$ on~$\cA^{\fg}$. 

\begin{lemma}
The functor
$\widehat{\text{{\em (--)}}}_Y$ %
is exact.
\end{lemma}
\begin{proof}
By~\cite[III.1, Corollaire~1]{Gabrielthesis}, any short exact sequence in $\cA_V^{\fg}$ may be written as the image under $j_V^*$
of a short exact sequence in~$\cA^{\fg}$.  The claim then follows from
Corollary~\ref{cor:completion is exact} and~\eqref{eqn:completing localizations}
\end{proof}

\begin{lemma}
Let $\pi_V$ be an object of $\cA_V^{\fg}$,
and %
write $\displaystyle \widehat{(\pi_V)}_Y = \quoteslim{I} \pi_i,$
where the $\pi_i$  are objects of $\cA^{\fg}_Y$.
If $\pi'$ is an object of $\cA^{\fg}_Y$,
then, for any value of~$n$, there are natural 
isomorphisms
$$\varinjlim_{i} \Ext^n_{\cA^{\fg}_Y}(\pi_i,\pi') \iso \Ext^n_{\Pro(\cA^{\fg}_Y)}(\widehat{(\pi_V)}_Y,\pi') \iso \Ext^n_{\cA^{\fg}_V}(\pi_V,j_V^*\pi').$$
\end{lemma}
\begin{proof}
If we write $\pi_V = j_V^* \pi$
for some object $\pi$ of~$\cA^{\fg}$,
take into account~\eqref{eqn:completing localizations}
and apply Lemma~\ref{lem:ext comparison},
then the lemma follows from the chain of isomorphisms
and adjunctions %
\begin{multline*}
\Ext^n_{\cA^{\fg}}(\pi,\pi')
= \Ext^n_{\cA}(\pi, \pi')
\iso \Ext^n_{\cA}(\pi,j_{V*} j_V^* \pi')
= \Ext^n_{\cA_V}(j_V^*\pi,j_V^*\pi')\\
= \Ext^n_{\cA^{\fg}_V}(j_V^*\pi,j_V^*\pi')
= \Ext^n_{\cA^{\fg}_V}(\pi_V,j_V^*\pi')
\end{multline*}
(the isomorphism being the consequence of Corollary~\ref{cor:localizing from inside}
that we've already noted, the first and third equalities being a consequence of Lemma~\ref{lem:f.g. Ext}, 
and the second equality being a consequence of Lemma~\ref{lem:adjunction-ext-iso-j} below).
\end{proof}

\begin{lemma}\label{lem:adjunction-ext-iso-j}
Let~$V \subset X$ be open, and $\pi, \pi' \in \cA$.
Then there is a natural isomorphism
\[
\Ext^i_{\cA_V}(j^*_V\pi, j^*_V\pi') \isom \Ext^i_{\cA}(\pi, j_{V*}j^*_V\pi').
\] 
\end{lemma}
\begin{proof}
Since~$\cA_V$ is a Grothendieck category, we can compute the $\Ext^i_{\cA_V}$-group via an injective resolution of~$j^*_V \pi'$.
Then the lemma follows from the corresponding statement for~$\Hom$, because~$j_{V*}$ is exact and preserves injectives (being right adjoint to the exact functor~$j_V^*$).
\end{proof}

\subsection{The functors \texorpdfstring{$\Gamma_Y$}{GammaY} and \texorpdfstring{$R^1\Gamma_Y$}{R1GammaY}}%
\label{subsec:local cohomology}
We fix throughout this subsection a closed subset~$Y \subset X$ with open complement~$U$, and omit~$U$ from the notation in the various localization functors.
\begin{defn}\label{defn: Gamma_Y}
We define a functor
$\Gamma_Y: \cA \to \cA_Y$ as follows:  $\Gamma_Y(\pi)$ is the maximal subobject
of $\pi$ lying in $\cA_Y$.  
\end{defn}
Equivalently, $\Gamma_Y(\pi)$ is the kernel
of the natural morphism $\pi \to j_*j^* \pi$. 
The functor $\Gamma_Y$ is right adjoint to the inclusion of $\cA_Y$
in $\cA$; since this latter functor is exact, we see that $\Gamma_Y$
takes injectives in $\cA$ to injectives in $\cA_Y$.
Corollary~\ref{cor:ext comparison} then shows that $\Gamma_Y$ in fact
takes injectives in $\cA$ to injectives in~$\cA$.
We may consider the derived functors $R^i\Gamma_Y$ of $\Gamma_Y$. 
Their computation is facilitated by the following result.

\begin{lemma}
\label{lem:injective surjective}
If $I$ is an injective object of $\cA$,
then the natural map $I \to j_*j^*I$ is surjective.
\end{lemma}
\begin{proof}
As already noted, it follows from Corollary~\ref{cor:ext comparison}
that $\Gamma_Y(I)$ is again an injective object of $\cA$.
Thus the inclusion $\Gamma_Y(I) \hookrightarrow I$ is split, and 
so the image $J$ of $I$ in $j_*j^*I$ is yet again injective.
The inclusion $J \hookrightarrow j_*j^*I$ is thus also split.
This shows that the cokernel~$C$ of~$I \to j_*j^*I$ injects into~$j_*j^*I$.
However, $C$ is an object of~$\cA_Y$, and so the adjoint of the inclusion $C \to j_*j^*I$
is the zero map.
Hence $C \to j_*j^*I$ is the zero map, and so $C = 0$, as desired.
\end{proof}

\begin{cor}
\label{cor:local cohom}
For any object $\pi$ of $\cA$,
the derived functors of $\Gamma_Y$ are computed by the complex
$$\pi \to j_*j^*\pi.$$
In particular, the only non-trivial
higher derived functor is~$R^1\Gamma_Y$.
\end{cor}
\begin{proof}
If $I^{\bullet}$ is an injective resolution of~$\pi$,
then the various $R^{\bullet}\Gamma_Y(\pi)$ are computed
as the cohomology of the kernel of the morphism of complexes
$I^{\bullet} \to j_*j^*I^{\bullet}$.  Lemma~\ref{lem:injective surjective}
shows that this kernel coincides with the cone (up to a shift).  Now since $j_*$
is exact, the complex $j_*j^*I^{\bullet}$ resolves $j_*j^*\pi$,
and so this shifted cone is quasi-isomorphic to the complex
$\pi \to j_*j^* \pi$, proving the lemma.
(Alternatively, one can use the snake lemma to construct a $\delta$-functor with amplitude in~$[0, 1]$ from the cohomology of~$\pi \to j_*j^*\pi$, which can be shown to be effaceable using Lemma~\ref{lem:injective surjective}.)
\end{proof}

Since $\Gamma_Y$ preserves injectives, we obtain, for any objects
$\pi$ of $\cA_Y$ and $\pi'$ of $\cA$, a spectral sequence
\begin{equation}\label{eqn: spectral sequence for Ext against AY}E_2^{p,q} \coloneq  \Ext^p_{\cA_Y}\bigl(\pi, R^q\Gamma_Y(\pi')\bigr)
\implies \Ext^{p+q}_{\cA}(\pi,\pi').\end{equation}
In the case when $\pi' \to j_*j^* \pi'$ is injective, i.e.\ when $\Gamma_Y(\pi') = 0$,
this simplifies to the formula
\begin{equation}\label{eqn:simplified spectral sequence}
\Ext^p_{\cA_Y}\bigl(\pi,R^1\Gamma_Y(\pi')\bigr) = \Ext^{p+1}_{\cA}
(\pi,\pi'),
\end{equation}
which can also be obtained directly by computing the long exact sequence of $\Ext$'s
arising from the short exact sequence
$$0 \to \pi' \to j_*j^*\pi' \to R^1\Gamma_Y(\pi') \to 0,$$
and taking into account Corollary~\ref{cor:ext vanishing}.

Lemma~\ref{lem:f.g. adjoint} shows that if %
$R^1\Gamma_Y(\pi)$ is finitely generated, then in fact $R^1\Gamma_Y(\pi) = 0$.
Thus $R^1\Gamma_Y(\pi)$ is typically {\em not} finitely generated.  Nevertheless,
we can gain some control over it, using the following results.
(By contrast, if~$\pi$ is finitely generated then so is~$\Gamma_Y(\pi)$, by Theorem~\ref{thm:noetherian}, since it is a submodule of~$\pi$.)

\begin{lemma}
\label{lem:R1Gamma on c-Ind of wts}
Suppose that $Y$ is a closed subset of~$X$,
and that $\sigma$ is a Serre weight.
As 
in the statement 
of Proposition~{\em \ref{prop:localizing compact inductions}}, 
write $f_{\sigma}^{-1}(Y) = V(g)$ for some $g \in \cH(\sigma)$.
If $g = 0$, then
$\Gamma_Y(\cInd_{KZ}^G \sigma) = \cInd_{KZ}^G \sigma$
and $R^1\Gamma_Y(\cInd_{KZ}^G \sigma) = 0,$
while if $g \neq 0,$ 
then $\Gamma_Y(\cInd_{KZ}^G \sigma) = 0,$
while 
$R^1\Gamma_Y(\cInd_{KZ}^G \sigma) = 
(\cInd_{KZ}^G \sigma)[1/g]/(\cInd_{KZ}^G \sigma)
.$
\end{lemma}
\begin{proof}
Proposition~\ref{prop:localizing compact inductions}
shows that the exact  sequence
$$0 \to \Gamma_Y(\cInd_{KZ}^G \sigma)
\to \cInd_{KZ}^{G}\sigma
\to j_*j^*\cInd_{KZ}^{G}\sigma
\to R^1\Gamma_Y(\cInd_{KZ}^G \sigma)\to 0$$
reduces to
either
$$0 \to \cInd_{KZ}^G\sigma \to \cInd_{KZ}^G \sigma \to 0 \to 0 \to 0$$
or
$$0 \to 0 \to \cInd_{KZ}^G\sigma \to 
(\cInd_{KZ}^G \sigma)[1/g]
\to  
(\cInd_{KZ}^G \sigma)[1/g]/(\cInd_{KZ}^G \sigma) \to 0,
$$
depending on whether $g = 0$ or not.
\end{proof}

\begin{lemma}
\label{lem:R1Gamma on fl}
If $Y$ is a closed subset of~$X$ and $\pi$ is of finite length,
then $R^1\Gamma_Y(\pi) = 0.$
\end{lemma}
\begin{proof} 
Since $\pi$ is of finite length, it is an object of $\cA_Z$ for some
finite closed subset $Z$ of~$X$, which we write in the form
$Z \coloneq  Z_1 \coprod Z_2,$ 
where $Z_1 = Y \cap Z$ and $Z_2  = Z \setminus Z_1.$ 
Correspondingly we may write $\pi = \pi_1 \oplus \pi_2$ with $\pi_i$  an object of $\cA_{Z_i}^{\fg}$.
By Lemma~\ref{localization of locally admissible objects}, the complexes 
$$0 \to \Gamma_Y(\pi_i) \to \pi_i \to j_*j^* \pi_i \to R^1\Gamma_Y(\pi_i) \to 0$$
($i =1$ or $2$)
then reduce to
$$
0 \to \pi_1 \to \pi_1 \to 0 \to 0 \to 0$$
and
$$0 \to 0 \to \pi_2 \to \pi_2 \to 0 \to 0$$
respectively.
In particular, in either case we have $R^1\Gamma_Y(\pi_i) = 0$, and so $R^1\Gamma_Y(\pi) = 0.$
\end{proof}

We may now prove the following result.

\begin{lemma}
\label{lem:R1Gamma on fg}
If $Y$ is a closed subset of~$X$ and $\pi$ is finitely generated,
then $R^1\Gamma_Y(\pi)$ is an Artinian object of $\cA_Y$.
\end{lemma}
\begin{proof} By the usual d\'evissage,
we reduce to the case when $\pi$ is a quotient of $\cInd_{KZ}^G(\sigma)$.
Since the derived functors of~$\Gamma_Y$ vanish in degree~$>1$, by Corollary~\ref{cor:local cohom}, we obtain a surjection
\[
R^1\Gamma_Y(\cInd_{KZ}^G(\sigma)) \to R^1\Gamma_Y(\pi).
\]
Hence Lemma~\ref{lem:R1Gamma on c-Ind of wts}
reduces us to showing that if $g$ is a non-zero element of $\cH(\sigma)$,
then $(\cInd_{KZ}^G \sigma)[1/g]/(\cInd_{KZ}^G \sigma)$ is Artinian,
which is %
Lemma~\ref{lem: uniserial}.
\end{proof}

\begin{df}\label{defn: pi fl}
If $\pi$ is any object of $\cA$,
then we let $\pi_{\finl}$ denote the maximal subobject of $\pi$ which
is locally of finite length; equivalently, this is the maximal subobject
of $\pi$ which is locally admissible.  
\end{df}

We may write $\pi_{\finl} = \varinjlim_Y \Gamma_Y(\pi)$,
where $Y$ runs over all finite closed subsets of~$X$.  
We then see that this functor has a single non-vanishing higher derived
functor, namely $\varinjlim_Y R^1\Gamma_Y$.

\subsection{More on \texorpdfstring{$\Pro(\cA_Y^{\fg})$}{Pro(AYfg)} for finite~\texorpdfstring{$Y$}{Y}}\label{subsec: more on Pro}
Assuming that~$Y$ is finite, we now extend some of the preceding constructions to the context 
of $\Pro(\cA_Y^{\fg})$, %
since we will need them  for the Beauville--Laszlo-type results
of the next section.
Since $\cA_Y^{\fg}$ is a finite length category, Lemma~\ref{Serre subcategory of Pro} applies to it. 
Hence $\cA_Y^{\fg}$ is a Serre subcategory of $\Pro(\cA_Y^{\fg})$,
and we may form the quotient category  $\Pro(\cA_Y^{\fg})_U \coloneq  \Pro(\cA_Y^{\fg})/
\cA_Y^{\fg}$.
Just as in the case of $\cA$ considered before, we let $j^*:
\Pro(\cA_Y^{\fg})
\to
\Pro(\cA_Y^{\fg})_U$ denote the canonical functor.

\begin{rem}\label{not a Serre subcategory of Pro}
We have just seen that, when~$Y$ is finite, Lemma~\ref{Serre subcategory of Pro} implies that $\cA_Y^{\fg}$ is a Serre subcategory of $\Pro(\cA_Y^{\fg})$.
However, it is no longer true that~$\cA_Y^{\fg}$ is Artinian when~$Y$ is infinite, and indeed $\cA_Y^{\fg} \subset \Pro(\cA_Y^{\fg})$ 
is no longer a Serre subcategory when~$Y$ is infinite.
For example, we can let~$Y$ be an irreducible component of~$X$, and then consider an object of~$\cA_Y^{\fg}$ of the form $\pi = \cInd_{KZ}^G(\sigma)$ for a Serre weight~$\sigma$.
Then the characterization of the essential image of $\cA_Y^{\fg} \to \Pro(\cA_Y^{\fg})$ given in~\cite[Proposition~6.2.1]{MR2182076} implies that the subobject
\[
\quoteslim{n}T^n\cInd_{KZ}^G(\sigma) \to \pi
\]
is not contained in this essential image.
We are grateful to the referee for alerting us to this counterexample.
\end{rem}

By Lemma~\ref{extending quotient functor}, the Ind-extension of~$j^{*}$ (which we denote by the same symbol)
\[
j^*: \Ind \Pro(\cA_Y^{\fg}) \to \Ind \left ( \Pro(\cA_Y^{\fg})_U \right )
\]
is the quotient functor with kernel $\Ind \cA_Y^{\fg} \cong \cA_Y$, which is a localizing subcategory of $\Ind \Pro(\cA_Y^{\fg})$.
It therefore admits a fully faithful right adjoint.
Since this adjoint is not constructed as an $\Ind$-extension, we will denote it by
\[
j_*^{\Ind}: \Ind \left ( \Pro(\cA_Y^{\fg})_U \right ) \to
\Ind \Pro(\cA_Y^{\fg}).
\]

\begin{rem}
We will often consider the restriction of~$j_*^{\Ind}$ to $\Pro(\cA_Y^{\fg})_U$, and denote it by the same symbol.
We emphasize that the target of this restriction is still $\Ind \Pro(\cA_Y^{\fg})$.
\end{rem}

The completion functor
$\widehat{(\text{--})}:
\cA^{\fg} \to \Pro(\cA_Y^{\fg})$
induces a functor
$$\cA_U^{\fg} \coloneq  
\cA^{\fg}/\cA_Y^{\fg} \to \Pro(\cA_Y^{\fg})/\cA_Y^{\fg} \eqcolon 
\Pro(\cA_Y^{\fg})_U,
$$
which we again denote by~$\widehat{(\text{--})}.$
By construction, then, the  diagram
\begin{equation}
\label{eqn:pull-back and completion}
\xymatrix{ 
\cA^{\fg}
\ar^-{j^*}[d]
\ar^-{\widehat{(\text{--})}}[r]
& \Pro(\cA_Y^{\fg}) 
\ar^-{j^*}[d]
\\
\cA^{\fg}_U
\ar^-{\widehat{(\text{--})}}[r]
& \Pro(\cA_Y^{\fg})_U}
\end{equation}
commutes up to natural transformation. 

Passing to Ind-categories, we obtain a completion functor
\[\widehat{(\text{--})}:
\cA_U \to \Ind\Pro(\cA_Y^{\fg})/\cA_Y \eqcolon 
\Ind\Pro(\cA_Y^{\fg})_U.\]
\begin{prop}
\label{prop:push-pull comparison}
Assume that $Y$ is finite.
If $\pi$ is an object of $\cA$, %
then there is a natural isomorphism 
$ \widehat{ j_*j^* \pi} \iso  j^{\Ind}_*j^* \widehat{\pi}.$
\end{prop}
\begin{proof}By the full faithfulness of~$j_{*}$ (which implies that $j^*j_* \to \id$ is an isomorphism), and by the definition of $\widehat{(\text{--})}:\cA_U \to 
\Ind\Pro(\cA_Y^{\fg})_U$ (i.e.\ by the commutativity of
the $\Ind$-extension of~\eqref{eqn:pull-back and completion}), we have a %
sequence of isomorphisms
$$
j^*(\widehat{j_*j^* \pi})\iso \widehat{j^* j_*j^* \pi} \iso \widehat{j^*\pi}\iso j^*\widehat{\pi}.
$$
By the adjunction between $j^*$ and~$j_*^{\Ind}$,
we obtain a morphism
\begin{equation}
\label{eqn:push-pull comparison}
\widehat{j_*j^*\pi} \to j^{\Ind}_*j^*\widehat{\pi},
\end{equation}
which we must show is an isomorphism.
Equivalently, by Lemma~\ref{lem:adjoint image}
(which applies because $\cA_Y$ is localizing in $\Ind \Pro (\cA_Y^{\fg})$), 
we need to prove that $\widehat{j_*j^*\pi}$ is contained in the image of $j^{\Ind}_*$.
To do so, again by Lemma~\ref{lem:adjoint image}, it suffices to prove that for all~$\tau \in \cA_Y$ we have
\[
\Hom_{\Ind \Pro (\cA_Y^{\fg})}(\tau, \widehat{j_*j^*\pi}) = 0 \text{ and } \Ext^1_{\Ind \Pro (\cA_Y^{\fg})}(\tau, \widehat{j_*j^*\pi}) = 0.
\]
This is a consequence of Proposition~\ref{prop:Gamma Y and completion II} and Corollary~\ref{cor:ext vanishing}.
\end{proof}
It follows from Proposition~\ref{prop:push-pull comparison} (and passing to Ind-categories in~\eqref{eqn:pull-back and completion}) that we have a commutative diagram 
\begin{equation}
\label{eqn:pull-back and completion II}
\xymatrix{ \cA \ar^-{\widehat{(\text{--})}}[r] \ar[d]_{j^*} & \Ind\Pro(\cA_Y^{\fg}) \ar[d]^{j^*} \\ \cA_U \ar^-{\widehat{(\text{--})}}[r] \ar[d]_{j_*} & \Ind\Pro(\cA_Y^{\fg})_U \ar[d]^{j_*^{\Ind}} \\ \cA \ar^-{\widehat{(\text{--})}}[r] & \Ind\Pro(\cA_Y^{\fg}) }
\end{equation}

\subsection{Beauville--Laszlo-type gluing for finite~\texorpdfstring{$Y$}{Y}}%
\label{subsec: BL gluing}
We finish this section by proving Theorem~\ref{thm:BL gluing}, which  is an analogue of Beauville--Laszlo gluing in our context. It was suggested to us by the main theorem of~\cite{DEGcategoricalLanglands}, which relates smooth representations of
$\GL_2(\Qp)$ to sheaves on stacks, but it is logically independent of the results of that paper.

Recall that Beauville--Laszlo gluing of sheaves on a scheme  involves  gluing sheaves over formal completions at a closed
locus to sheaves on the complementary open subscheme. In our setting, 
which involves a finite closed subset $Y \subset X$,
the category~$\Pro(\cA_Y^{\fg})$ (and its Ind-completion $\Ind\Pro(\cA_Y^{\fg})$) will play the role of the category of sheaves on the formal completion, and the category~$\cA_U$ will correspond to the category of sheaves on the complementary open. We glue along the ``overlap'' $\Pro(\cA_Y^{\fg})_{U}$ (and its Ind-completion $\Ind\Pro(\cA_Y^{\fg})_{U}=\Ind\Pro(\cA_Y^{\fg})/\cA_Y$).
Restriction of sheaves to the ``overlap'' corresponds in our setting to the functors $j^{*}:\Ind\Pro(\cA_Y^{\fg})\to \Ind\Pro(\cA_Y^{\fg})_U$ and 
$\widehat{(\text{--})}:\cA_U\to \Ind\Pro(\cA_Y^{\fg})_U$.

We are therefore led to consider the category \[\Ind\Pro(\cA_Y^{\fg}) \times_{\Ind\Pro(\cA_Y^{\fg})_U} \cA_U,\]whose objects are triples $(\widehat{\pi}_1 ,\pi_{U,2} ,f: j^{*}\widehat{\pi}_1\isoto \widehat{\pi_{U,2}} )$ %
where $\widehat{\pi}_1$, 
$\pi_{U,2}$ are objects of $\Ind\Pro(\cA_Y^{\fg})$ and $\cA_U$ respectively.
(A morphism $(\widehat{\pi}_1 ,\pi_{U,2} ,f)\to (\widehat{\pi}'_1 ,\pi'_{U,2} ,f')$ is a pair of morphisms $\widehat{\pi}_1\to\widehat{\pi}'_{1}$ ,$\pi_{U,2}\to \pi'_{U,2}$ satisfying an obvious compatibility with ~$f,f'$, see~\cite[\href{https://stacks.math.columbia.edu/tag/003R}{Tag 003R}]{stacks-project}.)

By construction, i.e.\ by virtue of the commutativity of the diagram~\eqref{eqn:pull-back and completion II}, %
the functors $\widehat{(\text{--})}:\cA\to\Ind\Pro(\cA_Y^{\fg})$ %
and %
$j^{*}:\cA\to\cA_U$ %
induce a natural functor \begin{equation}\label{eqn:functor-for-BL}\cA \, \longrightarrow \, \Ind\Pro(\cA_Y^{\fg}) \times_{\Ind\Pro(\cA_Y^{\fg})_U} \cA_U.\end{equation}

In the proof of the next result, we will use without comment that $\cA_Y$ is a Serre subcategory of $\Ind \Pro \cA_Y^{\fg}$ (being the
$\Ind$-completion of $\cA_Y^{\fg} \subset \Pro(\cA_Y^{\fg})$, which by Remark~\ref{not a Serre subcategory of Pro} is a Serre subcategory, since~$Y$ is finite),
and $\Ind \Pro \cA_Y^{\fg}$ is a Serre subcategory of $\Ind \Pro \cA^{\fg}$
(being the $\Ind \Pro$ completion of the Serre subcategory $\cA_Y^{\fg} \subset \cA^{\fg}$).
Furthermore, while~$\cA^{\fg}$ is not a Serre subcategory of $\Pro(\cA^{\fg})$,
it follows fromy~\cite[Cor.\ 8.6.8, Prop.\ 8.6.11]{MR2182076} %
that it is closed under kernels, cokernels and extensions in $\Pro(\cA^{\fg})$,
while$\cA = \Ind(\cA^{\fg})$ is closed under kernels, cokernels and extensions in 
$\Ind \Pro \cA^{\fg}$.
A consequence of this statement, which we will use in the proof of the next result, is that given an exact sequence
\[X_1 \to X_2 \to X_3 \to X_4 \to X_5\] in $\Ind \Pro \cA^{\fg},$ if $X_1, X_2, X_4, X_5$ are objects of~$\cA$, then so too is~$X_3$.

\begin{thm}%
\label{thm:BL gluing}
Suppose that $Y$ is a finite closed subset of~$X$,
with open complement~$U$.
Then the functor~\eqref{eqn:functor-for-BL} 
is an equivalence of categories, which restricts to an equivalence of categories  \begin{equation}\label{eqn:fg-BL}\cA^{\fg} \, \longrightarrow \, \Pro(\cA_Y^{\fg}) \times_{\Pro(\cA_Y^{\fg})_U} \cA_U^{\fg}.\end{equation}
\end{thm}
\begin{proof}
  We will construct a quasi-inverse functor to \eqref{eqn:functor-for-BL}.
We begin by noting that if~$\pi\in\cA$ then since completion is exact, we have a commutative diagram in $\Ind\Pro(\cA^{\fg})$ %
\begin{equation}\label{eqn:completion-of-RGammaY}\begin{tikzcd}
0 \arrow[r] & \Gamma_Y(\pi) \arrow[r]  \arrow[d, "\cong"] & \pi \arrow[r] \arrow[d] & j_*j^* \pi \arrow[r] \arrow[d] & R^1\Gamma_Y(\pi) \arrow[r] \arrow[d, "\cong"] & 0 \\
0 \arrow[r] & \Gamma_Y(\pi) \arrow[r] & \widehat{\pi} \arrow[r] & \widehat{j_*j^*\pi} \arrow[r] & R^1\Gamma_Y(\pi) \arrow[r] & 0
\end{tikzcd}\end{equation}where every term of the first row is in~ $\cA$, and every term of the second row is in $\Ind\Pro(\cA_Y^{\fg})$. 
We claim that the middle square of~\eqref{eqn:completion-of-RGammaY} is Cartesian and coCartesian.
To see this, it suffices to introduce $J \coloneq  \operatorname{im}(\pi \to j_*j^*\pi$), and to observe that the rightmost square in
\[
\begin{tikzcd}
0 \arrow[r] & \Gamma_Y(\pi) \arrow[r]  \arrow[d, "\cong"] & \pi \arrow[r] \arrow[d] & J \arrow[r] \arrow[d] & 0\\
0 \arrow[r] & \Gamma_Y(\pi) \arrow[r] & \widehat{\pi} \arrow[r] & \widehat{J} \arrow[r] & 0,
\end{tikzcd}
\]
resp.\ the leftmost square in
\[
\begin{tikzcd}
0 \arrow[r] & J \arrow[r]  \arrow[d] & j_*j^*\pi \arrow[r] \arrow[d] & R^1\Gamma_Y(\pi) \arrow[r] \arrow[d, "\cong"] & 0\\
0 \arrow[r] & \widehat{J} \arrow[r] & \widehat{j_*j^*\pi} \arrow[r] & R^1\Gamma_Y \arrow[r] & 0
\end{tikzcd}
\]
are Cartesian and coCartesian.
This implies in particular that we can reconstruct ~$\pi$ from the knowledge of~$j_*j^* \pi$ and the morphism $\widehat{\pi} \to \widehat{j_*j^*\pi}$. 
Furthermore, taking into account Proposition~\ref{prop:push-pull comparison}, %
the morphism $\widehat{\pi} \to \widehat{j_*j^*\pi}$ corresponds via adjunction to the natural isomorphism $j^* \widehat{\pi}\isoto \widehat{j^{*}\pi}$ used in the definition of~\eqref{eqn:functor-for-BL}.

Reversing this logic, we define a functor \begin{equation}\label{eqn:inverse-to-BL}\Ind\Pro(\cA_Y^{\fg}) \times_{\Ind\Pro(\cA_Y^{\fg})_U} \cA_U\to \cA\end{equation} as follows.
Given an object $(\widehat{\pi}_1 ,\pi_{U,2} ,f: j^{*}\widehat{\pi}_1\isoto \widehat{\pi_{U,2}} )$ of the source category, by adjunction we obtain a morphism $\widehat{\pi}_1\to j_*^{\Ind}\widehat{\pi_{U,2}} $, which by Proposition~\ref{prop:push-pull comparison} yields a morphism \begin{equation}\label{eqn:adjunction-version-of-gluing-condition}\widehat{\pi}_1\to \widehat{j_{*}\pi_{U,2}}%
.\end{equation}
Writing~$K$ and~$C$ for respectively the kernel and cokernel of~\eqref{eqn:adjunction-version-of-gluing-condition}, 
we may form a commutative diagram in $\Ind\Pro(\cA^{\fg})$ 
\begin{equation}\label{eqn:fabricated-completion-of-RGammaY}\begin{tikzcd}
0 \arrow[r] & K \arrow[r]  \arrow[d, "\cong"] & \pi' \arrow[r] \arrow[d] & j_*\pi_{U,2} \arrow[r] \arrow[d] & C' \arrow[r] \arrow[d, hook] & 0 \\
0 \arrow[r] & K \arrow[r] & \widehat{\pi}_{1} \arrow[r] & \widehat{j_*\pi_{U,2}} \arrow[r] & C \arrow[r] & 0
\end{tikzcd}\end{equation}
whose middle square is (by definition) Cartesian, and whose rows are (by definition) exact.

Note that the arrow $C' \hookrightarrow C$ is injective (as indicated).
We are now going to prove that it is an isomorphism, which implies that the middle square is Cartesian and coCartesian.
Since $f: j^{*}\widehat{\pi}_1\to \widehat{\pi_{U,2}}$ is an isomorphism by hypothesis, both~$K$ and~$C$ are objects of~$\cA_Y$.
It follows that~$C'$ is also an object of~$\cA_Y$.
The commutativity of the rightmost square in~\eqref{eqn:fabricated-completion-of-RGammaY} implies that $C' \hookrightarrow C$ becomes a surjection
after applying $\widehat{(-)}$, but then it must be a surjection, since both $C', C \in \cA_Y$
(and $\widehat{(-)}$ restricts to the identity on this category).
This concludes the proof that $C \isoto C'$ and that the middle square of~\eqref{eqn:fabricated-completion-of-RGammaY}
is Cartesian and coCartesian.

Every object in the first row of~\eqref{eqn:fabricated-completion-of-RGammaY}, except possibly~$\pi'$, 
is now known to be an object of~$\cA$, so ~$\pi'$ is also an object of~$\cA$.
Then the functor ~\eqref{eqn:inverse-to-BL} is defined by sending $(\widehat{\pi}_1 ,\pi_{U,2} ,f: j^{*}\widehat{\pi}_1\isoto \widehat{\pi_{U,2}} )$ to~$\pi'$.

Comparing~\eqref{eqn:completion-of-RGammaY} to~\eqref{eqn:fabricated-completion-of-RGammaY}, it is easy to see the functor \eqref{eqn:inverse-to-BL} is quasi-inverse to the functor~\eqref{eqn:functor-for-BL}, 
which is therefore an equivalence of categories.
More precisely, it is immediate from the definitions that the composite
\[
  \cA\xrightarrow{\eqref{eqn:functor-for-BL}}\Ind\Pro(\cA_Y^{\fg}) \times_{\Ind\Pro(\cA_Y^{\fg})_U} \cA_U \xrightarrow{\eqref{eqn:inverse-to-BL}} \cA
\]is an equivalence of categories.
For the other composition, note firstly that the morphism $\pi' \to j_*\pi_{U, 2}$ in~\eqref{eqn:fabricated-completion-of-RGammaY}
has kernel and cokernel in~$\cA_Y$, and is therefore adjoint to an isomorphism $j^*\pi' \isoto \pi_{U, 2}$.
Similarly, using the five lemma we see that $\pi' \to \widehat \pi_1$ is adjoint to an isomorphism $\widehat \pi' \to \widehat \pi_1$.
We thus find an isomorphism between the diagrams $\widehat \pi' \leftarrow \pi' \to j_*j^*\pi$
and $\widehat \pi_1 \leftarrow \pi' \to j_*\pi_{U, 2}$.
Passing to the pushout, and using that the middle squares in each
of~\eqref{eqn:completion-of-RGammaY} and~\eqref{eqn:fabricated-completion-of-RGammaY}
are coCartesian, we obtain an isomorphism 
\[
(\widehat \pi', j^*\pi', j^*\widehat \pi' \isoto \widehat{j^*\pi'} ) \to (\widehat{\pi}_1 ,\pi_{U,2} ,f: j^{*}\widehat{\pi}_1\isoto \widehat{\pi_{U,2}} )
\]
in~$\Ind\Pro(\cA_Y^{\fg}) \times_{\Ind\Pro(\cA_Y^{\fg})_U} \cA_U$, as required.

Finally, to see that the induced functor~\eqref{eqn:fg-BL} is an equivalence, we only need to show essential surjectivity, and for this it suffices to show
that if~$\pi\in \cA$ is such that $\widehat{\pi}$ is an object of~$\Pro(\cA_Y^{\fg})$ and $j^{*}\pi$ is an object of ~$\cA_U^{\fg}$, then~$\pi\in \cA^{\fg}$. Write $\pi=\colim_i\pi_i$ as a filtered colimit with injective transition maps of finitely generated subobjects. Noting that the exact functors $\widehat{(\text{--})}:\cA\to\Ind\Pro(\cA_Y^{\fg})$ 
and 
$j^{*}:\cA\to\cA_U$ preserve compact objects and filtered colimits, since~\eqref{eqn:functor-for-BL} is an equivalence we see that the identity morphism $\pi\to\pi$ factors through an isomorphism $\pi\isoto \pi_i$ for some~$i$, as required.
\end{proof}

\section{Applications and examples}\label{sec:examples}
\subsection{The structure of smooth representations}\label{subsec: structure of smooth
representations}%
In many of our arguments we have used d\'evissage to reduce to the
cases of irreducible representations, and representations of the form
$\cInd_{KZ}^G\sigma$. We can use our results to make more precise the
way in which these representations interact. We begin with the
following structural results. %
Recall from Definition~\ref{defn: pi fl} that for any~$\pi \in \cA$ we write $\pi_{\finl}$
for the maximal subobject of~$\pi$ which is locally of finite length
(equivalently, locally admissible). 

\begin{prop}
  \label{prop: Krull dimension of A}The category $\cA/\cA^{\ladm}$ is locally finite.%
\end{prop}
\begin{proof}
We need to prove that $\cA/\cA^{\ladm}$ has a set of generators of finite length.
Using the set of generators in the proof of Lemma~\ref{lem:Grothendieckcategory}, 
it thus suffices to prove that $\pi \coloneq  \cInd_{KZ}^G(\sigma)$ is irreducible in~$\cA/\cA^{\ladm}$ for all Serre weights~$\sigma$.
In order to do this, it suffices to prove that every nonzero arrow $\alpha: X \to \pi$ in the $\cA/\cA^{\ladm}$ is surjective.

Write $j^*: \cA \to \cA/\cA^{\ladm}$ for the quotient functor (which is the identity on objects).
By definition of the morphisms in $\cA/\cA^{\ladm}$ (see~\cite[\S III.1]{Gabrielthesis})
there exist a subobject $X' \to X$, a quotient $\pi \to \pi'$, 
and an arrow $\alpha': X' \to \pi'$ in~$\cA$
such that $\coker(X' \to X), \ker(\pi \to \pi') \in \cA^{\ladm}$
and the following diagram in $\cA/\cA^{\ladm}$ commutes:
\[
\begin{tikzcd}
X' \arrow[d, "\sim"] \arrow[r, "j^*(\alpha')"] & \pi'\\
X \arrow[r, "\alpha"] & \pi \arrow[u, "\sim"'].
\end{tikzcd}
\]
It thus suffices to prove that $j^*(\alpha')$ is surjective.
Since, by Lemma~\ref{quotient by polynomial}~(1), $\pi$ does not have any nonzero admissible subobject, 
the map $\pi \to \pi'$ is an isomorphism in~$\cA$.
By Corollary~\ref{cor:cofinite length}, this implies that the cokernel of~$\alpha'$ is admissible, 
and so $j^*(\alpha')$ is surjective, as desired. %
\end{proof}

\begin{rem}%
  \label{rem: Gabriel Krull dimension}By definition (see e.g.\
  \cite[\S 2.4]{MR4200808}), Proposition~\ref{prop: Krull dimension of
    A} says that the Krull--Gabriel dimension of~$\cA$ is~$1$.
\end{rem}%

We have the following more concrete variant of Proposition~\ref{prop:
  Krull dimension of A}. %

\begin{prop}
\label{prop:structure}
If $\pi \in \cA$ is finitely generated, 
then $\pi/\pi_{\finl}$ is a successive extension of representations
isomorphic to submodules of~$\cInd_{KZ}^G \sigma$, for various Serre weights ~$\sigma$.
\end{prop}
\begin{proof}
We suppose that $\pi_{\finl} = 0$, and show that
$\pi$ is a successive extension of submodules of representations of the
form~$\cInd_{KZ}^G \sigma$.  We write $\pi$ as the quotient
of $\cInd_{KZ}^G V$ for some smooth $\cO[KZ]_\zeta$-representation~$V$ of finite $\cO[KZ]$-length,
and argue by induction on the length of~$V$. %
Let $\sigma$ be an irreducible subrepresentation of $V$,
and consider the induced morphism $\alpha: \cInd_{KZ}^G \sigma \to \pi$.
If this morphism is zero, we may replace $V$ by $V/\sigma$,
and thus we may assume that~$\alpha \ne 0$.  Since $\pi_{\finl} = 0$ by assumption,
it follows from Corollary~\ref{cor:cofinite length} that~$\alpha$ is then injective. %
Now consider the short exact sequence
$$0 \to \cInd_{KZ}^G \sigma \to \pi \to \pi/\cInd_{KZ}^G \sigma \to 0$$
Pulling back along the finite length part of the target, we obtain a
short exact sequence
$$0 \to \cInd_{KZ}^G \sigma \to \pi' \to (\pi/\cInd_{KZ}^G \sigma)_{\finl}
\to 0$$
in which the second arrow is furthermore essential, %
since $\pi'_{\finl} \subseteq \pi_{\finl} = 0.$
Lemma~\ref{lem:ext control} 
then shows that $\pi'$ is isomorphic to a submodule of $\cInd_{KZ}^G \sigma$.
By construction, $(\pi/\pi')_{\finl} = 0$, and $\pi/\pi'$ is a quotient
of~$\cInd_{KZ}^G (V/\sigma).$   The result follows by induction.
\end{proof}

\begin{rem}\label{rem: submodules of cInd Serre weight}
If~$\sigma$ is not isomorphic to a twist of~$\Sym^0$ or~$\Sym^{p-1}$,
it follows from Lemma~\ref{lem:cind homs} that every submodule of~$\cInd_{KZ}^G(\sigma)$ is isomorphic to $\cInd_{KZ}^G(\sigma)$, and in fact coincides with the image of a Hecke operator in~$\cH(\sigma)$.
\end{rem}

\subsection{Extensions between irreducible representations and compact inductions}
\label{subsec: extensions between irreps and compact inductions}
It is %
possible to use our results to compute
$\Ext$ groups between irreducible representations and compact
inductions in complete generality.
We begin with the following qualitative statement, whose proof depends upon the main results of~\cite{MR3150248}.

\begin{lem}
  \label{lem: Dospinescu asked about this}If~$\pi$ is of finite length,
  and $\pi'$ is finitely generated,
  then the $\cO$-module
$\Ext^i_{\cA}(\pi,\pi')$
is of finite length for all~$i\geq 0$.
\end{lem}
\begin{proof}
By d\'evissage it suffices to prove the claim when~$\pi'$ has finite length or~$\pi'$ is the compact induction of a Serre weight.
The first case is a consequence of~\cite{MR3150248}, so we assume that~$\pi' = \cInd_{KZ}^G(\sigma)$.
The representation~$\pi$ is an object
  of~$\cA_Y^{\fg}$ for some finite~$Y$, and the spectral sequence~\eqref{eqn: spectral sequence for Ext against AY} implies that
\[
\Ext^{i+1}_\cA(\pi, \pi') \cong \Ext^i_{\cA_Y}(\pi, R^1\Gamma_Y(\pi')).
\]
By Lemma~\ref{lem:R1Gamma on fg}, we know that~$R^1\Gamma_Y(\pi')$ is Artinian, and so it suffices to prove that $\Ext^i_{\cA_Y}(\pi, \tau)$ has finite $\cO$-length whenever~$\tau$ 
is an Artinian object of~$\cA_Y$.
If~$\tau$ is Artinian, then the $\cA$-socle of~$\tau$ has finite length, because a semisimple object of infinite length is not Artinian. 
Hence, by Proposition~\ref{injectiveArtinian}, the injective envelope~$J$ of~$\soc_\cA \tau$ is Artinian.
Writing~$\tau' \coloneq  J/\tau$,
it follows from the exact sequence
\[
0 \to \tau \to J \to \tau' \to 0
\]
that~$\tau'$ is also Artinian, and so by dimension shifting and induction it suffices to prove that~$\Hom_{\cA_Y}(\pi, \tau')$ has finite $\cO$-length whenever~$\tau'$ is Artinian.
By d\'evissage, we can also assume that~$\pi$ is irreducible.
Then the finiteness follows from another application of Proposition~\ref{injectiveArtinian}.
\end{proof}

For extensions of absolutely irreducible representations and compact inductions of Serre weights we can be a lot more precise, as in the following proposition.

\begin{prop}\label{ColmezquestionI}
Let~$\sigma$ be a Serre weight and~$\pi$ an absolutely irreducible object of~$\cA$.
Then $\dim_k \Ext^1_\cA(\pi, \cInd_{KZ}^G(\sigma))\leq 1$, and if it is not zero then one of the following is true:
\begin{enumerate}
\item $\pi$ is a quotient of~$\cInd_{KZ}^G(\sigma)$ by~$T-\lambda$ for some~$\lambda \in \bF$, and the extension is isomorphic to
\[
0 \to \cInd_{KZ}^G(\sigma) \xrightarrow{T-\lambda} \cInd_{KZ}^G(\sigma) \to \pi \to 0.
\]
\item $\pi = \chi \circ \det$, $\sigma \cong \Sym^{p-1} \otimes(\chi \circ \det)|_{KZ}$, and the extension is isomorphic to the preimage of~$\pi$ under
\[
0 \to \cInd_{KZ}^G(\sigma) \to \cInd_{KZ}^G(\Sym^{0} \otimes (\chi \circ \det)|_{KZ}) \to \pi \oplus (\pi \otimes (\nr_{-1} \circ \det)) \to 0.
\]
This preimage is isomorphic to $\cInd_{N}^G(\chi \circ \det)$.
\item $\pi = (\chi \circ \det) \otimes \St$, $\sigma \cong (\chi \circ \det)|_{KZ}$, and the extension is isomorphic to the preimage of~$\pi$ under
\[
0 \to \cInd_{KZ}^G(\sigma) \to \cInd_{KZ}^G(\Sym^{p-1} \otimes (\chi \circ \det)|_{KZ}) \to \pi \oplus (\pi \otimes (\nr_{-1} \circ \det)) \to 0.
\]
\end{enumerate}
\end{prop}
\begin{proof}
Let~$x$ be the block of~$\pi$ and write~$Y \coloneq  \{x\}$.
We begin by showing that $\dim \Ext^1_{\cA}(\pi, \cInd_{KZ}^G(\sigma))\leq 1$.
Since~$\Gamma_Y(\cInd_{KZ}^G(\sigma)) = 0$ the spectral sequence~\eqref{eqn: spectral sequence for Ext against AY} yields an isomorphism
\[
\Hom_\cA(\pi, R^1\Gamma_Y(\cInd_{KZ}^G(\sigma))) \isom \Ext^1_\cA(\pi, \cInd_{KZ}^G(\sigma)).
\]
By Lemma~\ref{lem:R1Gamma on c-Ind of wts}, we know that
\[
R^1\Gamma_Y(\cInd_{KZ}^G(\sigma)) \cong \cInd_{KZ}^G(\sigma)[1/f_Y]/\cInd_{KZ}^G(\sigma),
\]
where~$f_Y$ is the polynomial defined in Definition~\ref{f_Y}.
Then Lemma~\ref{lem: uniserial} implies that the $G$-socle of $R^1\Gamma_Y(\cInd_{KZ}^G(\sigma))$ has multiplicity one, which implies the assertion on dimensions.
(Note that when~$\sigma$ is a twist of~$\Sym^{p-2}$, the polynomial~$f_Y$ need not be irreducible, and so Lemma~\ref{lem: uniserial} does not apply directly. 
However, $f_Y$ is then a product of two distinct linear factors, and so $R^1\Gamma_Y(\cInd_{KZ}^G(\sigma))$ is a direct sum of two objects, with non-isomorphic socle, to which Lemma~\ref{lem: uniserial} applies directly.)

We now assume that 
$\Ext^1_\cA(\pi, \cInd_{KZ}^G(\sigma)) \ne 0$
(so it has dimension~$1$)
and let~$E$ be a nonzero extension class.
To complete the proof of the proposition, we need to give an explicit description of~$E$.

Assume first that $\pi$ is not a character or a Steinberg twist.
In this case, Lemma~\ref{lem: uniserial} implies that the $G$-socle of $R^1\Gamma_Y(\cInd_{KZ}^G(\sigma))$ is isomorphic to a quotient of $\cInd_{KZ}^G(\sigma)$, 
and so if $\Ext^1_\cA(\pi, \cInd_{KZ}^G(\sigma))$ does not vanish then $\pi$ is an absolutely irreducible quotient of~$\cInd_{KZ}^G(\sigma)$.
Thus
$\pi = \cInd_{KZ}^G(\sigma)/(T-\lambda)\cInd_{KZ}^G(\sigma)$ for some~$\lambda \in \bF$, and
since we have already proved that $\dim_k \Ext^1_\cA(\pi, \cInd_{KZ}^G(\sigma)) = 1$, it follows that~$E$ must be the extension described in part~(1) of the proposition.

Assume now that~$\pi$ is a character; the case of Steinberg twists is treated similarly. 
Twisting by~$\chi^{-1} \circ \det$ we can assume that~$\pi$ is trivial. 
Because $\Ext^1_{\cA}(\pi, \cInd_{KZ}^G(\sigma)) \ne 0$ we have $\sigma \in \{\Sym^0, \Sym^{p-1}\}$, since otherwise $f_Y$ is a unit and so $R^1\Gamma_Y(\cInd_{KZ}^G(\sigma)) = 0$.
Now we deduce from Lemma~\ref{lem: uniserial} that
\begin{gather*}
\soc_G(R^1\Gamma_Y(\cInd_{KZ}^G(\Sym^{p-1}))) = \soc_G(\cInd_{KZ}^G(\Sym^{p-1})/(T-1)) \cong 1\\
\soc_G(R^1\Gamma_Y(\cInd_{KZ}^G(\Sym^0))) = \soc_G(\cInd_{KZ}^G(\Sym^0)/(T-1)) \cong \St
\end{gather*}
hence $\Ext_{\cA}^1(\pi, \cInd_{KZ}^G(\sigma)) \ne 0$ implies that~$\sigma = \Sym^{p-1}$. %
To find a representative for the nonsplit extension class it now suffices to consider the exact sequence
\[
0 \to \cInd_{KZ}^G \Sym^{p-1} \xrightarrow{\beta} \cInd_{KZ}^G \Sym^0 \to \pi \oplus (\nr_{-1} \otimes \det) \to 0
\]
defined in Lemma~\ref{lem:cind homs}, and let~$E$ be the preimage of~$\pi$.
Since~$E$ is a submodule of $\cInd_{KZ}^G \Sym^0$, it is a nonsplit extension of~$\pi$ by~$\cInd_{KZ}^G \Sym^{p-1}$
(the non-splitness following from Proposition~\ref{cind subobjects}, which implies that $\cInd_{KZ}^G \sigma$ has no nonzero admissible submodules). %
Finally, $E$ contains the image of the natural map $\cInd_{N}^G(1) \to \cInd_{KZ}^G(\Sym^0)$, whose cokernel is one-dimensional, by~\eqref{tree}.
Hence~$E = \cInd_N^G(1)$.
This concludes the proof.
\end{proof}

\begin{rem}\label{ColmezquestionII}
Let~$Y$ be the block of the trivial representation of~$G$.
Using Lemma~\ref{lem: uniserial} we can refine the conclusion of Proposition~\ref{ColmezquestionI}. 
For example, if~$\tau$ is an extension of~$1$ by~$\St$ which is not a quotient of~$\cInd_{KZ}^G(\Sym^0)$, then
\[
\Hom_{\cA}(\tau, R^1\Gamma_Y(\cInd_{KZ}^G(1))) = 0
\]
and so there are no nonsplit extensions of~$\tau$ by~$\cInd_{KZ}^G(1)$.
Note that $\dim_\bF\Ext^1_{\cA}(1, \St) = 2$ by~\cite[(164)]{MR3150248}, so there exist nonsplit extensions of~$1$ by~$\St$ which are not quotients of $\cInd_{KZ}^G(\Sym^0)$.
\end{rem}

Using our results in Section~\ref{subsec: Bernstein
  centres of blocks} it is also possible to compute $\Ext^i$-groups in the other direction.
We prove a general finiteness result, and then illustrate it by computing the dimension of~$\Ext^1$ for generic Serre weights.

\begin{prop}\label{prop: Colmez backwards extensions}%
Let~$\sigma$ be a Serre weight, and let~$\pi$ be an irreducible object of~$\cA$. 
Then 
\begin{enumerate}
\item $\Ext^n_{\cA}(\cInd_{KZ}^G \sigma, \pi)$ is a finite-dimensional $\bF$-vector space for all~$n \geq 0$, and it vanishes if the completion of~$\cInd_{KZ}^G(\sigma)$ at the block of~$\pi$ is zero.
\item Assume that~$\sigma$ is not a twist of~$\Sym^0, \Sym^{p-2}$ or~$\Sym^{p-1}$, that~$\pi$ is absolutely irreducible, and that the completion of~$\cInd_{KZ}^G(\sigma)$ at the block of~$\pi$ is not zero.
  Then
  \[
\dim \Ext^1_{\mathcal{A}}(\cInd_{KZ}^G\sigma, \pi) = 
\begin{cases} 
2 & \text{if } \pi \text{ is supersingular,} \\
1 & \text{otherwise.}
\end{cases}
\]
\end{enumerate}
\end{prop}
\begin{proof}
We begin with~(1).
Let~$Y$ be the block of~$\pi$. If~$\cInd_{KZ}^G(\sigma)^{\wedge}_Y = 0$ then
\[
\Ext^n_{\cA}(\cInd_{KZ}^G(\sigma), \pi) = 0
\]
for all~$n \geq 0$ by Lemma~\ref{lem:ext comparison}.
Otherwise, by Proposition~\ref{prop:ext and base-change}, to prove that $\Ext^n_{\cA}(\cInd_{KZ}^G \sigma, \pi)$ is finite-dimensional 
we can assume without loss of generality that~$\pi$ is absolutely irreducible.
We claim that~$f_Y$ is nilpotent on~$\Ext^n_{\cA}(\cInd_{KZ}^G(\sigma), \pi)$, and zero if~$Y$ does not have type~(3). Assuming the claim, we conclude the proof of part~(1) as follows.
Let $\tau_1 \coloneq  \cInd_{KZ}^G(\sigma)/f_Y$ be the maximal multiplicity-free quotient of~$\cInd_{KZ}^G(\sigma)$ in the same block of~$\pi$, 
and let $\tau_i \coloneq  \cInd_{KZ}^G(\sigma)/f_Y^i$.
If~$f_Y^j$ acts by~$0$ on $\Ext^n_{\cA}(\cInd_{KZ}^G(\sigma), \pi)$, the long exact sequence associated to
\begin{equation}\label{exact sequence for nilpotent f}
0 \to \cInd_{KZ}^G(\sigma) \xrightarrow{f_Y^j} \cInd_{KZ}^G(\sigma) \to \tau_j \to 0
\end{equation}
implies that we have a surjection
\[
\Ext^n_{\cA}(\tau_j, \pi) \to \Ext^n_{\cA}(\cInd_{KZ}^G(\sigma), \pi) \to 0.
\]Since~$\Ext^n_{\cA}(\tau_j, \pi)$ is finite-dimensional (by Lemma~\ref{lem: Dospinescu asked about this}), so is
~$\Ext^n_{\cA}(\cInd_{KZ}^G(\sigma), \pi)$ as required.

We now prove the claim.
By Lemma~\ref{lem:ext comparison}, for all~$n$ there is an isomorphism
\begin{equation}\label{vanishing of f(T) via colimits}
\varinjlim_i \Ext^n_{\cA_Y}(\tau_i, \pi) \isom \Ext^n_{\cA}(\cInd_{KZ}^G(\sigma), \pi)
\end{equation}
which is equivariant for the action of~$\cH(\sigma)$ on the first factor.
Recall that either~$f_Y = T-\lambda$ for some $\lambda \in \bF$, or~$\sigma$ is a twist of $\Sym^{p-2}$ and~$f_Y = (T-\lambda)(T-\lambda^{-1})$ for some $\lambda \in \bF^\times \setminus \{\pm 1\}$.
In the first case, if~$Y$ does not have type~(3), then by %
Proposition~\ref{centreelementsI} there exists an element~$z$ of the Bernstein centre of~$\cA_Y$ 
inducing the action of~$f_Y \in \cH(\sigma)$ on the first factor of $\Ext^n_{\cA_Y}(\tau_i, \pi)$.
Since~$f_Y$ acts by zero on~$\tau_1$, we find that~$z$ is zero on all irreducible subquotients of~$\tau_1$.
Since these are in the same block as~$\pi$, it follows that~$z$ is zero on~$\pi$, hence on $\Ext^n_{\cA_Y}(\tau_i, \pi)$.
We conclude that~$f_Y$ acts by zero on~$\Ext^n_{\cA_Y}(\tau_i, \pi)$ for all~$i$, and so on~$\Ext^n_{\cA}(\cInd_{KZ}^G(\sigma), \pi)$, by~\eqref{vanishing of f(T) via colimits}.
On the other hand, if~$Y$ has type~(3),
then we deduce from Proposition~\ref{centreelementsII} that there exist~$j \geq 1$ and~$z_0, \ldots, z_j$ in the maximal ideal of the centre of~$\cA_Y$ such that $f_Y = T\pm 1$ satisfies the equation
\[
f_Y^{j+1} + z_jf_Y^j + \cdots + z_0 = 0
\] 
in~$\End_G(\tau_1)$ (note that~$\tau_1 = \pi$).
Since~$z_i$ acts by~$0$ on~$\pi$ for all~$i$, we conclude in the same way as before that~$f_Y^{j+1}$ acts by~$0$ on $\Ext^n_{\cA_Y}(\tau_i, \pi)$ for all~$i$, as desired.
Finally, if $f_Y = (T-\lambda)(T-\lambda^{-1})$ for some $\lambda \in \bF^\times \setminus \{\pm 1\}$, then
\[\tau_i = \cInd_{KZ}^G(\sigma)/(T-\lambda)^i \oplus \cInd_{KZ}^G(\sigma)/(T-\lambda^{-1})^i,\]
and the same argument as before shows that $f_Y$ acts by zero on
\[\Ext^n_{\cA}(\cInd_{KZ}^G(\sigma)/(T-\lambda^{\pm 1})^i, \pi)\]
for all~$i$.
This concludes the proof of part~(1)

We now prove part~(2).
If~$\sigma$ is not a twist of~$\Sym^0, \Sym^{p-2}$ or~$\Sym^{p-1}$, and $\cInd_{KZ}^G(\sigma)^{\wedge}_Y \ne 0$, then~$Y$ does not have type~(3).
Furthermore, $f_Y = T-\lambda$ for some~$\lambda \in \bF$, and~$\tau_1$ is absolutely irreducible. 
So~\eqref{exact sequence for nilpotent f} with~$j=1$ yields an exact sequence
\begin{multline*}
0 \to \Hom_{\cA}(\tau_1, \pi) \to \Hom_\cA(\cInd_{KZ}^G(\sigma), \pi) \to \Hom_\cA(\cInd_{KZ}^G(\sigma), \pi) \to \\ 
\to \Ext^1_{\cA}(\tau_1, \pi) \to \Ext^1_{\cA}(\cInd_{KZ}^G(\sigma), \pi) \to 0.
\end{multline*}
The proposition follows from this sequence together with the computation of $\Ext^1$-groups between absolutely irreducible objects of~$\cA$.
More specifically, we know that $\Ext^1_{\cA}(\pi, \pi)$ has dimension~$3$ if~$\pi$ is supersingular~\cite[Theorem~10.13]{Paskunasextensions} and dimension~$2$ otherwise~\cite[Section~8]{MR3150248}.
On the other hand, if~$\tau_1$ is in the same block of~$\pi$ but is not isomorphic to~$\pi$ then~$\pi$ is not supersingular, 
the block~$Y$ of~$\pi$ has type~(2), and $\dim \Ext^1_{\cA}(\tau_1, \pi) = 1$, see~\cite[Section~8]{MR3150248}.

\end{proof}

\begin{cor}\label{finiteness in the other direction}
Let~$\pi, \pi' \in \cA$.
Assume that~$\pi$ is finitely generated and~$\pi'$ has finite length.
Then $\Ext^i_\cA(\pi, \pi')$ has finite $\cO$-length for all~$i \geq 0$.
\end{cor}
\begin{proof}
By d\'evissage, we can assume without loss of generality that~$\pi'$ is irreducible, and that~$\pi$ is irreducible or isomorphic to $\cInd_{KZ}^G(\sigma)$ for some Serre weight~$\sigma$.
Then the claim follows from Lemma~\ref{lem: Dospinescu asked about this} and Proposition~\ref{prop: Colmez backwards extensions}~(1). 
\end{proof}

The finiteness part of Proposition~\ref{prop: Colmez backwards extensions} can be used to prove the following counterpart to Corollary~\ref{cor:ext vanishing}.

\begin{lemma}\label{almost orthogonal}
Let~$\pi \in \cA$, let~$Y \subset X$ be a finite closed subset, and let~$\tau \in \cA^{\fg}_Y$.
Let~$U = X \setminus Y$.
Then
\[
\Ext^i_{\cA}(j_{U*}j_U^*\pi, \tau) = 0
\]
for all~$i$.
\end{lemma}
\begin{proof}
Write $\pi = \varinjlim_{i \in I} \pi_i$ as a filtered colimit of finitely generated subobjects with injective transition maps. 
Since~$j_{U*}$ and~$j_U^*$ are exact and preserve filtered colimits,
we have $j_{U,*}j_{U}^*\pi = \varinjlim_{i \in I} j_{U,*}j_{U}^*\pi_i$, and the filtered system $\{i \in I: j_*j^*\pi_i\}$ has injective transition maps.
Hence Lemma~\ref{Rlim spectral sequence} applies, and we obtain a spectral sequence
\[
E_2^{ij} = \operatorname{R}^i \varprojlim_{i \in I}\Ext^j_{\cA}(j_{U,*}j_{U}^*\pi_i, \tau) \Rightarrow \Ext^{i+j}_{\cA}(j_{U*}j_U^*\pi, \tau).
\]
Hence it suffices to prove the vanishing under the additional assumption that~$\pi$ is finitely generated.
By d\'evissage, we then reduce to the case where~$\tau$ is irreducible, and $\pi = \cInd_{KZ}^G(\sigma)$ or~$\pi$ has finite length.

If~$\pi$ has finite length, then~$\pi \in \cA^{\ladm}$, and so
$j_{U, *}j_U^* \pi = \bigoplus_{x \in U} \pi_x$, by Lemma~\ref{localization of locally admissible objects}.
Since $U$ and~$Y$ are disjoint, the vanishing 
$\Ext^i_{\cA}(j_{U*}j_U^*\pi, \tau) = 0$
is then an immediate consequence of the block decomposition of $\cA^{\ladm}$, in the form of Lemma~\ref{lem: block decomposition}. 
If $\pi = \cInd_{KZ}^G(\sigma)$ then
\[
j_{U*}j_U^*\pi \cong \varinjlim_{\times f} \pi
\]
for the polynomial~$f = f_Y \in \cH(\sigma)$ defined in Definition~\ref{f_Y}, and Lemma~\ref{Rlim spectral sequence} yields a spectral sequence
\[
E_2^{pq} = \operatorname{R}^p\varprojlim_{\times f}\Ext_{\cA}^q(\pi, \tau) \Rightarrow \Ext_{\cA}^{p+q}(j_{U*}j_U^*(\pi), \tau).
\]
By Proposition~\ref{prop: Colmez backwards extensions}, the $\F$-vector space $\Ext_{\cA}^q(\pi, \tau)$ is finite-dimensional, and so multiplication by~$f$ is a nilpotent endomorphism of this vector space, by Lemma~\ref{lem:another ext vanishing}.
Hence $E_2^{pq} = 0$ for all~$p, q$.
This concludes the proof. 
\end{proof}

\begin{rem}\label{not orthogonal} %
We emphasize that Lemma~\ref{almost orthogonal} fails without the assumption that~$\tau$ is finitely generated: a counterexample is provided by the surjection
\[
\cInd_{KZ}^G(\sigma)[1/T] \to \cInd_{KZ}^G(\sigma)[1/T]/\cInd_{KZ}^G(\sigma).
\]
This fact has significant structural consequences from the viewpoint of~\cite{DEGcategoricalLanglands}, as it prevents the recollements considered there from becoming orthogonal decompositions.
\end{rem}

\subsection{Extensions between compact inductions}\label{subsec:
  extensions between compact inductions}We can also consider the
$\Ext$ groups between full compact inductions. We have the following
general result.

\begin{lemma}
\label{lem:ext constraints}
Let~$\sigma_0$ and~$\sigma_1$ be Serre weights, and assume they are not twists of~$\Sym^{p-1}$.
If $\Ext_{\cA}^i(\cInd_{KZ}^G \sigma_1,\cInd_{KZ}^G \sigma_0)$ is non-zero for some~$i$,
then 
either $\sigma_0$ and $\sigma_1$ are isomorphic, or they are the Jordan--H\"older
factors of the reduction of a tame type.
\end{lemma}
\begin{proof}%
For~$i = 0, 1$ let~$\tau_i$ be the unique cuspidal type containing~$\sigma_i$ in its semisimplified mod~$p$ reduction.
If~$\tau_0 = \tau_1$ then the conclusion of the lemma is true, since $\sigma_0, \sigma_1$ are the constituents of the reduction of the tame type $\tau_0 = \tau_1$.
Assume that~$\tau_0 \ne \tau_1$.
Then $\sigma_0 \ne \sigma_1$, and we need to prove that if 
$\sigma_0, \sigma_1$ are not constituents of the same principal series type, then
$\Ext_{\cA}^i(\cInd_{KZ}^G \sigma_1,\cInd_{KZ}^G \sigma_0) = 0$ for all~$i$.

By Definition~\ref{f_sigma}, if $\tau_0 \ne \tau_1$ and $\sigma_0, \sigma_1$ are not constituents of the same principal series type then
the image of $f_{\sigma_0}$ 
does not intersect~$X(\tau_1)$, and similarly the image of~$f_{\sigma_1}$ does not intersect~$X(\tau_0)$.
Hence,  
writing~$j_{U_i}$ for the inclusion of the complement~$U_i = X \setminus X(\tau_i)$,
Proposition~\ref{prop:localizing
compact inductions}~(1) shows that the natural maps
\begin{equation} \label{localizetocomplement} 
\cInd_{KZ}^G(\sigma_i) \to j_{U_{1-i}*}j_{U_{1-i}}^*\cInd_{KZ}^G(\sigma_i)\end{equation}
are isomorphisms.
On the other hand, by Definition~\ref{localizingsubcategory}, $\cInd_{KZ}^G(\sigma_i)$ is an object of~$\cA_{X(\tau_i)}$.
Hence
\[
\Ext_{\cA}^i(\cInd_{KZ}^G \sigma_1,\cInd_{KZ}^G \sigma_0) = 0
\]
by Corollary~\ref{cor:ext vanishing}. 
\end{proof}
\begin{rem}
The same argument proves that if $\Ext_{\cA}^i(\cInd_{KZ}^G \Sym^{p-1}, \cInd_{KZ}^G \sigma)$ is non-zero for some~$i$ then $\sigma \in \{\Sym^0, \Sym^{p-1}, \Sym^{p-3} \otimes \det\}$.
\end{rem}

The computation of~$\Ext^0$ (i.e.\ of $\Hom$) between compact inductions of Serre weights is due to Barthel and Livn\'e~\cite{BarthelLivneDuke} and is treated in Lemma~\ref{lem:cind homs}.
In the remainder of this section we explicitly compute some instances
of the $\cH(\sigma_0)\otimes_\F \cH(\sigma_1)$-module  $\Ext_{\cA}^1(\cInd_{KZ}^G \sigma_1,\cInd_{KZ}^G \sigma_0)$.   If we write
$\cH(\sigma_0) = \F[S]$ and $\cH(\sigma_1)=\F[T]$, then
$\cH(\sigma_0)\otimes_{\F}\cH(\sigma_1) = \F[S,T]$. (Note that even if
$\sigma_0$ and $\sigma_1$ are isomorphic, the Hecke operators $S$ and
$T$ give distinct actions on $\Ext^i$, at least {\em a priori}: this is not the same situation as Lemma~\ref{lem:mapstoquotient}.)

\begin{lemma}\label{lem: PS Ext1}
Suppose that $\sigma_0$ and $\sigma_1$ are the Jordan--H\"older factors of the reduction
of an irreducible principal series type. Then
$\Ext_{\cA}^1(\cInd_{KZ}^G \sigma_1,\cInd_{KZ}^G \sigma_0)$  
is one-dimensional, spanned by the class of~$\cInd_{KZ}^G \tau$,
where~$\tau$ denotes the non-split extension of $\sigma_1$ by
$\sigma_0$. %
As an $\cH(\sigma_0)\otimes_\F \cH(\sigma_1)$-module it is isomorphic
to $\F[S,T]/(S,T).$ 
\end{lemma}
\begin{proof}
Recall from Definition~\ref{f_sigma} that the images of~$f_{\sigma_0}, f_{\sigma_1}$ intersect in a single point, corresponding to the supersingular representation
\[
\pi = (\cInd_{KZ}^G \sigma_1)/T (\cInd_{KZ}^G \sigma_1) \cong (\cInd_{KZ}^G \sigma_0)/S (\cInd_{KZ}^G \sigma_0).
\]
Corollary~\ref{cor:ext vanishing}   
then shows that
$\Ext_{\cA}^i\bigl(\cInd_{KZ}^G \sigma_1,(\cInd_{KZ}^G \sigma_0)[1/S]\bigr) = 0$
for all~$i$,
while Lemma~\ref{lem:ext and colimits} shows that
$$\Ext_{\cA}^i\bigl(\cInd_{KZ}^G \sigma_1,(\cInd_{KZ}^G \sigma_0)[1/S]\bigr)
= 
\Ext_{\cA}^i(\cInd_{KZ}^G \sigma_1,\cInd_{KZ}^G \sigma_0)[1/S];$$
thus we see that $\Ext_{\cA}^i(\cInd_{KZ}^G \sigma_1,\cInd_{KZ}^G \sigma_0)$ consists of
$S$-power torsion elements.

The long exact sequence in $\Ext_{\cA}$ associated to the short exact sequence
$$0 \to \cInd_{KZ}^G \sigma_0 \buildrel S^n \cdot \over \longrightarrow \cInd_{KZ}^G \sigma_0
\to (\cInd_{KZ}^G \sigma_0)/S^n (\cInd_{KZ}^G \sigma_0) \to 0$$
yields an isomorphism
\begin{multline}\label{S-torsion in Ext^1}
\Hom_{\cA}\bigl( \cInd_{KZ}^G \sigma_1 ,
(\cInd_{KZ}^G \sigma_0)/S^n (\cInd_{KZ}^G \sigma_0)\bigr)
\\
\iso S^n\text{-torsion in }
\Ext_{\cA}^1(\cInd_{KZ}^G\sigma_1,\cInd_{KZ}^G\sigma_0).
\end{multline}
When $n = 1$, the quotients
$(\cInd_{KZ}^G \sigma_1)/T (\cInd_{KZ}^G \sigma_1)$
and 
$(\cInd_{KZ}^G \sigma_0)/S (\cInd_{KZ}^G \sigma_0)$ are isomorphic
to~$\pi$,
but the thickenings 
$$(\cInd_{KZ}^G \sigma_1)/T^n (\cInd_{KZ}^G \sigma_1)$$
and
$$(\cInd_{KZ}^G \sigma_0)/S^n (\cInd_{KZ}^G \sigma_0)$$
don't coincide to any higher
order.
More precisely, the elements of~$\Ext^1_{\cA}(\pi, \pi)$ classifying these extensions for~$n = 2$ are linearly independent: see~\cite[Theorem~1.2, Section~7.3]{MR3283706} for a proof of this fact.
Hence Lemma~\ref{lem: uniserial} implies that for all~$n, m$ we have
\begin{multline*}
\Hom_{\cA}((\cInd_{KZ}^G \sigma_1)/T^n (\cInd_{KZ}^G \sigma_1), (\cInd_{KZ}^G \sigma_0)/S^m (\cInd_{KZ}^G \sigma_0)) = \\
\Hom_{\cA}((\cInd_{KZ}^G \sigma_1)/T (\cInd_{KZ}^G \sigma_1), (\cInd_{KZ}^G \sigma_0)/S^m (\cInd_{KZ}^G \sigma_0)),
\end{multline*}
and this space has $\bF$-dimension one.
We now note that
$$\Hom_{\cA}\bigl( \cInd_{KZ}^G \sigma_1 ,
(\cInd_{KZ}^G \sigma_0)/S^n (\cInd_{KZ}^G \sigma_0)\bigr)$$
is one-dimensional, no matter what the value of $n$ is. 
Indeed, any element of this group has to factor through an admissible quotient of $\cInd_{KZ}^G\sigma_1$ in the same block of $\pi$;
by Lemma~\ref{f_sigma preimages}, it has to factor through $\cInd_{KZ}^G \sigma_1/T^m (\cInd_{KZ}^G \sigma_1)$ for some~$m$, and so through $\cInd_{KZ}^G \sigma_1/T (\cInd_{KZ}^G \sigma_1)$, by the discussion above.
Together with~\eqref{S-torsion in Ext^1}, this implies that the $S$-power torsion module
$\Ext^1_{\cA}(\cInd_{KZ}^G\sigma_1,\cInd_{KZ}^G\sigma_0)$
is one-dimensional, and annihilated by $S$. We now prove that it is also annihilated by~$T$.

If $g \in \cH(\sigma_1)$ is coprime to $T$ (i.e.\ does not vanish at $T = 0$),
then Corollary~\ref{cor:ext vanishing} shows that
$\Ext^i_{\cA} \bigl ( (\cInd_{KZ}^G\sigma_1)/g(\cInd_{KZ}^G \sigma_1),\cInd_{KZ}^G\sigma_0 \bigr ) = 0$
for all~$i$.  A consideration of the long exact $\Ext$ sequence associated to
the short exact sequence
$$0 \to \cInd_{KZ}^G \sigma_1 \buildrel g \cdot \over \longrightarrow \cInd_{KZ}^G \sigma_1
\to (\cInd_{KZ}^G \sigma_1)/g (\cInd_{KZ}^G \sigma_1) \to 0$$
then shows that $g$ acts invertibly on 
$\Ext_{\cA}^1(\cInd_{KZ}^G\sigma_1,\cInd_{KZ}^G\sigma_0)$, and thus this Ext module is annihilated 
by $T$, since it has dimension one over~$\F$.

It remains to show that if~$\tau$ denotes the non-split extension of $\sigma_1$ by
$\sigma_0$, then $\cInd_{KZ}^G \tau$ is a non-split
extension. Suppose otherwise.
Then the surjection
$$\cInd_{KZ}^G \sigma_0 \to (\cInd_{KZ}^G \sigma_0)/S (\cInd_{KZ}^G \sigma_0)$$
would extend to a surjection
$$\cInd_{KZ}^G \tau  \to
(\cInd_{KZ}^G \sigma_0)/S (\cInd_{KZ}^G \sigma_0),$$
inducing a $KZ$-equivariant embedding $\tau \hookrightarrow 
(\cInd_{KZ}^G \sigma_0)/S (\cInd_{KZ}^G \sigma_0).$
However, the $K_1$-invariants of
$(\cInd_{KZ}^G \sigma_0)/S (\cInd_{KZ}^G \sigma_0)$
are known to be a direct sum of the mod~$\varpi$ reduction of two cuspidal types
(see for example \cite[Prop.\ 20.1]{BreuilPaskunas} and the explicit formulas in~\cite[\S 16]{BreuilPaskunas})
and so there is no such embedding.
So $\cInd_{KZ}^G \tau$ is indeed non-split.
\end{proof}%

\begin{rem}\label{rem: PS computation gives weight cycling Hecke operator}
  Here is another viewpoint on this computation.  If we let
  $\tau'$ denote the non-split extension of $\sigma_0$ by~$\sigma_1$,
  then there is an isomorphism
  $\cInd_{KZ}^G \tau \iso \cInd_{KZ}^G \tau'$: indeed, the source is
  isomorphic to $\cInd_{IZ}^G \chi$ (for some character $\chi$) and
  the target is isomorphic to $\cInd_{IZ}^G \chi^s$, and these are
  well-known to be isomorphic.  Furthermore, the composite
  $$\cInd_{KZ}^G \sigma_0 \to \cInd_{KZ}^G \tau
\iso \cInd_{KZ}^G \tau' \to \cInd_{KZ}^G \sigma_0$$ is equal to
the Hecke operator~$S$.  Thus we see that the element of the extension group
$\Ext_{\cA}^1(\cInd_{KZ}^G \sigma_1, \cInd_{KZ}^G \sigma_0 )$ classified by
$\cInd_{KZ}^G \tau$ is annihilated by~$S$.
\end{rem}

\begin{rem}
  \label{rem: splitting after inverting T} 
Let~$\pi_i = \cInd_{KZ}^G(\sigma_i)$ where~$\sigma_0, \sigma_1$ are as in Lemma~\ref{lem: PS Ext1}. 
If we push out the extension $\cInd_{KZ}^G\tau$ along the inclusion
$$\pi_0 \hookrightarrow \pi_0[1/S],$$ then the resulting extension {\em does}
split.
Indeed, if~$\rho$ is the cuspidal type containing~$\sigma_1$, then $\pi_1$ lies in~$\cA_{X(\rho)}$.
On the other hand, Proposition~\ref{prop:localizing compact inductions}
shows that $\pi_0[1/S] = j_{U*}j^*_U (\pi_0)$, where $U = X \setminus X(\rho)$.
So we see that $\Ext_{\cA}^1(\pi_1,\pi_0[1/S]) = 0$.

Concretely, this means that there is a $KZ$-equivariant embedding
$\tau \hookrightarrow \pi_0[1/S]$
extending the canonical embedding $\sigma_0 \hookrightarrow \pi_0 \hookrightarrow
\pi_0[1/S].$   This corresponds to the fact that $\pi_0[1/S]$ is a family
of principal series representations of Serre weight $\sigma_0$.
\end{rem}

\begin{lemma}%
\label{tamecuspidalExt1}
If $\sigma_0$, $\sigma_1$, are the constituents of a tame cuspidal
type, then there are isomorphisms %
\begin{multline*}
\Ext_{\cA}^i(\cInd_{KZ}^G \sigma_1, \cInd_{KZ}^G \sigma_0)
\\
\iso
\Ext_{\cA}^i\bigl((\cInd_{KZ}^G \sigma_1)[1/T], \cInd_{KZ}^G \sigma_0\bigr)
\\
\iso
\Ext_{\cA}^i\bigl(\cInd_{KZ}^G \sigma_1, (\cInd_{KZ}^G \sigma_0)[1/S]\bigr)
\\
\iso
\Ext_{\cA}^i\bigl((\cInd_{KZ}^G \sigma_1)[1/T], (\cInd_{KZ}^G \sigma_0)[1/S]\bigr).
\end{multline*}
\end{lemma}
\begin{proof}
Let~$Y$ be the block of $(\cInd_{KZ}^G \sigma_1)/T(\cInd_{KZ}^G \sigma_1)$, and let~$U = X \setminus Y$.
Then Proposition~\ref{prop:localizing compact inductions} implies that $\cInd_{KZ}^G \sigma_0 \cong j_*j^*\cInd_{KZ}^G \sigma_0$.
Hence Corollary~\ref{cor:ext vanishing} implies that
$\Ext^i_{\cA}\bigl( (\cInd_{KZ}^G \sigma_1)/T(\cInd_{KZ}^G \sigma_1), 
\cInd_{KZ}^G \sigma_0 \bigr) = 0$ for all $i$,
so that multiplication by $T$ induces an isomorphism from
$\Ext^i_{\cA}(\cInd_{KZ}^G \sigma_1, \cInd_{KZ}^G \sigma_0)$ 
to itself.  
Writing $(\cInd_{KZ}^G\sigma_1)[1/T]$ as the colimit
of $\cInd_{KZ}^G\sigma_1$ under multiplication by $T$, 
we obtain from Lemma~\ref{Rlim spectral sequence} a spectral sequence
\begin{multline*}
E_2^{p,q} \coloneq  \operatorname{R}^p\varprojlim_{T \cdot} \Ext^q_{\cA}(\cInd_{KZ}^G \sigma_1,
\cInd_{KZ}^G \sigma_0 ) \\ \implies
\Ext_{\cA}^{p+q}\bigl( (\cInd_{KZ}^G \sigma_1)[1/T], \cInd_{KZ}^G \sigma_0\bigr).
\end{multline*}
Recalling that for an inverse system of surjections, the higher
derived inverse limits all vanish, we obtain the first claimed isomorphism of the lemma.

By Corollary~\ref{cor:another ext vanishing}, we have
\[
\Ext_{\cA}^i(\cInd_{KZ}^G \sigma_1,  (\cInd_{KZ}^G \sigma_0)/S (\cInd_{KZ}^G \sigma_0)) = 0
\]
for all~$i$, 
hence~$S$ is invertible on $\Ext_{\cA}^i(\cInd_{KZ}^G \sigma_1, \cInd_{KZ}^G \sigma_0)$. 
Taking the colimit of multiplication by~$S$ and applying Lemma~\ref{lem:ext and colimits}, we obtain the isomorphism 
\[
\Ext_{\cA}^i(\cInd_{KZ}^G \sigma_1, \cInd_{KZ}^G \sigma_0) \cong \Ext_{\cA}^i\bigl(\cInd_{KZ}^G \sigma_1, (\cInd_{KZ}^G \sigma_0)[1/S]\bigr).
\]

Finally, the exact sequence
\[
0 \to \cInd_{KZ}^G \sigma_1 \to \cInd_{KZ}^G \sigma_1[1/T] \to \varinjlim_n (\cInd_{KZ}^G \sigma_1)/T^n\cInd_{KZ}^G \sigma_1 \to 0
\]
implies that
\[
\Ext_{\cA}^i(\cInd_{KZ}^G \sigma_1, \cInd_{KZ}^G \sigma_0[1/S]) \cong \Ext_{\cA}^i\bigl(\cInd_{KZ}^G \sigma_1[1/T], (\cInd_{KZ}^G \sigma_0)[1/S]\bigr),
\]
because
\[
\Ext_{\cA}^i(\varinjlim_n (\cInd_{KZ}^G \sigma_1)/T^n\cInd_{KZ}^G \sigma_1,  (\cInd_{KZ}^G \sigma_0)[1/S]) = 0
\]
for all~$i$, by Corollary~\ref{cor:ext vanishing}.
\end{proof}

Note that the localizations $(\cInd_{KZ}^G \sigma_i)[1/T]$ are algebraic families of principal series over the multiplicative group~$\bG_m$, a fact that goes back to~\cite{BarthelLivneDuke}.
It follows that the Ext groups in Lemma~\ref{tamecuspidalExt1} can be computed using the techniques of
\cite[\S 6-7]{BreuilPaskunas} to relate these~$\Ext^1$ of compact inductions to certain $\Ext^1$ of Hecke modules.
The result is the following analogue of Lemma~\ref{lem: PS Ext1}, which would also follow by the method of~\cite[Section~4]{MR2667883}, if the $\delta$-functor of derived ordinary parts 
had an extension to non-admissible representations.
It can also be proved using the the techniques of~\cite{Heyerderived}.
However, from the optic of this paper, a more natural way to prove Lemma~\ref{lem: lemma we won't prove about cuspidal types}
is as a consequence of our expected results on Bernstein centres explained in Section~\ref{subsec:
  Bernstein centre}, %
so we don't give a proof here.
\begin{lemma}\label{lem: lemma we won't prove about cuspidal types}
Let $\sigma_0$, $\sigma_1$ be the constituents of a tame cuspidal
type, and let~$\tau$ be the nonsplit extension of~$\sigma_1$ by~$\sigma_0$.
Then as an $\F[S,T]$-module,
$\Ext_{\cA}^1(\cInd_{KZ}^G \sigma_1, \cInd_{KZ}^G \sigma_0)$
is free of rank one over
$\F[S,T]/(ST-1)$,
generated by the class of $\cInd_{KZ}^G \tau$.
\end{lemma}

\begin{rem}\label{rem: atomes automorphes remark}
In the context of Lemma~\ref{lem: lemma we won't prove about cuspidal types}, both compact inductions are objects of~$\cA_{X(\tau)}$ for the same cuspidal type~$\tau$. 
The compact induction $\cInd_{KZ}^G \tau$ is supported on the entirety of $X(\tau) \cong \bP^1$, in the sense that it lies in $\cA_{X(\tau)}$ and has irreducible quotients lying in every block parameterized by the closed points of $X(\tau)$. 
Indeed, its quotient $\cInd_{KZ}^G \sigma_1$ has irreducible quotients lying over the whole of $f_{\sigma_1}(\bA^1)$; the remaining point $X(\tau) \setminus f_{\sigma_1}(\bA^1)$ corresponds to the supersingular quotient $(\cInd_{KZ}^G \sigma_0)/T(\cInd_{KZ}^G \sigma_0)$, which contains a copy of $\tau$ and is thus also a quotient of $\cInd_{KZ}^{G} \tau$.

  If we let $U \cong \bG_m$ denote the complement of the marked points in~$X(\tau)$, then
  $j_{U*}j^*_U(\cInd_{KZ}^G \tau)$ ``lies over'' $U$ in
  an intuitive sense, and gives a family of {\em atomes automorphes}.
  It is an extension of
$$j_{U*}j^*_U(\cInd_{KZ}^G \sigma_1) = (\cInd_{KZ}^G \sigma_1)[1/T]$$
by
$$j_{U*}j^*_U(\cInd_{KZ}^G \sigma_0) = (\cInd_{KZ}^G \sigma_0)[1/S].$$
(Both the equalities follow from Proposition~\ref{prop:localizing
  compact inductions}.)  

To see the relationship between Lemma~\ref{lem: lemma we won't prove about cuspidal types} and the Bernstein centre, note that the equality~$ST = 1$ in the endomorphism ring of 
\[
\Ext_{\cA}^1(\cInd_{KZ}^G \sigma_1[1/T], \cInd_{KZ}^G \sigma_0[1/S])
\]
implies that there exists an automorphism~$T^+$ of $j_{U*}j^*_U(\cInd_{KZ}^G \tau)$ which induces~$S^{-1}$ on the subobject and~$T$ on the quotient.
We expect the automorphism~$T^+$ to be induced by the Bernstein centre of $\cA_U$.
Taking the cokernel of~$T^+-\lambda$ for~$\lambda \in \kbase^\times$ corresponds to taking the fibre at~$\lambda \in \bG_m$, and yields the corresponding \emph{atome automorphe}.
\end{rem}

\appendix
\section{Category-theoretic background}\label{sec: cat theory background}
In this appendix we recall various results about abelian categories
and their localizations. Much of this material goes back to Gabriel's
thesis~\cite{Gabrielthesis}, and most of the rest of it can be found
in~\cite{MR2182076}. Some of the results that we need from the
previous two references are collected in~\cite[\S2]{MR4200808}, which
we sometimes refer to for convenience.

In order to be able to use these references we need to fix a %
Grothendieck universe, which we do without further comment. Sets are
\emph{small} if they belong to this fixed universe, and all limits and
colimits are assumed to be small, i.e.\ can be written as (co)limits
over small indexing categories. %

\subsection{Grothendieck and locally Noetherian categories}
We recall that an abelian category~$\cA$ is a \emph{Grothendieck
category} if it 
satisfies (AB5) (which is to say that $\cA$ is cocomplete
and that the formation of filtered colimits in $\cA$ is exact),
and it furthermore admits a set of generators (i.e.\ a small set of objects
$\{G_i: i \in I\}$ with the property that for any nonzero morphism $f:X\to Y$ in
$\cA$, there is a morphism $g:G_i\to X$ for some~$G_i$ such that
$fg\ne 0$). Every object in a Grothendieck category admits an injective envelope~\cite[\S II.6 Thm. 2]{Gabrielthesis}.  %

Recall that an object~$X$ of an abelian category is \emph{Noetherian}, resp.\ \emph{Artinian}, if it
satisfies the ascending chain condition on subobjects, resp.\ the descending chain condition on subobjects.
An object is~\emph{finite} if it is both Noetherian and Artinian.
An abelian category is called {\em Noetherian}, resp.\ {\em Artinian}, if so are all its objects.
An abelian category is called {\em locally Noetherian}, resp.\ \emph{locally finite} if it is
Grothendieck, and furthermore admits a set of generators that are
Noetherian, resp.\ finite.
Recall also that an object~$X$ is 
\emph{compact} if $\Hom(X,-)$ commutes with filtered colimits.

\begin{prop}
  \label{prop: properties of locally Noetherian categories}Suppose
  that~$\cA$ is a locally Noetherian abelian category. Then
  \begin{enumerate}
  \item\label{item: compact iff Noetherian} An object of~$\cA$ is
    compact if and only if it is Noetherian.
  \item\label{item: filtered colimit of injective is injective}A
    filtered colimit of injective objects of~$\cA$ is injective.
  \item\label{item: Exts commute with colimits in the right hand
      variable}If $X$ is a Noetherian object of~$\cA$, then for each
    $n\ge 0$ the functor $\Ext^n_{\cA}(X,-)$ commutes with filtered colimits.
  \end{enumerate}
\end{prop}
\begin{proof}
For~\eqref{item: compact iff Noetherian}, note that any Noetherian object of~$\cA$ is compact by \cite[\S2.4
  Cor.\ 1]{Gabrielthesis}. Conversely if~$X$ is compact then by writing
  $X=\varinjlim_iX_i$ as a filtered colimit of Noetherian objects, we see
  that we can factor the identity morphism $1_X$ through some~$X_i$;
  so~$X$  is a retract of a Noetherian object and is thus itself
  Noetherian. Part~\eqref{item: filtered colimit of injective is
    injective} is part of \cite[\S2.4
  Cor.\ 1]{Gabrielthesis}. Part~\eqref{item: Exts commute with colimits
    in the right hand variable} is~\cite[Prop.\ 2.7]{MR4200808}; alternatively, it is a special case of Lemma~\ref{lem:limitExt}~(1) below.
\end{proof}

Proposition~\ref{prop: properties of locally Noetherian categories}\eqref{item: Exts commute with colimits in the right hand variable} is about $\Ext$ and colimits in the second variable. 
We will also make use of the following standard result about $\Ext$ and colimits in the first variable, which does not require Noetherian assumptions.
(The assumption in Lemma~\ref{Rlim spectral sequence} that the transition maps of~$F$ are monomorphisms can be relaxed, making use of the construction in \cite[Chapter~5]{MR0407091},
but we won't need this more general result.)

\begin{lemma}\label{Rlim spectral sequence}
Let~$\cA$ be a Grothendieck abelian category. %
Let~$I$ be a filtered small set, and $F: I \to \cA$ a diagram whose transition maps are monomorphisms.
Then there is a spectral sequence
\[
E_2^{pq} = \mathrm{R}^p\varprojlim{}_{i \in I}\Ext^q_{\cA}(F(i), -) \Rightarrow \Ext^{p+q}_{\cA}(\varinjlim{}_{i \in I}F(i), -).
\]
\end{lemma}
\begin{proof}
Since
\[
\Hom_{\cA}(\varinjlim{}_{i \in I}F(i), -) \cong \varprojlim{}_{i \in I}\Hom_{\cA}(F(i), -),
\]
it suffices to prove that for every injective object~$\cI \in \cA$, the system 
\[
\left ( \Hom_{\cA}(F(i), \cI) \right )_{i \in I} \in \Fun(I^{\op}, \Mod(\bZ)) 
\]
is acyclic for $\varprojlim{}_{i \in I}$.
By~\cite[Thm.\ 1.8]{MR0407091}, a sufficient condition for acyclicity is \emph{weak flasqueness}, 
in the sense of~\cite[Lem.\ 1.3, Definition]{MR0407091}.
This means that it suffices to check that if $J \subset I$ is a filtered subset, then
\[
\lim_{i \in I} \Hom_{\cA}(F(i), \cI) \to \lim_{j \in J}\Hom_{\cA}(F(j), \cI)
\]
is surjective.
However, this map is isomorphic to
\[
\Hom_{\cA}(\colim_{i \in I}F(i), \cI) \to \Hom_{\cA}(\colim_{j \in J}F(j), \cI),
\]
which is surjective because~$\cI$ is injective and (under our assumptions on the transition maps) 
$\colim_{j \in J}F(j) \to \colim_{i \in I}F(i)$ is a monomorphism.
\end{proof}

\subsection{Localizing categories}\label{app: localizing cats}

Suppose that $\cA$ is a Grothendieck category.
If $\cB$ is a Serre subcategory of $\cA$ (i.e.\ a non-empty full
subcategory closed under the formation of subquotients and extensions in $\cA$),
then we may form the quotient category $\cA/\cB$.    We are interested in the question
of when the natural functor $\cA \to \cA/\cB$ has a right adjoint; if such
a right adjoint exists, we say that $\cB$ is a {\em localizing subcategory}.  
This right adjoint, if it exists, is necessarily left-exact: indeed, being a right
adjoint, it preserves all limits. 
Furthermore, the counit of the adjunction is necessarily an isomorphism~\cite[\S III.2 Proposition~3]{Gabrielthesis}, hence the right adjoint is fully faithful if it exists.

Since left adjoints
necessarily preserve colimits, we see that for such a right adjoint to exist,
the Serre subcategory $\cB$ must be closed under the formation of (small) colimits in $\cA$.
Since $\cB$ is closed under forming quotients in $\cA$, this is equivalent
to asking that $\cB$ be closed under the formation of arbitrary direct sums in~$\cA$.
It turns out %
that this condition is also sufficient for $\cB$ to be localizing.  The quotient 
category $\cA/\cB$ is then also a Grothendieck category. (See e.g.\
\cite[Rem.\ 2.21]{MR4200808} for these facts.) %

We note that 
the construction of the adjoint is not too difficult, modulo set-theoretic issues.
Namely, if $A$ is an object of $\cA$, we consider the category of morphisms
$A \to A'$ which project to an isomorphism in~$\cA/\cB$ (equivalently,
whose kernel and cokernel both lie in~$\cB$).    Then taking the colimit
of $A'$ over this category gives the value of the adjoint on $A$.

We assume throughout the rest of this subsection that $\cB$ is a
localizing subcategory of~$\cA$, we let $j^*:\cA \to \cA/\cB$ denote the
natural projection, and we let $j_*:\cA/\cB \to \cA$ denote the right
adjoint to~$j^*$. Then~$j^{*}$ is exact by construction, while as already noted, since $j_*$ is a right adjoint,
it is automatically left-exact.  However, it need not be exact in
general.  Another way to think of this is that the essential image of
$j_*$ is a full subcategory of $\cA$ which is intrinsically an abelian
category, but is not necessarily an abelian subcategory of~$\cA$.
Indeed, this essential image will be an abelian subcategory of $\cA$
if and only if $j_*$ is actually exact.

\begin{rem}%
  \label{rem: where our notation comes from}Our notation is motivated
  by the following example. If $j:U \subseteq X$ is the inclusion
  of a retrocompact open subset in a scheme~$X$, then we have
the adjoint pair $(j^*,j_*): \QCoh(U) \to \QCoh(X)$, which realizes 
$\QCoh(U)$ as a Serre quotient category of $\QCoh(X)$: it is the quotient
of $\QCoh(X)$ by the Serre subcategory $\QCoh_Z(X)$
consisting of quasicoherent sheaves,
all of whose sections are supported on $Z \coloneq  X \setminus U$. (See~\cite[VI.1, Prop.\ 3]{Gabrielthesis}.)
\end{rem}

The essential image of $j_*$ admits the following characterization:

\begin{lemma}
\label{lem:adjoint image}
An object $A$ of $\cA$ lies in the essential image of $j_*$ if and only
if $\Hom_{\cA}(B,A) = \Ext_{\cA}^1(B,A)=0$ for any object $B$ of~$\cB$, if and only if the unit $A \to j_*j^*A$ is an isomorphism.
\end{lemma}
\begin{proof}
This is immediate from~\cite[III.2, Lem.\ 1(b), Lem.\ 2 Cor.]{Gabrielthesis}.  
\end{proof}

We now have the following result, which relates higher Ext's to higher 
direct images of $j_*$.

\begin{lemma}\label{lem: vanishing Ext vanish Ri}
Let $A'$ be an object of $\cA/\cB$, and fix some $n \geq 1$. 
Then the following are equivalent:
\begin{enumerate}
\item $\Ext^i_{\cA}(B,j_*A') = 0$ for all objects $B$ of~$\cB$, and all $0 \leq i \leq n$.
\item $R^i j_* A' = 0$ for all $1 \leq i \leq n-1$.
\end{enumerate}
\end{lemma}
\begin{proof}
Let $A'\hookrightarrow I^{\bullet}$ be an injective resolution (in the category~$\cA/\cB$, which is a Grothendieck category, hence has enough injectives).
Then $j_*I^{\bullet}$ computes the various~$R^ij_*A'$.  Since $j^*j_*I^{\bullet} \iso
I^{\bullet}$, we find that $j^* R^i j_*A' = 0$ for all~$i~>~0$,
and thus that $R^ij_*A'$ is an object of $\cB$ for~$i~>~0$.
In particular, we see that $R^ij_*A' = 0$ if and only if $\Hom_{\cA}(B, R^ij_*A') = 0$
for all objects $B$ of~$\cB$.

Now since $j_*$ is right adjoint to an exact functor, it preserves injectives,
and thus $\Hom_{\cA}(B,j_*I^{\bullet})$ computes
$\RHom_{\cA}(B,j_*I^{\bullet}),$
for any object $B$ of~$\cB$.   But this $\Hom$ is identically zero,
by Lemma~\ref{lem:adjoint image}, and so 
$\RHom_{\cA}(B,j_*I^{\bullet})~=~0.$
On the other hand, we have the usual $E_2$ spectral sequence computing this $\RHom$,
which thus becomes 
\begin{equation}\label{converges to zero}
E_2^{p,q} \coloneq   \Ext^p_{\cA}(B, R^qj_* A') \implies 0.
\end{equation}
We now prove by induction on~$n$ that~(1) implies~(2).
The base case~$n=1$ is true since there is nothing to prove.
We can now assume that $R^qj_* A' = 0$ 
for all $1 \leq q \leq n-2$, and we need to prove $R^{n-1}j_*A' = 0$.
By the conclusion of the previous paragraph, it suffices to prove that $E_2^{0, n-1} = \Hom_{\cA}(B, R^{n-1}j_* A') = 0$ for all~$B \in \cB$.
Since $R^qj_*A' = 0$, all differentials out of $E_r^{0, n-1}$ vanish except possibly $d_n : E_n^{0, n-1} \to E_n^{n, 0}$.
However, Assumption~(1) implies that $E_2^{n, 0} = 0$, hence also~$d_n = 0$, and so $E_2^{0, n-1} = E_\infty^{0, n-1}$.
Since~\eqref{converges to zero} converges to zero, we conclude that~$E_2^{0, n-1} = 0$.

We now prove that~(2) implies~(1), again by induction on~$n$.
The base case~$n = 1$ is true by Lemma~\ref{lem:adjoint image}.
Fix~$B \in \cB$.
By the inductive assumption, it suffices to prove that $E_2^{n, 0} = \Ext^n_\cA(B, j_*A') = 0$, which can be done by a similar argument as in the previous paragraph, 
since all the differentials into $E_r^{n,0}$ vanish, by Assumption~(2).
\end{proof}%

As a corollary, we have the following result.

\begin{cor}
\label{cor:exactness criterion}
The following are equivalent:
\begin{enumerate}
\item  The functor $j_*$ is exact.
\item The derived functor $R^1 j_*$ is identically zero.
\item For all objects $A$ lying in the essential image of~$j_*$,
and for all objects $B$ of~$\cB$, we have
$\Ext^2_{\cA}(B,A) = 0$.
\end{enumerate}
If these equivalent conditions hold, then in fact all the derived functors
$R^ij_*$ ($i \geq 1$) vanish, and $\Ext_{\cA}^i(B,A) = 0$ for all $i \geq 0$
and objects $A$ lying in the essential image of $j_*$ and $B$ lying in~$\cB$.
\end{cor}
\begin{proof}
The equivalence of~(1) and~(2) is immediate.
Lemma~\ref{lem: vanishing Ext vanish Ri} with~$n = 2$ yields that~(2) implies~(3), and the converse implication follows from Lemma~\ref{lem: vanishing Ext vanish Ri} and Lemma~\ref{lem:adjoint image}.
\end{proof}

We also note the following result, which in practice can simplify
the checking of the various vanishing conditions introduced in the preceding
results.

\begin{lemma}
\label{lem:checking on generators}
If $\{B_j\}$ is a system of generators of $\cB$, 
and $A$ is an object of $\cA$, then the following are equivalent:
\begin{enumerate}
\item $\Ext_{\cA}^i( B_j, A) = 0$ for all $B_j$ and all $i \leq n$.
\item $\Ext_{\cA}^i(B, A) = 0$ for all objects $B$ of $\cB$, and all $i \leq n$.
\end{enumerate}
\end{lemma}
\begin{proof} Clearly (2) implies~(1), and so we focus on proving the
converse.
To say that $\{B_j\}$ is a system of generators of $\cB$ is
to say that for any object $B$ of $\cB$, we may find an epimorphism
$ \bigoplus_j B_j^{\oplus I_j} \to B.$
Since $\cB$ is a Serre subcategory which is closed under the formation of colimits 
in~$\cA$, the kernel of this epimorphism is again an object of $\cB$.
A straightforward dimension-shifting argument reduces us to checking
the vanishing of~(2) in the case when 
$ B =  \bigoplus_j B_j^{\oplus I_j}.$
In this case, we immediately compute that
$\Ext_{\cA}^i(B,A) = \prod_j \Ext_{\cA}^i(B_j,A)^{I_j}$.  The vanishing claimed in~(2)
thus follows from the vanishing assumed in~(1).
\end{proof}

We conclude with some results that are specific to the locally Noetherian case.

\begin{lemma}
\label{lem:adjoint and colimits abstract version}%
Suppose that $\cA$ is locally Noetherian.
Then:
\begin{enumerate}
\item
$j_*$ commutes with filtered colimits.
\item
Both $\cB$ and $\cA/\cB$ are locally Noetherian.
\item
An object of $\cB$ is Noetherian if and only if it is Noetherian as an object of~$\cA$,
while an object of $\cA/\cB$ is Noetherian
if and only if it is isomorphic to the image of a Noetherian object of~$\cA$.
Furthermore, the induced functor $\cA^{\fg} \to (\cA/\cB)^{\fg}$ is a Serre quotient with kernel~$\cB^{\fg}$.
\end{enumerate}
\end{lemma} 
\begin{proof}These are immediate from  \cite[\S III.4  Cor.\
  1, Prop.\ 9]{Gabrielthesis} together with~\cite[Prop.\ 2.22]{MR4200808}.
\end{proof}

\begin{lemma}\label{Rj_* commutes with filtered colimits}
Assume that~$\cA$ is locally Noetherian.
Then the derived functors~$R^ij_*: \cA/\cB \to \cA$ commute with filtered colimits.
\end{lemma}
\begin{proof}
Since~$\cA/\cB$ is a Grothendieck abelian category, filtered colimits are exact in~$\cA/\cB$.
Since it is furthermore locally Noetherian, by Lemma~\ref{lem:adjoint and colimits abstract version}(2), filtered colimits of injective objects of~$\cA/\cB$ are injective, by Proposition~\ref{prop: properties of locally Noetherian categories}(2).
So in~$\cA/\cB$ we can form an injective resolution of a filtered colimit by taking a colimit of injective resolutions.
This reduces the claim to the case~$i = 0$, which is proved in Lemma~\ref{lem:adjoint and colimits abstract version}(1).
\end{proof}

\subsection{Pro-categories (and Ind-categories)}\label{subsec: pro
  categories}
We now recall some standard properties of $\Pro$- and
$\Ind$-categories. It suffices to develop the theory of
$\Pro$-categories, because for any category~$\cC$, there is an
equivalence $\Ind(\cC)\cong\Pro(\cC^\op)^\op$; in the body of the
paper we make far more use of $\Pro$-categories than
$\Ind$-categories, so we do not explicitly state results for
$\Ind$-categories in this appendix. (Note also that some of our
references, in particular~\cite{MR2182076}, take the opposite
approach, so the corresponding statements for $\Ind$-categories are
often more readily available in the literature.)

If $\cC$ is a category,
one can define its associated  pro-category $\Pro(\cC)$
to be the category whose objects are the diagrams
$F: I \to \cC$ indexed by a cofiltered small category~$I$,
with the morphisms between two diagrams $F: I \to \cC, G: J \to \cC$ being defined
by the formula
\[
\Hom_{\Pro(\cC)}(F, G) = \varprojlim_J \varinjlim_I \Hom_\cC(F(i), G(j)).
\]
We will sometimes denote the diagram $F: I \to \cC$ via
$\quoteslim{I}  F(i)$,  %
and refer to it as a pro-object of~$\cC$.
The point of this notation is to distinguish the pro-object
from the limit $\varprojlim_I F(i)$ in $\cC$ itself,
if this limit happens to exist.
Often, when employing this notation, we write $X_i$ rather than~$F(i)$,
and so denote the  object  $F$ of $\Pro(\cC)$ as $\quoteslim{I} X_i.$

As a special case of the definition of morphisms,
we note that
if~$X \in \cC$ and $F: I \to \cC$ is a cofiltered diagram,
then the definitions imply that
\begin{equation}
\label{eqn:first hom formula}
\Hom_{\Pro(\cC)}( \quoteslim{I} X_i, X) = \varinjlim_I \Hom_\cC(X_i, X)
\end{equation}
and
\begin{equation}
\label{eqn:second hom formula}
\Hom_{\Pro(\cC)}(X, \quoteslim{I} X_i)
= \varprojlim_I \Hom_\cC(X, X_i).  
\end{equation}

There is  an equivalent definition of $\Pro(\cC)$, in which
the object $\quoteslim{I} F(i)$ is interpreted as the functor %
\begin{equation}
\label{eqn:pro-object as functor}
\varinjlim_i \Hom_{\cC}\bigl(F(i),\text{--})
\end{equation}
on~$\cC$.
To give a little more detail:
$\cC$ admits its co-Yoneda embedding $\cC \hookrightarrow \Fun(\cC, \Sets)^{\op}$
via $X \mapsto \Hom_{\cC}(X,\text{--}).$  
If $F:I \to \cC$ is a diagram indexed by a cofiltered small category,
then we can form the actual projective limit $\varprojlim_I F(i)$  in $\Fun(\cC,\Sets)^{\op}$; 
this yields precisely the functor~\eqref{eqn:pro-object as functor}.
In this way we obtain a functor $\Pro(\cC) \to \Fun(\cC,\Sets)^{\op}$,
which is fully faithful.

This functor $\Pro(\cC) \to \Fun(\cC,\Sets)^{\op}$ is an instance of
a more general construction, coming from the universal mapping
property that $\Pro(\cC)$ satisfies: namely, if $\cD$ is a category 
that admits all cofiltered limits, then pullback of functors 
along the embedding $\cC \hookrightarrow \Pro(\cC)$ induces an equivalence
of categories of functors %
\begin{equation}
\label{eqn:universal pro-property}
\Fun'\bigl(\Pro(\cC),\cD) \iso \Fun(\cC,\cD),
\end{equation}
where the domain denotes the full subcategory of
$\Fun\bigl(\Pro(\cC),\cD)$ consisting  
of those functors that preserve cofiltered limits. (See  \cite[Cor.\ 6.3.2]{MR2182076}.)
The quasi-inverse generalizes the preceding construction:
if $F: \cC \to \cD$ is a functor whose target admits cofiltered
limits, we extend $F$ to $\Pro(\cC)$ via 
$$F(\quoteslim{I} X_i) \coloneq  \varprojlim_I F(X_i).$$
We now assume, until the end of the paper, that~$\cC$ is an abelian category.
Then
one can alternatively define~$\Pro(\cC)$ as the opposite of the category of left-exact covariant functors $\cC \to \mathrm{Ab}$,
the point being that~\eqref{eqn:pro-object as functor} is now abelian group-valued
and left exact. The category $\Pro(\cC)$ is again abelian and has exact cofiltered
limits,
and the natural fully faithful functor $\cC\to\Pro(\cC)$ is exact
(see~\cite[Thm.\ 8.6.5]{MR2182076} for these facts).

\subsubsection{Finite limits}
\label{subsubsec:equalizers}
One useful fact \cite[Cor.\ 6.1.14]{MR2182076} is that any morphism in $\Pro(\cC)$ can be written
as a cofiltered limit, over some small category~$I$,
of morphisms $X_i \to Y_i$ in~$\cC$. 
Similarly, any pair of morphisms having the same domain and codomain may be written
as a cofiltered limit, over some small category~$I$,
of morphism pairs $X_i \rightrightarrows Y_i$ in~$\cC$ (see \cite[Cor.\ 6.1.15]{MR2182076}). 
Then, if $\cC$ admits equalizers, the same is true
of~$\Pro(\cC)$, 
and furthermore we can compute the equalizer of $\quoteslim{I} (X_i
\rightrightarrows Y_i)$ 
as the limit in $\Pro(\cC)$ of the equalizers of the morphisms
$X_i \rightrightarrows Y_i$ (see~\cite[Appendix, Prop.\
4.1]{ArtinMazur}).
Likewise, if~$\cC$ admits finite products, then so does~$\Pro(\cC)$, and these are computed ``pointwise" in the same way as equalizers.
It follows that if~$\cC$ admits finite limits then so does~$\Pro(\cC)$, and then $\Pro(\cC)$ admits {\em all} small limits
(since it admits finite limits and cofiltered limits).
Compare~\cite[Proposition\ 6.1.18]{MR2182076}.

\subsubsection{Finite colimits}
\label{subsubsec:finite colimits}
If $\cC$ admits finite colimits, the same
is true of $\Pro(\cC)$ (see \cite[Cor.\ 6.1.17]{MR2182076}), and finite coproducts and coequalizers can be computed ``pointwise", as for equalizers: this follows again from an application of~\cite[Appendix, Prop.\
4.1]{ArtinMazur} (see also \cite[Prop.\ 6.1.16]{MR2182076}).
More precisely, the coequalizer of  $\quoteslim{I} (X_i) \rightrightarrows Y_i$
is the limit in $\Pro(\cC)$ of the coequalizers of the various $X_i \rightrightarrows Y_i$,
while 
$$\quoteslim{I} (X_i) \coprod \quoteslim{J} (Y_j) = \quoteslim{I\times J} \left(X_i \coprod Y_j\right).$$
In particular, the inclusion $\cC \to \Pro(\cC)$ preserves finite colimits.
Furthermore, finite colimits commute with cofiltered limits:
compare~\cite[Prop.\ 6.1.19]{MR2182076}.

\subsubsection{Pro-adjoints}
\label{subsubsec:pro-adjoints}
Suppose that $F:\cC_1  \to \cC_2$ is a functor preserving finite limits, between
categories each of which admits finite limits.  Then we obtain an induced
functor $\Pro(F): \Pro(\cC_1) \to \Pro(\cC_2)$, which,
by the discussions of Section~\ref{subsubsec:equalizers},   %
also preserves equalizers and finite products, hence finite limits. 
The functor $\Pro(F)$ can also be regarded as
corresponding, under the equivalence~\eqref{eqn:universal pro-property},
to the composite $\cC_1 \to \cC_2 \to \Pro(\cC_2)$, and so is seen to preserve
cofiltered limits (or see \cite[Prop.\ 6.1.9]{MR2182076}).  Thus $\Pro(F)$ preserves arbitrary limits, and hence admits
a left adjoint $G: \Pro(\cC_2) \to \Pro(\cC_1).$ %
(This is a consequence of the special adjoint functor theorem, which applies
because~$\Pro(\cC_i)$ is complete, as recalled in Section~\ref{subsubsec:equalizers}, and has a
cogenerator, see e.g.~\cite[Thm.\ 8.6.5]{MR2182076}.)  %

\subsubsection{Monomorphisms}
The following lemma provides a criterion for testing if a morphism
in~$\Pro(\cC)$ is a monomorphism.

\begin{lemma}
\label{lem:monomorphism criterion}
Suppose given a morphism
$Y \to Z$ in $\Pro(\cC)$ with the property that,
for any morphism $Y \to Y'$ with $Y'$ an object  of~$\cC$, 
we may find a commutative square
\begin{equation}
\label{eqn:monomorphism diagram}
\xymatrix{Y \ar[r]\ar[d] & Z \ar[d] \\ Y' \ar[r] & Z'}
\end{equation}
in $\Pro(\cC)$ in which the
bottom horizontal arrow is a monomorphism in~$\cC$;
then the given morphism $Y \to Z$ is a monomorphism.
\end{lemma}
\begin{proof}
  This is (the equivalent $\Pro$-category version of) \cite[Prop.\ 8.6.9]{MR2182076}.
\end{proof}

\subsubsection{Ext groups.}
We continue to assume that~$\cC$ is an abelian category.
The following results will be used to compare $\Ext$ groups in $\cC$ and~$\Pro(\cC)$.

\begin{lemma}\label{lem:staysinjective}
Let~$\cC$ be an abelian category. Then:
\begin{enumerate}
\item If~$X \in \cC$ is an injective object then it remains injective in~$\Pro(\cC)$.
\item If~$\cC$ has enough injectives, then $\Pro(\cC)$ has enough injectives.
\item If~$\cC$ is small, then $\Ind(\cC)$ has enough injectives, and $\Pro(\cC)$ has enough projectives.
\end{enumerate}
\end{lemma}
\begin{proof}
The second statement is dual to~\cite[Prop.\ 1.3.3.13]{LurieHA}.
We now prove the first statement.
Let~$\alpha: F \to G$ be a monomorphism in~$\Pro(\cC)$.
By~\cite[Appendix, Prop.\ 4.6]{ArtinMazur} 
one can represent~$F$ and~$G$ by diagrams $F: I \to \cC, G: I \to \cC$ from the same index category~$I$ in such a way that~$\alpha$ is 
represented by a natural transformation $\alpha: F \to G$ such that $\alpha(i): F(i) \to G(i)$ is a monomorphism for all~$i \in I$.
Assume given a map $\lambda: F \to X$. 
By definition, it arises from a map $\lambda: F(i) \to X$ for some~$i$.
Since~$\alpha(i)$ is a monomorphism and~$X$ is injective, we can extend~$\lambda$ to a map $G(i) \to X$, which proves that~$X$ is injective in~$\Pro(\cC)$.

Finally, the third statement is a consequence of~\cite[Thm.\ 8.6.5~(vi), Thm.\ 9.6.2]{MR2182076}.
\end{proof}

\begin{lemma}\label{proExtgroups}%
Let~$\cC$ be an abelian category.
If~$X,Y \in \cC$, then the natural maps $\Ext^i_{\cC}(X, Y) \to \Ext^i_{\Pro(\cC)}(X, Y)$
and $\Ext^i_{\cC}(X, Y) \to \Ext^i_{\Ind(\cC)}(X, Y)$
are isomorphisms for all~$i \geq 0$.
\end{lemma}
\begin{proof}
This is a consequence of~\cite[Thm.\ 15.3.1]{MR2182076}.
\end{proof}

\begin{lemma}\label{lem:limitExt}
Let~$\cC$ be a small abelian category.
Let~$X$ be an object of~$\cC$.
\begin{enumerate}
\item The functor $\Ext^i_{\Ind(\cC)}(X, -)$ commutes with filtered colimits. %
\item The functor $\Ext^i_{\Pro(\cC)}(-, X)$ takes cofiltered limits to filtered colimits.
\end{enumerate}
\end{lemma}
\begin{proof}
  Part~(2) is dual to part~(1), which in turn is a consequence of \cite[~15.3.9]{MR2182076} (and its proof).
  A slight variant on that argument is as follows: write 
\[
F \coloneq  \Hom_{\cC}(X, -) : \cC \to \Mod(\bZ),
\] 
so that the functor $\Hom_{\Ind \cC}(X, -)$ is the composite
\[
\Ind(\cC) \xrightarrow{\Ind F} \Ind(\Mod(\bZ)) \xrightarrow{\colim} \Mod(\bZ).
\]
Since~$\colim$ is exact, and preserves colimits (being a left adjoint),
it suffices to prove that the derived functors of $\Ind F$ preserve filtered colimits.
This follows from~\cite[Corollary~15.3.6]{MR2182076}. 
Note that \cite[Assumption (15.3.7)]{MR2182076} holds, because~$\cC$ is small.
\qedhere
\end{proof}

In general,
it's harder to say anything about Ext groups computed in the other direction,
i.e.\ of the form $\Ext^n_{\Pro(\cC)}(X, -),$
but we have the following result, which treats a special case.

\begin{lemma}\label{lem:limitExt 2}%
Assume that $\cC$ is Artinian and $\cO$-linear, for some commutative ring~$\cO$.
Let $X = \quoteslim{i} X_i$ be a countably indexed object of~$\Pro(\cC)$ such that 
the $\cO$-modules $\Ext_{\Pro(\cC)}^n(X,Y)$ are of finite length, for all $n$, and all objects
$Y$ of~$\cC$.
If $\quoteslim{m}Y_m$ is a countably indexed object of $\Pro(\cC)$,
then the natural map 
\begin{equation}
\label{eqn:pro-ext map}
\Ext^n_{\Pro(\cC)}(X, \quoteslim{m}Y_m) \to \varprojlim_m \Ext^n_{\Pro(\cC)}(X, Y_m)
\end{equation}
is an isomorphism for every~$n$.
\end{lemma}
\begin{proof}
Replacing each $Y_m$ by the ``infimum'' of the collection of images of
transition morphisms $Y_{m'}  \to Y_m$  for $m' \geq m$ (since~$\cC$ is Artinian,
this descending  sequence  of images stabilizes) --- which replaces $\quoteslim{m} Y_m$
by an isomorphic pro-object ---
we may, and do, assume that the transition morphisms between  the $Y_m$ are epimorphisms.
Since~$\Pro(\cC)$ has enough projectives (by Lemma~\ref{lem:staysinjective}~(3)),
we can compute~$\displaystyle \Ext^n_{\Pro(\cC)}(X, \quoteslim{m} Y_m )$
by taking a projective resolution $P_\bullet \to X$.

The projectivity of the terms of $P_{\bullet}$,
together with the fact that the transition maps between the $Y_m$ are epic,
implies that the transition maps in the inverse system
of complexes
$\Hom_{\Pro(\cC)}(P_\bullet, Y_m)$
are surjective, and hence have vanishing $\operatorname{R}^1\varprojlim.$
We thus find that
\begin{equation}\label{projectivecomplex}
\Hom_{\Pro(\cC)}(P_\bullet, \quoteslim{m} Y_m)
\iso
\varprojlim_m
\Hom_{\Pro(\cC)}(P_\bullet, Y_m)
\iso
\operatorname{R}\varprojlim_m
\Hom_{\Pro(\cC)}(P_\bullet, Y_m)
\end{equation}
The composite isomorphism gives rise to a spectral sequence
\begin{multline*}
E_2^{p,q} \coloneq   
\operatorname{R}^p\varprojlim_m \Ext^q_{\Pro(\cC)}(X,Y_m)
\implies H^{p+q}\bigl(
\Hom_{\Pro(\cC)}(P_\bullet, \quoteslim{m} Y_m)
\big)
\\
= \Ext^{p+q}_{\Pro(\cC)}(X,\quoteslim{m}Y_m).
\end{multline*}
Of course the $R^p\varprojlim_m$ vanish automatically  for
$p \geq 2$,
but they also vanish for $p = 1$, since the inverse systems
$\Ext^q_{\Pro(\cC)}(X,Y_m)$ are inverse systems of finite length $\cO$-modules
by assumption, and so satisfy the Mittag--Leffler condition.
This proves the lemma.
\end{proof}

\subsubsection{\texorpdfstring{$\cA$}{A} vs.\ \texorpdfstring{$\cA^{\fg}$}{Afg}}
\label{subsubsec:fg cat}
If $\cA$ is a locally Noetherian  abelian  category,
we write $\cA^{\fg}$ for the full subcategory of compact (or equivalently,
by Proposition~\ref{prop: properties of locally Noetherian categories}, Noetherian) 
objects.
The canonical functor $\Ind(\cA^{\fg})\to\cA$
given by evaluating inductive limits in $\cA$ is
then an equivalence, by~\cite[Prop.\ 6.3.4]{MR2182076}.

We can't expect the category $\cA^{\fg}$ to have enough injectives,
but we can define $\Ext^i_{\cA^{\fg}}$ as a Yoneda extension group.  We then
have the following result.

\begin{lemma}
\label{lem:f.g. Ext}
Suppose that~$\cA$ is locally Noetherian.
If $X$ and $Y$ are objects of $\cA^{\fg}$, 
then  $\Ext^i_{\cA^{\fg}}(X,Y) \iso  \Ext^i_{\cA}(X,Y)$
for all~$i$.
\end{lemma}
\begin{proof}
Since $\Ind(\cA^{\fg}) \isoto \cA$, this is a special case of Lemma~\ref{proExtgroups}. \qedhere 

\end{proof}

\begin{rem}\label{rem:Popescu result}
Lemma~\ref{lem:limitExt}~(1) generalizes 
Proposition~\ref{prop: properties of locally Noetherian categories}~\eqref{item: Exts commute with colimits in the right hand variable},
because if $\cA$ is locally Noetherian, then $\Ind(\cA^{\fg}) \isoto \cA$.
On the other hand, Proposition~\ref{prop: properties of locally Noetherian categories}~\eqref{item: filtered colimit of injective is injective}
does not generalize, in the sense that if $\cC$ is a small abelian category (not necessarily Noetherian), 
then filtered colimits of injectives in $\Ind(\cC)$ are not necessarily injective.
In fact,
by \cite[Theorem~5.8.7]{Popescubook}, 
locally Noetherian categories are characterized amongst compactly generated Grothendieck categories by the property that filtered colimits of injective objects remain injective.

Correspondingly, Lemma~\ref{lem:limitExt}~(1) does not imply that filtered colimits of injectives in $\Ind \cC$ are injective.
It \emph{does} imply that if $I\coloneq  \colim_{j \in J}I_j \in \Ind \cC$ is a filtered colimit of injective objects of $\Ind \cC$, 
then $\Ext^i_{\Ind \cC}(-, I)$ is zero on objects of~$\cC$.
But this property need not imply that~$I$ is injective in $\Ind \cC$, 
because 
if $X = \colim_{j \in J}X_j$ and~$X_j \in \cC$, then the derived functors of~$\lim_{j \in J}$ can intervene in the computation
of $\Ext^i_{\Ind \cC}(X, I)$.
\end{rem}

\subsubsection{Completion along an abelian subcategory}
\label{subsubsec:completion}
Suppose that $k_*: \cC_0 \hookrightarrow \cC$ is the inclusion of an exact abelian subcategory (i.e.\ $\cC_0$ is a full subcategory
which is also abelian and for which the inclusion is exact).  
Then we may apply the discussion
of Section~\ref{subsubsec:pro-adjoints} to see that the induced inclusion $k_* : \Pro(\cC_0) \hookrightarrow
\Pro(\cC)$ (which is evidently fully faithful,
and by the discussions of Sections~\ref{subsubsec:equalizers}
and~\ref{subsubsec:finite colimits} is again exact)
admits a left adjoint $\Pro(\cC) \to \Pro(\cC_0)$,
which we refer to as the functor of {\em completion} of $\cC$
along~$\cC_0$,
and which we denote by $X \mapsto \widehat{X}$.

The identity map from $\widehat{X}$ to itself induces, by the adjunction that
defines~$\widehat{X}$, a canonical morphism (the unit of the adjunction)
$X \to k_*\widehat{X}$ in $\Pro(\cC)$.
Concretely, 
the unit is a canonical morphism %
$X \to \widehat{X}$ (in $\Pro(\cC)$) through which any morphism $X \to k_*X_0$,
with $X_0$ an object of~$\cC_0$,
uniquely factors.

The following lemma describes $\widehat{X}$ explicitly for objects $X$ of $\cC$, 
in the case when $\cC_0$ is furthermore closed under the formation of subobjects
in~$\cC$ (e.g.\ a Serre subcategory of~$\cC$).

\begin{lemma}
\label{lem:abstract completion description}
If $\cC_0$ is closed under the formation  of subobjects in~$\cC$, and~$X$ is an
object of~$\cC$,
then there is a natural isomorphism
\begin{equation}
\label{eqn:abstract completion description}
\widehat{X} \iso \quoteslim{}X',
\end{equation}
where $X'$ runs over the cofiltered set %
of quotients of~$X$ 
lying in~$\cC_0$.  
\end{lemma}
\begin{proof}
If $X \to k_* X'$ is a surjection
with $X'$ lying in~$\cC_0$, then by adjunction there is an induced morphism
$\widehat{X}\to X'.$ 
These induced morphisms collectively give rise
to the morphism~\eqref{eqn:abstract completion description}, 
which we claim is an isomorphism.
To see this, it suffices to show that the induced morphism of functors 
\begin{equation}
\label{eqn:completion morphism}
\Hom_{\Pro(\cC_0)}(\quoteslim{} X',\text{--}) 
\to \Hom_{\Pro(\cC_0)}(\widehat{X},\text{--})
\end{equation}
on~$\Pro\cC_0$ is an isomorphism. %
It in turn suffices to check that the restriction of~\eqref{eqn:completion
  morphism} to~$\cC_0$ is an isomorphism. Now, if $Y\in\cC_0 $ is arbitrary, then by adjunction
we have \[\Hom_{\Pro(\cC_0)}(\widehat{X},Y)=\Hom_{\Pro(\cC)}(X,k_*Y)=\Hom_{\cC}(X,k_*Y). \]
On the other hand, we
have \[\Hom_{\Pro(\cC_0)}(\quoteslim{} X',Y)=\varinjlim_{X'}\Hom_{\cC_0
  }(X',Y)=\varinjlim_{X'}\Hom_{\cC
  }(k_*X',k_*Y)=\Hom_{\cC}(X,k_*Y),\] 
and~\eqref{eqn:completion morphism} becomes the identity, as required. 
(In more detail, the first equality follows from the definition of $\Hom_{\Pro(\cC_0)}$, the second equality is a consequence of the full faithfulness of $k_*$,
and the third equality follows from the fact that~$\cC_0$ is closed under the formation of subobjects in~$\cC$.)
\end{proof}

In the context of the preceding lemma, the unit of adjunction
$X \to k_*\widehat{X} = \quoteslim{} k_*X'$ is the morphism induced by the various
quotient morphisms $X  \to  k_*X'$.
It also follows from Lemma~\ref{lem:abstract completion description} that the counit of adjunction 
$X' \to \widehat{k_* X'}$
is an isomorphism for all objects~$X'$ of~$\cC_0$.

\begin{lemma}
\label{lem:abstract ext comparison}
Let $\cC$ be a small abelian category, let $k_*: \cC_0 \to \cC$ be the inclusion of a Serre subcategory,
and assume that the completion functor $\Pro(\cC) \to \Pro(\cC_0)$ is exact.
Let $X$ be an object of $\cC$,
and write 
\[
\displaystyle \widehat{X} = \quoteslim{I} X_i,
\]
as in Lemma~{\em \ref{lem:abstract completion description}}.
If $X'$ is an object of $\cC_0$,
then for any value of~$n$, the natural morphism 
\begin{equation}\label{chain of maps for ext comparison app version}
\Ext^n_{\Pro(\cC_0)}(\widehat{X},X') 
\inverseiso
\varinjlim_{I} \Ext^n_{\cC_0}(X_i,X')
\to \Ext^n_{\cC}(X, k_* X')
\end{equation}
is an isomorphism, whose inverse is the composite
\[
\Ext^n_{\cC}(X, k_*X') \xrightarrow{\widehat{(-)}} \Ext^n_{\Pro (\cC_0)}(\widehat X, \widehat{k_* X'}) \xrightarrow{\mathrm{counit}_*} \Ext^n_{\Pro (\cC_0)}(\widehat X, X').
\]
\end{lemma}
\begin{proof}
We note that the first isomorphism in~\eqref{chain of maps for ext comparison app version} is an instance of Lemma~\ref{lem:limitExt} (bearing in mind Lemma~\ref{proExtgroups}).
The composite of~\eqref{chain of maps for ext comparison app version} and the isomorphism
\[
\Ext^n_{\cC}(X, k_*X') \isoto \Ext^n_{\Pro(\cC)}(X, k_*X'),
\]
of Lemma~\ref{proExtgroups} %
is the map
\begin{equation}\label{n-extensions isomorphism}
\Ext^n_{\Pro(\cC_0)}(\widehat{X},X') \to \Ext^n_{\Pro(\cC)}(X,k_*X'),
\end{equation}
induced by the inclusion $k_*: \Pro(\cC_0) \to \Pro(\cC)$ and the unit $X \to k_*\widehat X$. 
Since~$k_*$ is exact, \eqref{n-extensions isomorphism} 
is part of a morphism of $\delta$-functors on $\Pro(\cC)$
\[
\Ext^n_{\Pro (\cC_0)}\bigl(\widehat {(-)}, X'\bigr) \to \Ext^n_{\Pro(\cC)}(-, k_* X'), 
\]
and it coincides with the adjunction morphism when~$n = 0$.
On the other hand, since $\widehat{(-)}$ is assumed to be exact, the composite
\begin{equation}\label{inverse of n-extension isomorphism}
\Ext^n_{\Pro(\cC)}(-,k_*X') \xrightarrow{\widehat{(-)}} \Ext^n_{\Pro(\cC_0)}\bigl(\widehat{(-)}, \widehat{k_* X'}\bigr) \xrightarrow{\mathrm{counit}_*}
\Ext^n_{\Pro(\cC_0)}\bigl(\widehat{(-)}, X'\bigr)
\end{equation}
is also part of a morphism of $\delta$-functors on $\Pro(\cC)$, and it coincides with the adjunction map in degree zero.
The morphisms~\eqref{n-extensions isomorphism} and~\eqref{inverse of n-extension isomorphism} are therefore mutually inverse in degree zero.
Using the fact that $\Pro(\cC)$ has enough projectives, by Lemma~\ref{lem:staysinjective}~(3), 
and that the completion functor preserves projectives (being left adjoint to the exact functor~$k_*$)
we see that
the $\delta$-functors $\Ext^n_{\Pro(\cC)}(-,k_*X')$ and $\Ext^n_{\Pro(\cC_0)}(\widehat{(-)}, \widehat{k_* X'})$ are effaceable.
Hence~\eqref{n-extensions isomorphism} and~\eqref{inverse of n-extension isomorphism} are inverse in all degrees, as desired. \qedhere
\end{proof}

\subsubsection{Completion and quotients}
In Sections~\ref{subsec: more on Pro} and ~\ref{subsec: BL gluing} we  make use of the following results 
about the interaction of $\Ind$-categories and $\Pro$-categories with Serre quotients.

\begin{lemma}\label{Serre subcategory of Pro}
If~$\cB$ is an Artinian abelian category, then the essential image of $\cB \to \Pro(\cB)$ is a Serre subcategory of~$\Pro(\cB)$.
\end{lemma}
\begin{proof}
By~\cite[Proposition~8.6.11]{MR2182076}, the essential image is %
closed under extensions, kernels and cokernels in~$\Pro(\cB)$.
So it suffices to prove that for all~$\pi \in \cB$, all the subobjects and quotients of~$\pi$ in $\Pro(\cB)$ are isomorphic to objects of~$\cB$.
In fact, since~$\cB$ is closed under cokernels, it suffices to prove the statement about subobjects.
Let $\alpha: \quoteslim{i \in I} \tau_i \to \pi$ be a monomorphism, for some cofiltered index category~$I$.
The definition of $\Hom$-sets in $\Pro(\cB)$ shows that~$\alpha$ factors through a map $\alpha_i : \tau_i \to \pi$ for some~$i$.

Let~$(I/i)$ be the category of morphisms in~$I$ to~$i$.
Then $(I/i) \to I$ is a ``co-cofinal functor'' in the sense of~\cite[Definition~2.5.1]{MR2182076}, as can be seen by checking condition~(iii) in~\cite[Proposition~3.2.2]{MR2182076}.
Hence the natural map
\[\quoteslim{k \in (I/i)} \tau_k \to \quoteslim{j \in I} \tau_j\]
is an isomorphism, by~\cite[Proposition~2.5.2(ii)]{MR2182076}.
Replacing~$I$ with $(I/i)$, and~$\pi$ with the constant pro-object indexed by~$(I/i)$ with value~$\pi$, 
we can therefore assume that~$\alpha$ is the cofiltered limit of maps $\alpha_k: \tau_k \to \pi$.

Since~$\alpha$ is an isomorphism onto its image, and cofiltered limits are exact in~$\Pro(\cB)$, it now suffices to prove that the pro-object 
$\quoteslim{k}(\operatorname{im} \alpha_k)$ is isomorphic to an object of~$\cB$.
For this, it suffices to prove that there exists~$i_0 \in I$ such that for all maps $j \to i_0$, the induced map 
\[
\operatorname{im}\alpha_j \to \operatorname{im}\alpha_{i_0}
\]
is an isomorphism.
This is a consequence of the fact that~$\pi$ is Artinian.
\end{proof}

Assume now that~$\cB$ is an Artinian category. %
We now recall some results of~\cite{MR1426488}, where the compact objects of any additive category~$\cC$ with filtered colimits are called {\em finitely presented},
and form a full subcategory denoted~$\fp(\cC)$.
The essential image of $\Pro \cB \to \Ind(\Pro \cB)$ is the subcategory of compact objects of~$\Ind(\Pro \cB)$ (and similarly for $\cB \to \Ind(\cB)$),
since abelian categories are idempotent-complete.
Hence the categories~$\Ind \Pro \cB$ and~$\Ind \cB$ are {\em locally coherent}, meaning that they are compactly generated, and their full subcategories of compact objects are abelian.
Conversely, if~$\cA$ is a locally coherent category then the natural map $\Ind(\fp \cA) \to \cA$ is an equivalence.

By Lemma~\ref{Serre subcategory of Pro}, $\cB$ is a Serre subcategory of~$\Pro \cB$.
Hence~\cite[Theorem~2.6]{MR1426488} implies that $\Ind(\cB)$ is a localizing Serre subcategory of~$\Ind(\Pro \cB)$.
(Alternatively, use that the $\Ind$-completion of a Serre subcategory is a Serre subcategory, by~\cite[Prop.\ 8.6.6, Prop.\ 8.6.12]{MR2182076},
and that the functor $\Ind(\cB) \to \Ind(\Pro \cB)$ preserves filtered colimits, by~\cite[Proposition~6.1.9]{MR2182076}.)
Furthermore, $\Ind \cB \subset \Ind (\Pro \cB)$ has {\em finite type} in the sense that the right adjoint to the inclusion $\Ind \cB \to \Ind (\Pro \cB)$ preserves filtered colimits.

\begin{lemma}\label{extending quotient functor}
Let~$\cB$ be an Artinian category. %
Let $j^*: \Pro \cB \to \Pro \cB/\cB$ be the quotient functor.
Then $\Ind(j^*): \Ind(\Pro \cB) \to \Ind(\Pro \cB/\cB)$ is the quotient functor with kernel $\Ind(\cB)$, i.e.\ the exact functor
\[i: \Ind (\Pro \cB)/\Ind \cB \to \Ind(\Pro \cB/\cB)\]
induced by~$\Ind(j^*)$ is an equivalence of categories.
\end{lemma}
\begin{proof}
By~\cite[Theorem~2.6]{MR1426488}, the quotient $\Ind (\Pro \cB)/\Ind \cB$ is a locally coherent category, 
so it is enough to prove that~$i$ preserves filtered colimits and induces an equivalence on compact objects.

By~{\em loc.\ cit.}, the restriction of the quotient functor 
$k^*: \Ind (\Pro \cB) \to \Ind (\Pro \cB)/\Ind \cB$ 
to~$\Pro \cB$ induces an equivalence of $\Pro \cB/\cB$ onto the category of compact objects of~$\Ind (\Pro \cB)/\Ind \cB$.
Hence 
\[
i : \fp(\Ind (\Pro \cB)/\Ind \cB) \to \fp(\Ind(\Pro \cB/\cB)) 
\]
is an equivalence, because there is a commutative triangle
\[
\begin{tikzcd}
&\Pro \cB \arrow[dl, "k^*|_{\Pro \cB}"] \arrow[dr, "\Ind(j^*)|_{\Pro \cB}"] &\\
\fp(\Ind (\Pro \cB)/\Ind\cB) \arrow[rr, "i"] & & \fp(\Ind(\Pro \cB/\cB))
\end{tikzcd}
\]
where both diagonal arrows are quotient functors with kernel~$\cB$.
(We have used that $\Ind(j^*) = ik^*$, by definition of~$i$.)

There remains to prove that~$i$ preserves filtered colimits.
Let $k_*$ be the right adjoint of~$k^*$.
By~\cite[Theorem~2.6]{MR1426488} once again, $k_*$ preserves filtered colimits.
Let $\{X_t: t \in T\}$ be a filtered system of objects of~$\Ind (\Pro \cB)/ \Ind \cB$, and let $X \coloneq  \varinjlim_t X_t$.
We need to prove that the natural map $\varinjlim_t i X_t \to iX$ is an isomorphism.
Since the counit $k^*k_* \to 1$ is an isomorphism, this map is isomorphic to
$\varinjlim_t ik^*k_* X_t \to ik^*k_*X$.
Since $k_*$ and $ik^* = \Ind(j^*)$ preserve filtered colimits,
this map is in turn isomorphic to $ik^*k_*(\varinjlim_t X_t \isoto X)$, which is an isomorphism, as desired.
\end{proof}

\bigskip

\emergencystretch=3em
\bibliographystyle{amsalpha}
\bibliography{universalBM}

\end{document}